\documentclass{amsart}
\usepackage{marktext}
\usepackage{gimac}
\newcommand{\GI}[2][]{\sidenote[colback=yellow!20]{\textbf{GI\xspace #1:} #2}}

\newcommand{\HL}[2][]{\sidenote[colback=orange!20]{\textbf{HL\xspace #1:} #2}}
\newcommand{\CG}[2][]{\sidenote[colback=cyan!10]{\textbf{CG\xspace #1:} #2}}

\renewcommand{\setminus}{-}
\newcommand{\wconv}[1][]{\xrightarrow[#1]{\mathcal D}}

\DeclareMathOperator*{\plim}{p-lim}

\makeatletter
\newcommand{\leqnomode}{\tagsleft@true\let\veqno\@@leqno}
\newcommand{\reqnomode}{\tagsleft@false\let\veqno\@@eqno}
\newcommand{\mylabel}[2]{\def\@currentlabel{#2}\label{#1}}
\makeatother

\newcommand{\II}{\mathit{II}}
\newcommand{\III}{\mathit{III}}
\newcommand{\IV}{\mathit{IV}}

\newcommand{\e}{\epsilon}
\newcommand{\eps}{\epsilon}
\newcommand{\Aep}{\mathcal{A}^{\eps}}
\newcommand{\B}{\mathcal B}

\newcommand{\fraka}{\mathfrak{a}}
\newcommand{\frakb}{\mathfrak{b}}

\newcommand{\calH}{{\mathcal{H}}}

\renewcommand{\st}{\,|\,}

\begin{document}
\title{An oscillator driven by algebraically decorrelating noise}
\author{Chistophe Gomez}
\address{%
  Aix Marseille Univ, CNRS, I2M, Marseille, France
  }
\email{christophe.gomez@univ-amu.fr}
\author{Gautam Iyer}
\address{%
  Department of Mathematical Sciences,
  Carnegie Mellon University,
  Pittsburgh, PA 15213, USA}
\email{gautam@math.cmu.edu}
\author{Hai Le}
\address{%
  Department of Mathematics,
  Pennsylvania State University,
  State College PA 16802}
\email{leviethaipsu3@gmail.com}
\author{Alexei Novikov}
\address{%
  Department of Mathematics,
  Pennsylvania State University,
  State College PA 16802}
\email{anovikov@math.psu.edu}

\thanks{%
  This work has been partially supported by
  the National Science Foundation under grants
NSF DMS-1813943, DMS-1515187 to AN and HL,
   NSF DMS-1814147 to GI, 
  and the Center for Nonlinear Analysis.
}
\begin{abstract}
  We consider a stochastically forced nonlinear oscillator driven by a stationary Gaussian noise that has an algebraically decaying covariance function.
  It is well known that such noise processes can be renormalized to converge to \emph{fractional} Brownian motion, a process that has memory.
  In contrast, we show that the renormalized limit of the nonlinear oscillator driven by this noise converges to diffusion driven by standard (not fractional) Brownian motion, and thus retains no memory in the scaling limit.
  The proof is based on the study of a fast-slow system using the perturbed test function method.
\end{abstract}

\maketitle 

\section{Introduction}

Consider a stochastically forced nonlinear oscillator with 1 degree of freedom:
\begin{equation}\label{e:oscillator}
  \ddot x_t  + f(x_t)  =  \epsilon  v(t),\quad x_0\in \R,\quad \dot{x}_0 = y_0\in \R\,.
\end{equation}
Here $f\colon \R \to \R$ is a given smooth function, and $v$ is a stochastic process representing the noise.
Our interest is to study the asymptotic long-time behavior of $x$ when the noise $v$ has an algebraically decaying covariance function.
In this case, the rescaled noise converges to fractional Brownian motion (fBm), a process that has a memory.
We aim to study how the nonlinear dynamics affects the limiting behavior, and study whether or not the rescaled oscillator also converges to a process with memory.

A similar question was studied by Komorowski et\ al.~\cite{KomorowskiNovikovEA14} in the case of a passive tracer particle advected by a periodic shear flow.
In this case, it turns out that there is a certain parameter regime where the time rescaled dynamics is Markovian and the memory effect of the noise is forgotten.
However, there is also a regime (namely the ``very long-time'' behavior when the Hurst index of the driving noise is larger than $1/2$) where the memory effect persists.
In contrast, for the oscillator~\eqref{e:oscillator}, we will show that the memory effect \emph{never} persists, and the effective long-time behavior is always Markovian.
The main reason for this is that the oscillatory nature of the deterministic dynamics cancels slowly decaying correlations, and destroys the memory effect.
To study~\eqref{e:oscillator} we cannot use the limit theorem for additive functionals fBm's used by~\cite{KomorowskiNovikovEA14}.
Instead, we recast~\eqref{e:oscillator} as a fast-slow system and use the perturbed test function method.
Even though this method was introduced in the '76~\cite{PapanicolaouStroockEA76,Kushner84}, it has applications to many problems arising today~\cites{FouqueGarnierEA07,DebusscheRoselloEA21,CaillerieVovelle21}, and is the only method we presently know that can be used to prove our main result.

To describe our setup, we first introduce the Hamiltonian
\begin{equation*}
  H(x,y) \defeq \frac{1}{2}y^2+\int_0^x f(s) \, ds\,,
\end{equation*}
and recast~\eqref{e:oscillator} as a stochastically perturbed Hamiltonian system
\begin{subequations} 
  \begin{align}
  \label{e:system1}
  &\dot x_t  =
  y_t= \partial_y H(x_t ,y_t )\,,
  \\
  \label{e:system2}
  &\dot y_t
  = \ddot x_t= - f(x_t ) + \epsilon v(t)= -\partial_x H(x_t ,y_t )+ \epsilon v(t )\,.
  \end{align}
\end{subequations}
Setting $X_t \defeq (x_t,y_t)^T$, we rewrite~\eqref{e:system1}--\eqref{e:system2} compactly as
\begin{equation}\label{e:noisyHamiltonianUnscaled}
\dot X_t = \nabla^{\bot}H(X_t ) + \epsilon v(t) e_2,\quad X_0 = (x_0,y_0)\in \R^2\,.
\end{equation} 
where $\nabla ^{\bot}\defeq (\partial_y,-\partial_x)^T$, and $e_2 = (0, 1)^T$.

To study the long-time behavior, 
we consider the time rescaled process $X^\epsilon(t) \defeq X(t/\epsilon^2)$, and observe
\begin{equation}\label{e:hamiltonian}
\dot X_t^{\epsilon} = \frac{1}{\epsilon^2}\nabla^{\bot}H(X_t^{\epsilon} ) + \frac{1}{\epsilon} v\left(\frac{t}{\epsilon^2}\right)e_2 \,,
\quad X_0^{\epsilon} = X_0 = (x_0,y_0)\in \R^2.
\end{equation}
When the noise $v$ is white, the behavior of $X^\epsilon$ as $\epsilon \to 0$ is completely described by the averaging principle of Freidlin and Wentzell~\cites{FreidlinWentzell93,FreidlinWentzell94,FreidlinWeber98}.
To briefly explain this, we note that in the absence of noise (i.e.\ when $v\equiv 0$), the process $X^{\epsilon}$ travels very fast along level sets of the Hamiltonian.
In the presence of noise, the process $X^\epsilon$ will also diffuse slowly across these level sets.
To capture the limiting behavior we factor out the fast motion by projecting to the  \emph{Reeb graph} of the Hamiltonian, which has the effect of identifying all closed trajectories of the Hamiltonian system to points.
Now, as $\epsilon \to 0$, this projection converges to a diffusion driven by a standard Brownian motion and we refer the reader to~\cites{FreidlinWentzell93,FreidlinWentzell94,FreidlinWeber98} for details.
(We also remark that when the noise~$v$ is white, one may also study the behavior of~$X$ on time scales shorter than $1/\epsilon^2$.
In some scenarios, a robust and stable limiting behavior is also observed on these time scales and has been studied by several authors~\cites{Young88,YoungJones91,Bakhtin11,HaynesVanneste14,HaynesVanneste14a,HairerKoralovEA16,TzellaVanneste16,HairerIyerEA18}.)

In this paper, our interest is to study the case when the driving noise in~\eqref{e:hamiltonian} is colored in time and has algebraically decaying correlations.
It is well known that this noise can be renormalized to converge to a fractional Brownian motion (fBm), with Hurst parameter that is determined from the decay rate of the correlation function (see for instance ~\cite{Taqqu74,Taqqu77,Marty05}).
When the Hurst parameter is not half (which happens when the correlation function decays like $1/t^\gamma$ with $\gamma \neq 1$), the renormalized limit of the noise has memory and is a non-Markovian process.
We aim to study whether the memory effect is also present in the~$\epsilon \to 0$ limit of the nonlinear oscillator~\eqref{e:hamiltonian}, where the noise is combined with the Hamiltonian dynamics.

At present, we are only able to analyze the scenario where the Hamiltonian is smooth with exactly one non-degenerate critical point.
In this case, it is convenient to think about the dynamics of~$X^\epsilon$ in terms of action-angle coordinates.
The angular coordinate of $X^\epsilon$ changes very fast, and has no meaningful limit as $\epsilon \to 0$.
The action coordinate of~$X^\epsilon$, on the other hand, changes slowly as a result of the interaction between the noise and the averaged angular coordinate and has a non-trivial limit $\epsilon \to 0$.
To study this we use the Hamiltonian itself as a proxy for the action coordinate and state our results in terms of convergence of the process $H(X^\epsilon)$.
(To relate this to the Freidlin--Wentzell framework, we note that when $H$ has exactly one non-degenerate critical point, the Reeb graph has exactly one vertex and one edge, and~$H(X^\epsilon)$ is precisely the projection of $X^\epsilon$ onto the Reeb graph.)

In this paper, we prove two results. They show that even though the driving noise has memory, the time correlations are destroyed by the oscillatory dynamics of~\eqref{e:hamiltonian}, and $H(X^\epsilon)$ converges to a diffusion driven by standard Brownian motion (a process without memory).
Roughly speaking our main results are as follows:
\begin{enumerate}
  \item
    If the Hamiltonian~$H$ is quadratic, and the noise~$v$ is any stationary Gaussian process with an algebraically decaying covariance function, then $H(X^\epsilon)$ converges to a diffusion process driven by a Brownian motion. In this particular case, the limiting diffusion is a rescaling of the square of the 2-dimensional Bessel process.

  \item
    If the Hamiltonian is not quadratic (but still smooth with exactly one non-degenerate critical point), then we can show $H(X^\epsilon)$ converges to a diffusion driven by Brownian motions, provided the noise~$v$ is chosen suitably.
    (The noise~$v$ is still a stationary Gaussian process with an algebraically decaying covariance function.
    However, it must be expressible as a superposition of Ornstein--Uhlenbeck processes as we make use of certain exponential mixing estimates in the proof.)
\end{enumerate}


\subsection{Main results for a quadratic Hamiltonian.}

We now state our results precisely when $H$ is quadratic.
Hereafter we assume that the noise $v$ is a stationary Gaussian process with covariance function
\begin{equation}\label{e:rDef}
  R(t) \defeq \E v(t)v(0) = \E v(t + s) v(s) \,. 
\end{equation}
Our main result when the Hamiltonian is quadratic is as follows:
\begin{proposition}\label{p:qConv}
  Let~$H$ be the quadratic Hamiltonian
  \begin{equation}\label{e:Hquadratic}
    H(x) \defeq \frac{\abs{x}^2}{2} \,, \qquad x = (x_1, x_2) \in \R^2\,,
  \end{equation}
  and suppose the noise~$v$ is a stationary Gaussian process whose covariance function, $R$, is of the form
  \begin{equation}\label{e:RregularlyVarying}
    R(t) = \frac{L(t)}{t^\gamma}
  \end{equation}
  for some $\gamma \in (0, 2)$ and a function~$L$ that is slowly varying at infinity.   \footnote{%
  Recall a function~$L$ is said to be slowly varying at infinity if for every $s > 0$ we have $L(st) / t \to 1$ as $t \to \infty$.} If~$\gamma \in (0, 1]$, we further assume that $L$ has the slow increase property
  \begin{equation}\label{e:slowIncrease}
    \lim_{t \to \infty} \frac{L'(t)}{L(t) / t} = 0\,.
  \end{equation}
  Then the family of processes $H(X^\epsilon)_{\eps>0}$ converges in distribution to $H(W^D)$, where $W^D$ is a 2D Brownian motion with $W_0 = X_0$ and covariance matrix $D$ given by
  \begin{equation} \label{e:qcov}
  D_{11} = D_{22} \defeq  \lim_{x \to \infty} \int_0^x R(z) \cos(z) \, dz \qquad\text{and}\qquad D_{12}=D_{21}=0 \,.
  \end{equation}
\end{proposition}

\begin{remark}\label{r:Dexists}
  For~$\gamma \in (0, 1]$ we note that~\eqref{e:slowIncrease} implies that~$R' < 0$ near infinity.
  Since~$\cos(z)$ oscillates periodically, the limit in~\eqref{e:qcov} exists and is finite.
  For~$\gamma \in (1, 2)$, we note $R \in L^1(\R)$ and so the limit in~\eqref{e:qcov} also exists and is finite.
\end{remark}

As mentioned earlier, it is well known~\cite{Taqqu74,Taqqu77,Marty05} that the renormalized noise converges to a fractional Brownian motion.
For the reader's convenience, we state a result to that effect next.

\begin{proposition}\label{p:fbmconverge}
  Let~$v$ be a stationary Gaussian process with covariance function~$R$ given by~\eqref{e:RregularlyVarying} for some~$\gamma \in (0, 2)$ and a slowly varying function~$L$.
  If~$\gamma \in (1, 2)$, we additionally suppose
  \begin{equation}\label{e:intR0}
    \int_0^\infty R(t) \, dt = 0\,.
  \end{equation}
  Let
  \begin{equation}\label{e:udef}
    \sigma(\epsilon) = \begin{dcases}
      L(\epsilon^{-2})^{1/2} \epsilon^\gamma & \gamma \neq 1\,,\\
      L(\epsilon^{-2})^{1/2} \epsilon \abs{\ln \epsilon}^{1/2} & \gamma = 1\,,
    \end{dcases}
    \quad\text{and}\quad
    u^\eps(t) \defeq \frac{1}{\sigma(\epsilon)}\int_0^{t}v\left(\frac{s}{\epsilon^2}\right) \, ds \,.
  \end{equation}
  Then, as~$\epsilon \to 0$, the family of processes $(u^\eps)_{\eps>0}$ converges in distribution
  to $\sigma_\calH B^{\calH}$, where $B^\calH$ is a standard fractional Brownian motion with Hurst index~$\calH =1-\gamma/2$, and
  \begin{equation}\label{e:sigmaH}
    \sigma_{\mathcal H}^2 \defeq
      \begin{dcases}
  \frac{1}{\mathcal H \abs{2\mathcal H - 1}} & \mathcal H \neq \frac{1}{2}\,,
  \\
  1 & \mathcal H = \frac{1}{2}\,.
      \end{dcases}
  \end{equation}
\end{proposition}

To reiterate our main point, we note that Proposition~\ref{p:qConv} implies
\begin{equation}\label{e:p1convExpl}
  H\paren[\Big]{X_0 + 
    \frac{1}{\epsilon^2} \int_0^t \grad^\perp H(X^\epsilon_s) \, ds
    + \frac{1}{\epsilon} \int_0^t v\paren[\Big]{\frac{s}{\epsilon^2}} e_2 \, ds }
  \xrightarrow{\epsilon \to 0} H(W^D_t) \,,
\end{equation}
for a Brownian motion $W^D$ with covariance matrix~$D$.
However, 
\begin{equation*}
  \frac{1}{\sigma(\epsilon)} \int_0^t v \paren[\Big]{ \frac{s}{\epsilon^2} } \, ds
    \xrightarrow{\epsilon \to 0} \sigma_\mathcal H B^H_t\,,
\end{equation*}
for a fBm $B^\mathcal H$ with Hurst index~$\mathcal H$.
Of course, $H(W^D)$ is a Markov process with no memory, but $B^\mathcal H$ is a non-Markovian process with memory.

Note further that when $\gamma \in (0, 1)$, we see $\epsilon\ll \sigma(\epsilon) $ and so the term $\frac{1}{\epsilon} \int_0^t v(s/\epsilon^2) \, ds$ appearing in~\eqref{e:p1convExpl} diverges as~$\epsilon \to 0$.
On the other hand when $\gamma \in (1, 2)$ (and equation~\eqref{e:intR0} holds), we see $\sigma(\epsilon) \ll \epsilon$, and so this term vanishes.
In both cases, the oscillatory Hamiltonian term contributes non-trivially and as a result, the time correlations are forgotten and $H(X^\epsilon)$ converges to a memoryless Markov process.
\medskip

We now explain the reason for the assumption~\eqref{e:intR0}, which is required in Proposition~\ref{p:fbmconverge} when $\gamma \in (1, 2)$, but not in Proposition~\ref{p:qConv}.
If~$\gamma \in (1, 2)$ by the central limit theorem one can show that $\frac{1}{\epsilon} \int_0^t v(s / \epsilon^2) \, ds$ converges to a Brownian motion.
The covariance of this Brownian motion is proportional to~$\int_0^\infty R(t) \, dt$.
If this is~$0$ (as required by assumption~\eqref{e:intR0}), then one can divide $\int_0^t v(s/\epsilon^2) \, ds$ by the smaller factor $\epsilon^\gamma = \sigma(\epsilon)$, and look for a non-trivial limit.
This is precisely what is given by Proposition~\ref{p:fbmconverge}.
The reason the assumption~\eqref{e:intR0} is not needed for Proposition~\ref{p:qConv} is because in~\eqref{e:p1convExpl} the noise is only scaled by a factor of~$1/\epsilon$ and not $1/\sigma(\epsilon)$.
\medskip

Finally, we briefly summarize the main idea of the proofs of Propositions~\ref{p:qConv} and~\ref{p:fbmconverge}.
First, Proposition~\ref{p:fbmconverge} only requires convergence of Gaussian processes and can be proved by directly computing covariances.
Similar results are well known~\cites{Taqqu74,Taqqu77,Marty05}.
Since the exact result we need isn't readily available, we prove it in Appendix~\ref{s:fbmConvProof}.

The proof of Proposition~\ref{p:qConv} requires a little more work.
The main simplification obtained from the assumption that $H$ is quadratic is that the evolution of the angle coordinate is independent of the action coordinate.
As a result one can perform a (time dependent) rotation in space and prove convergence of $H(X^\epsilon)$ by studying integrals of the form
\begin{equation*}
  \frac{1}{\epsilon} \int_0^t v\paren[\Big]{\frac{s}{\epsilon^2}} \cos\paren[\Big]{ \frac{s}{\epsilon^2} } \, ds
  \qquad\text{and}\qquad
  \frac{1}{\epsilon} \int_0^t v\paren[\Big]{\frac{s}{\epsilon^2}} \sin\paren[\Big]{ \frac{s}{\epsilon^2} } \, ds\,.
\end{equation*}
These are now Gaussian and one can study convergence by computing covariances directly.
We note that the study of similar integrals arises in the study of random media where several authors~\cites{Marty05,FouqueGarnierEA07,Garnier95} have used oscillatory dynamics to decorrelate the medium.

\subsection{Main results for Hamiltonians with only one critical point.}

We now study the case where the Hamiltonian has exactly one non-degenerate critical point but is not necessarily quadratic.
In this case we will need to work with a stationary Gaussian noise process~$v$ that can be obtained by super-imposing Ornstein--Uhlenbeck processes and written in the form
\begin{equation*}
  v(t) \defeq \int_S e^{-\mu \abs{p}^{2\beta} (t - u)} \mathcal B( du, dp )\,.
\end{equation*}
Here $S \subseteq \R$ is a bounded symmetric open interval, $\mu$, $\beta$ are constants, and $\mathcal B$ is a Gaussian random measure that is white in time and colored in space and will be constructed explicitly in Section~\ref{s:noise} below to ensure that the covariance of~$v$ decays algebraically.
Our main theorem can now be stated as follows.

\begin{theorem}\label{t:mainIntro}
  Let~$H\colon \R^2 \to \R$ be a smooth Hamiltonian with exactly one non-degenerate minimum at $(0,0)$, and~$v$ be the stationary Gaussian noise described above.
  Then
  \begin{equation}\label{e:HXepConv}
  H(X^{\epsilon}) \wconv[\epsilon \to 0]  \mathcal X\,,
  \end{equation}
  where $\mathcal X$ is a diffusion driven by standard Brownian motion, with infinitesimal generator
  \begin{equation}\label{eq:generator_X}
  	L = \frac{1}{2\Lambda(\mathcal X)} \partial_{\mathcal X}(\Sigma(\mathcal X)\partial_{\mathcal X})  \,.
  \end{equation}
  Here the coefficients~$\Lambda$ and~$\Sigma$ are defined by
  \[
  \Lambda(\mathcal X)=\oint_{\{H=\mathcal X\}} \, \frac{dl}{|\nabla H|} \,,
  \]
  and
  \[
  \Sigma(\mathcal X) = 2\int_{u=0}^\infty du \, R(u)\oint_{\{H=\mathcal X\}} \partial_y H(\check X^{x,y}_u) \partial_y H(x,y) \frac{ dl(x,y)}{|\nabla H(x,y)|} \,,
  \]
  where $\check X^{x,y}$ is the solution to
  \[
  \partial_t \check X^{x,y}_t = \nabla^\perp H( \check X^{x,y}_t),\qquad\text{with}\qquad \check X^{x,y}_{t=0} = (x,y)\,.
  \]
\end{theorem}

Theorem~\ref{t:mainIntro} follows immediately from  the more general Theorem~\ref{mainresult}, below, concerning fast-slow systems, and the proof of Theorem~\ref{t:mainIntro} is presented shortly after the statement of Theorem~\ref{mainresult}.
Moreover, Theorem~\ref{mainresult} also explicitly characterizes the limiting system and can be used to explicitly characterize the diffusion~$\mathcal X$ in Theorem~\ref{t:mainIntro}.

Note that when~$H$ is not quadratic the convergence of $H(X^\epsilon)$ can't be reduced to Gaussian integrals as in Proposition~\ref{p:qConv}, and so the proof of Theorem~\ref{t:mainIntro} is more involved.
We prove Theorem~\ref{t:mainIntro} by first switching to action-angle coordinates and converting~\eqref{e:hamiltonian} to a fast-slow system, where the slow variables are mean-zero in the fast variable.
We then use the perturbed test function (PTF) method~\cite{Kushner84} to prove convergence of this system.
This approach provides more information than just the convergence of $H(X^\epsilon)$, as stated in Theorem~\ref{t:mainIntro}, and we refer the reader to Theorem~\ref{mainresult}, below, for the precise statement.
Of course, the convergence of action coordinate is equivalent to convergence of the processes~$H(X^\epsilon)$.
However, our proof also identifies an asymptotic diffusive behavior of the angle coordinate on the slow scale.

We remark that our fast-slow system is different from that recently considered by Hairer and Li~\cite{HairerLi20}.
Indeed, in~\cite{HairerLi20}, the authors proved an averaging principle for a coupled fast-slow system driven by fBm, with Hurst index $\mathcal H > \frac{1}{2}$, and they obtained convergence in probability to the naively  averaged system.
However, they require the fast variable to be driven by an independent \emph{standard} Brownian motion.
In the scenario that arises in the present paper, the fast variable is driven by noise that converges to a \emph{fractional} Brownian motion when renormalized, and the same noise is also used to drive the slow variable.
Moreover the limiting system we obtain is not the naively averaged one, but a diffusion with an averaged generator akin to the limiting behavior of fast-slow systems driven by standard Brownian noise (see for instance~\cite{PavliotisStuart08}).
\medskip

 Theorem~\ref{t:mainIntro} is stated for Hamiltonians which have only one (non-degenerate) critical point.
The Reeb graph of such Hamiltonians is simply the half line, and the limiting process in Theorem~\ref{t:mainIntro} is a diffusion on the half line.
For general Hamiltonians, Reeb graph may have multiple vertices, and the limiting process should be a diffusion on this Reeb graph.
When the driving noise is a standard Brownian motion (as in the Freidlin--Wentzell theory), this diffusion is characterized through its generator on each edge of the Reeb graph, and certain gluing conditions on the vertices.
While the proofs in the Freidlin--Wentzell theory rely heavily on the Markov property, and do not apply to our situation, the form of the infinitesimal generator~\eqref{eq:generator_X}
is similar to that in~\cite[Equation 1.8]{FreidlinWeber98}.
As a result, we expect similar gluing conditions to hold in our context. 

Explicitly, for a sufficiently general class of Hamiltonians, we expect that the process $(H(X^\eps))_{\eps>0}$ converges in distribution to a diffusion process on the Reeb graph.
The generator of this diffusion should be
\[
L_j h_j(\mathcal X) = \frac{1}{2\Lambda_j(\mathcal X)} \partial_{\mathcal X}(\Sigma_j(\mathcal X)\partial_{\mathcal X} h_j) 
\]
on each edge $I_j$.
The coefficients~$\Lambda_j$ and $\Sigma_j$ are obtained by taking line integrals as follows.
Let~$\mathcal C_j(\mathcal X)$ the connected component of the level set~$\set{H = \mathcal X}$ corresponding to a point $\mathcal X \in I_j$.
Now~$\Sigma_j$ and~$\Lambda_j$ are given by
\begin{gather*}
	\Sigma_j(\mathcal X) = 2\int_{u=0}^\infty du \, R(u)\oint_{C_j(\mathcal X)} \partial_y H(\check X^{x,y}_u) \partial_y H(x,y) \, \frac{dl(x,y)}{|\nabla H(x,y)|} \,,
	\\
	\Lambda_j(\mathcal X) = \oint_{C_j(\mathcal X)} \frac{dl(x, y)}{\abs{\grad H(x, y)}} \,.
\end{gather*}

To describe the gluing condition, consider an interior vertex $O_k$ corresponding to the saddle point $(\mathbf{x}_k,\mathbf{y}_k)$ and the level set $\set{H = \mathcal X_k}$, at which the collection of edges $\set{ I_j}$ meet.
(The typical situation is a ``figure eight'' scenario, which has three edges meeting at the vertex $O_k$.)
Now define
\begin{equation*}
	\beta_{k j} = \lim_{\mathcal X \to \mathcal X_k} C_{j}(\mathcal X)\,,
\end{equation*}
where the limit is taken along the edge $I_{j}$.
The gluing condition at the vertex~$O_k$ can then be written as
\begin{equation}\label{eq:gluing}
	\sum_{j:\, I_j\sim O_k} \pm \beta_{kj}\,h'_j(H(\mathbf{x}_k,\mathbf{y}_k)) = 0 \,,
\end{equation}
where sign before $\beta_{kj}$ is positive if $\mathcal X > \mathcal X_k$ along the edge $I_j$, 
and negative if $\mathcal X < \mathcal X_k$ along the edge $I_j$.

The main difference between our main result, Theorem~\ref{t:mainIntro}, and the standard Freidlin--Wentzell theory~\cite{FreidlinWeber98} is in the coefficient $\Sigma$.
The definition of~$\Sigma$ in Theorem~\ref{t:mainIntro} involves an average of $\partial_y H(\check X^{x,y}_u) \partial_y H(x,y)$, with a shift of one factor to the point $\check X^{x,y}_{t=u}$.
In~\cite{FreidlinWeber98}  the coefficient $\Sigma$ depends only on $ (\partial_y H(x,y))^2 $.
The reason for this modification lies in the fact that 
in our case the oscillations of the Hamiltonian system and the random 
fluctuations in~\eqref{e:hamiltonian} are strongly coupled, and are both of order $O(1/\eps^2)$. 
The oscillatory behavior of the Hamiltonian system and random fluctuations fully interact 
with each other and produce an effective diffusion coefficient involving
the correlation of $\partial_y H(\check X^{x,y}_u)$ with $\partial_y H(x,y)$ for pairs of points on the same orbit.
If the oscillatory behavior of the Hamiltonian system is slower than the random fluctuations, then their coupling is weak and we would obtain the same generators as in~\cite{FreidlinWeber98}.  
This will arise if, for example, we consider white noise in~\eqref{e:hamiltonian} as the driving random force.  
This dichotomy of limiting behaviors between strongly and weakly coupled  setups is  
well-known in other approximation-diffusion limits of random differential equations with periodic 
components, (for instance, compare and contrast Theorems~6.4 and~6.5 in~\cite{FouqueGarnierEA07}). In the present work, we need the oscillations of the Hamiltonian system and the random fluctuations to be of the same order; otherwise, the shift $\check X^{x,y}_u$, and hence $\Sigma$, is not well-defined.
We hope to present a proof of this phenomenon in future work.

\subsection{Plan of this paper.}
In Section~\ref{sec:quad} we prove Proposition~\ref{p:qConv}, modulo two computational lemmas.
The proofs of these lemmas are relegated to Appendix~\ref{append:quad}.
The proof of Proposition~\ref{p:fbmconverge} is presented in Appendix~\ref{s:fbmConvProof}.
In Section~\ref{s:noise} we construct the noise process~$v$ that will be used in Theorem~\ref{t:mainIntro}.
We also establish basic properties of the noise in this section and put two computationally involved proofs in Appendices~\ref{proofpropmar} and~\ref{proofpropbound}. 
In Section~\ref{s:main} we state a precise, more general, version of Theorem~\ref{t:mainIntro} (Theorems~\ref{mainresult} and~\ref{t:mainGen}), and prove Theorem~\ref{t:mainIntro}.
We prove regularity of the coefficients of the limiting equation in Theorem~\ref{mainresult} (equation~\eqref{eq:I}) in Appendix~\ref{proofpropbound_a}, and study the limiting equation itself in Section~\ref{sec:proofI}.
Finally, we prove Theorem~\ref{mainresult} in Sections~\ref{sec:property_limit_eq}, \ref{sec:propmain} and~\ref{sec:propSDE}.

\subsection*{Acknowledgments}
The authors thank Lenya Ryzhik for suggesting this problem to us, and for many helpful discussions.

\section{Proof of convergence for a quadratic Hamiltonian.}\label{sec:quad}

This section is devoted to proving Proposition~\ref{p:qConv}.
The main idea behind the proof is that when~$H$ is quadratic, the deterministic dynamics rotates with constant angular speed in all trajectories.
Performing a spatial rotation will now reduce the problem to studying convergence of Gaussian processes, which can be resolved by computing covariances.
\begin{proof}[Proof of Proposition~\ref{p:qConv}] 
  When $H$ is given by~\eqref{e:Hquadratic}, we use Duhamel's formula to write the solution of~\eqref{e:hamiltonian} as
  \begin{equation*}
  X^{\epsilon}_t =
  \begin{pmatrix}
  \phantom-\cos (\frac{t }{ \epsilon^2}) & \sin ( \frac{t}{\epsilon^2})\\
  -\sin (\frac{t}{\epsilon^2}) & \cos (\frac{t}{\epsilon^2})
  \end{pmatrix}
  X_0
  + \frac{1}{\epsilon} \int_0^t
  v\paren[\Big]{\frac{\tau}{\epsilon^2}}
  \begin{pmatrix}
  \sin( \frac{t - \tau)}{\epsilon^2} )\\
  \cos( \frac{t - \tau}{\epsilon^2} )
  \end{pmatrix}
  \, d\tau.
  \end{equation*}
  Let $M(t)$ be the rotation matrix
  \begin{equation*}
  M(t) \defeq \begin{pmatrix}
  \cos (t ) & \sin ( t)\\
  -\sin (t) & \cos (t)
  \end{pmatrix}\,,
  \end{equation*}
  and define the rotated process~$Y^{\eps}$ by $Y^\epsilon_t=M(-t/\epsilon^2)X^{\epsilon}_t$.
  Clearly $H(X^\epsilon) = H(Y^\epsilon)$, and so convergence of the processes~$H(X^\epsilon)$ reduces to convergence of the processes~$H(Y^\epsilon)$.

  We claim that the processes~$Y^\epsilon$ itself converge to a Brownian motion as~$\epsilon \to 0$.
  To see this, observe
  \begin{equation*}
  Y^{\epsilon}_t
  = X_0
  + \frac{1}{\epsilon}\int_{0}^{t}
  v\paren[\Big]{\frac{\tau}{\epsilon^2} }
  \begin{pmatrix}
  - \sin \paren{ \frac{\tau}{\epsilon^2} }\\
  \phantom- \cos \paren{ \frac{\tau}{\epsilon^2} }
  \end{pmatrix}
  \, d\tau\,.
  \end{equation*}
  We will now show that the second term above converges to a Brownian motion.
  For this, define
  \begin{gather}
    \label{e:w1}
    w_1^{\epsilon}(t) \defeq \frac{1}{\epsilon}\int_0^tv\left(\frac{\tau}{\epsilon^2}\right)\sin\left(\frac{\tau}{\epsilon^2}\right)d\tau\,,
    \\
    \label{e:w2}
    w_2^{\epsilon}(t) \defeq \frac{1}{\epsilon}\int_0^tv\left(\frac{\tau}{\epsilon^2}\right)\cos\left(\frac{\tau}{\epsilon^2}\right)d\tau \,.
  \end{gather}
Note, the noise~$v$ (when normalized similarly by~$\sigma(\epsilon)$ without the periodic terms) converges to fBm.
However, the expressions above use a normalization factor of~$\epsilon$ instead of $\sigma(\epsilon)$, and have an oscillatory factor.
  For this reason we claim $(w_1^\epsilon, w_2^\epsilon)$ converges to a Brownian motion.
  To prove convergence of $w_1^\epsilon$, $w_2^\epsilon$, we first state two lemmas:
  \begin{lemma}\label{l:quad1}
    For any $T > 0$, there exists a constant $C>0$ such that for every $s,t\in [0,T]$,  $i \in \set{1, 2}$, we have
    \begin{equation}\label{e:EwiTight}
    \E (w^\epsilon_i (t)-w^\epsilon_i(s))^2
      \leq \begin{dcases}
  C|t-s|^{1-\gamma}  & \gamma \in (0, 1) \,,\\
  C\abs{t - s}^{2-\gamma}  & \gamma \in [1, 2)\,.
      \end{dcases}
    \end{equation} 
  \end{lemma}
  \begin{lemma}\label{l:quad2}
    For every $s,t\geq 0$, and $i \in \set{1, 2}$ we have 
    \begin{gather}
    \label{e:Ewi2}
    \lim_{\epsilon\to 0}\E (w^\epsilon_i (t)- w^\epsilon_i (s))^2 = D_{11} |t-s|\,,
    \\
    \label{e:Ewij}
    \lim_{\epsilon\to 0}\E (w^\epsilon_1 (t)- w^\epsilon_1 (s))(w^\epsilon_2(t) - w^\epsilon_2(s)) = 0\,.
    \end{gather}
  \end{lemma}
  
  The proofs of both Lemmas~\ref{l:quad1} and~\ref{l:quad2} are lengthy, but direct computations.
  Thus, for clarity of presentation, we postpone the proofs to Appendix~\ref{append:quad}.
  Once Lemmas~\ref{l:quad1} and~\ref{l:quad2} are established, the proof of Proposition~\ref{p:qConv} follows quickly.
  Indeed, Lemma~\ref{l:quad1} implies that the family $(Y^\epsilon)_{\eps>0}$ is tight on $\mathcal{C}([0, T],\mathbb{R}^2)$

  To see this, it suffices to show that the processes~$w^\epsilon_i$ are tight.
  Without loss of generality suppose~$T = 1$, and suppose~$\gamma \in (0, 1)$ (the case when $\gamma \in [1, 2)$ is similar).
  Choose an integer $M$ such that $M_{\gamma}=M(1-\gamma)>1$.
  Since $w_i^{\eps}$ is Gaussian process, we know
  \[
    \E(w^\epsilon_i (t)-w^\epsilon_i(s))^{2M} = C_M[\E(w^\epsilon_i (t)-w^\epsilon_i(s))^{2}]^M
  \]
  for some constant $C_M$.
  We will allow $C_M$ to change from line to line as long as it only depends on $M$, and remains independent of~$\epsilon$.
  The above implies
  \begin{equation}\label{e:moment_est}
    \E(w^\epsilon_i (t)-w^\epsilon_i(s))^{2M}
      \leq C_M |t-s|^{M_{\gamma} }\,.
  \end{equation}
  Since $M_\gamma > 1$ by choice, Kolmogorov's criterion (see for instance~\cite{EthierKurtz86} Proposition 10.3) implies that the family~$w_i^\epsilon$ is tight.
  This implies that the processes $Y^\epsilon$ converge in distribution along a subsequence to a continuous process~$Y$. 


  Since each process $Y^\epsilon$ is a Gaussian process, the limiting process $Y$ must also be a Gaussian process.
  Now Lemma~\ref{l:quad2} implies $Y$ is a Brownian motion in $\mathbb R^2$ with $Y_0 = X_0$ and covariance matrix~$D$.
  Since the law of the limiting process is uniquely determined, the family $Y^\epsilon$ itself must converge in distribution, without having to select a subsequence.
  Finally, since $H(X^\epsilon_t) = H(Y^\epsilon_t)$, we obtain convergence of $(H(X^\epsilon))_{\eps>0}$ as claimed.
\end{proof}

\section{Construction of the Noise.}\label{s:noise}

In this section, we will construct noise~$v$ that will be used in Theorem~\ref{t:mainIntro}, and establish a few properties that will be used in the course of the proof of Theorem~\ref{t:mainIntro}.
We require the covariance function~$R$ to be of the form
\begin{equation}\label{e:R}
R(t)
  \defeq \int_S r(p)e^{-\mu|p|^{2\beta}|t|} \, dp\,,
\end{equation}
where $S = (-r_s, r_s)$ is a symmetric, bounded, open interval, $\mu, \beta > 0$ are constants, and $r\colon S\setminus \{0\} \to [0, \infty)$ is defined by
\begin{equation*}
r(p)\defeq \frac{ \lambda(p)}{\vert p\vert^{2\alpha}}\,.
\end{equation*}
Here $\lambda \colon S \to \R$ is a smooth bounded even function such that $\lambda(0) \neq 0$ and
\begin{equation*}
  \int_S r(p)\, dp >0\,,
\end{equation*}
and
\begin{equation}\label{cond_alpha_beta}
  \alpha <\frac{1}{2}\,.
\end{equation} 

We will now construct a stationary Gaussian process~$v$ with covariance function~$R$.
The chosen form of~$R$ allows us to construct~$v$ by superimposing Ornstein--Uhlenbeck processes as follows.
Let~$\xi$ be 2D white noise and define a Gaussian random measure $\B$ by
\begin{equation*}
\mathcal{B}(du,dp) =\sqrt{2\mu r(p)} \, |p|^{\beta} \mathbf{1}_{S}(p)\, \xi(du,dp) \,,
\end{equation*}
Clearly, the covariance of~$\B$ is given by
\begin{equation*}
\E \mathcal{B}(du,dp) \, \mathcal{B}(du', dp')
= 2\mu\, r(p) \, |p|^{2\beta} \delta(u-u')\delta(p-p') \, du \, du'\, dp\, dp'\,.
\end{equation*}
Now define the measure-valued Gaussian random process $V$ by
\begin{equation*}
V(t,dp) = \int_{-\infty}^t e^{-\mu|p|^{2\beta}(t-u)}\,\B(du,dp)\,,
\end{equation*}
and finally define the noise $v$ by
\begin{equation*}
v(t) = \int_S V(t,dp)\,.
\end{equation*}

We now establish a few properties of the process~$v$ that will be used in the proof of Theorem~\ref{t:mainIntro}.
\begin{lemma}
  The process~$v$ defined above is a stationary Gaussian process with covariance function~$R$.
\end{lemma}
\begin{proof}
  Clearly~$v$ is a Gaussian process.
  To see that~$v$ is stationary with covariance function~$R$ we choose any $s \leq t$ and compute
  \begin{align*}
  \E v(s) v(t)
  &=
  \begin{multlined}[t]
  \int_{-\infty}^s \int_S
  \exp\paren[\Big]{ -\mu \paren[\big]{
      \abs{p}^{2\beta} (s - u) + \abs{p}^{2 \beta} (t - u)
  }}
  \cdot 2\mu\abs{p}^{2\beta}
  r(p) \, du \, dp
  \end{multlined}
  \\
  &= \int_S  r(p) e^{-\mu|p|^{2\beta}(t-s)} \,dp
    = R(t - s)\,.
    \qedhere
  \end{align*}
\end{proof}

\begin{lemma}
  If~$R$ is given by~\eqref{e:R}, then
  \begin{equation*}
    \lim_{t \to \infty} t^\gamma R(t) = c_0
  \end{equation*}
  where $c_0$ and~$\gamma$ are defined by
  \begin{equation*}
    \gamma \defeq \frac{1 - 2\alpha}{2 \beta}
    \qquad\text{and}\qquad
    c_0 \defeq \lambda(0) \int_{p \in \R} \frac{e^{-\mu \abs{p}^{2\beta}}}{\abs{p}^{2\alpha}}  \, dp\,.
  \end{equation*}
\end{lemma}
\begin{proof}
  By a change of variable, note that
  \begin{align*}
  t^{\gamma}R(t)
  = t^{\gamma}\int_S
  \frac{\lambda (p)}{|p|^{2\alpha }}
  e^{ - \mu |p|^{2\beta } t}
  \, dp
  &= 
  \int_{(t^{1/(2\beta)} S)}
  \frac{1}{|p'|^{2\alpha }}
  \lambda \paren[\Big]{ \frac{p'}{t^{1 / (2\beta) }} }
  e^{ - \mu |p'|^{2\beta }}
  \, dp'\,,
  \end{align*}
  which converges to~$c_0$ as $t \to \infty$.
\end{proof}

\begin{remark}
  Note that~\eqref{cond_alpha_beta} already implies~$\gamma > 0$.
  If additionally we assume
  \begin{equation}\label{cond_alpha_beta2}
    2 \alpha + 4 \beta > 1
  \end{equation}
  then we will also have $\gamma < 2$.
  In this case the process~$v$ can be renormalized as in~\eqref{e:udef} to converge to a fBm.
  For Theorem~\ref{t:mainIntro}, however, the assumption~\eqref{cond_alpha_beta2} is unnecessary.
\end{remark}

Let $\set{\mathcal G_t}$ be the augmented filtration generated by~$V$.
That is,
\begin{equation}\label{def_flitr_v}
\mathcal{G}_{t} \defeq \sigma\paren[\big]{ \mathcal N \cup \set{ \sigma(V(s,\cdot)) \st 0\leq s\leq t} }\,,
\end{equation}
where $\mathcal N$ is the set of all null sets known at time infinity.
Our next lemma computes conditional expectations of $V(\cdot,dp)$ with respect to the filtration~$\mathcal G$.

\begin{lemma}\label{l:mar} 
  We have for any $t, h\geq 0$
  \begin{equation}\label{markovesp}
  \E\big[ V(t+h,dp) \vert \mathcal{G}_{t}\big]=e^{-\mu\vert p \vert^{2\beta}h}\,V(t,dp)\end{equation}
  and 
  \begin{equation}\label{markovvar}\begin{split}
  \E\Big[V(t+h,dp)V(t+h,dq) \Big\vert \mathcal{G}_{t}\Big]&-  \E\big[V(t+h,dp) \vert \mathcal{G}_{t}\big]\E\big[V(t+h,dq)\vert \mathcal{G}_{t}\big] \\
  &=(1-e^{-2\mu\vert p\vert^{2\beta}h})r(p)\delta(p-q)\,dp\,dq\,.
  \end{split}\end{equation}
\end{lemma}

Our last result concerns the boundedness of our noise process.   
\begin{lemma}\label{l:bound} Let $T>0$, $M>0$,
  \begin{equation*} D_{k,M} \defeq [0,T]\times L^\infty([0,T],W_{k,M}) 
  \end{equation*}
  with
  \begin{equation*}
  W_{k,M} \defeq \big\{\varphi\in W^{1,k}(S):\quad \|\varphi\|_{W^{1,k}}\leq M \big\}
  \end{equation*}
  and where $W^{1,k}(S)$ stands for the Sobolev space with $k\in(1,\infty]$. We have
  \begin{equation}\label{boundV1}\E\Big[\sup_{(t,\varphi)\in D_{k,M}} \Big\vert V\Big(\frac{t}{\eps^2},\varphi(t,\cdot)\Big)\Big\vert\Big]\leq C + \frac{C(\eps)}{\eps}\,,\end{equation}
  and for any $n\in\N^*$
  \begin{equation}\label{boundV2}\sup_{\eps}\sup_{t\in[0,T]}\E\Big[\sup_{\varphi\in W_{k,M}} \Big\vert V\Big(\frac{t}{\eps^2},\varphi\Big)\Big\vert ^n\Big]\leq C_n\,,\end{equation}
  where $C$, $C_n$ and $C(\eps)$ are three positive constants where the latter satisfies 
  \[\lim_{\eps\to 0} C(\eps) = 0.\] 
\end{lemma} 

Here the notation $V(t/\eps^2,\varphi)$ denotes integral of a function $\varphi(p)$ with respect to the measure $V(t/\eps^2,dp)$.
That is,
\[
  V\Big(\frac{t}{\e^2},\varphi\Big) \defeq \int_S \varphi(p)\,V\Big(\frac{t}{\e^2},dp\Big)\,,
\]
and we will drop the $S$ in the above notation for simplicity unless specified otherwise.
The proofs of these two Lemmas are computationally involved, and we relegate them to Appendices~\ref{proofpropmar} and~\ref{proofpropbound} respectively.

\section{The Main Theorem}\label{s:main}

One canonical way to analyze integrable Hamiltonian systems is to use a set of \emph{action-angle coordinates}.
These coordinates separate the slow and fast motion, and preserves the Hamiltonian structure.

\subsection{Action-angle coordinates}\label{s:actionangle}

The Liouville-Arnold theorem~\cite{Arnold89} asserts that there exists a symplectic canonical transformation $\varphi \colon X = (x,y)\mapsto (I,\theta) \in \R \times \T$, where the action variable $I$ and the angle variable $\theta$ satisfy
\begin{equation}\label{eq:KH}
K(I) = H(x,y)\qquad \text{and} \qquad \{I, \theta\} = 1\,,
\end{equation}
where $K$ is a given one variable smooth enough increasing function such that $K(0)=0$, $\{\cdot,\cdot\}$ stands for the standard Poisson bracket and is defined by
\[\{g,h\} = \partial_x g \partial_y h - \partial_y g \partial_x h\,,\]
for our Hamiltonian system with one degree of freedom.
The relation on the right-hand side of \eqref{eq:KH} can be also stated as
\begin{equation*}
\grad I \cdot \grad^\perp \theta = 1\,.
\end{equation*}
Note that in the action-angle coordinate the Hamiltonian is a function of the action coordinate alone. 

The existence of such transformation $\varphi$ is guaranteed by the Liouville-Arnold theorem, and can be constructed through a generating function \cite[Section 50 pp. 281]{Arnold89}, and implicitly defined by the relations
\begin{equation}\label{generating_eq}
y = \partial_x \mathcal{S}(I,x)\,,\qquad \theta = \partial_I \mathcal{S}(I,x)\,,\qquad\text{and}\qquad H(x, \partial_x \mathcal{S}(I,x)) = K(I)\,.
\end{equation}  
Such a construction is not unique.
For convenience in the forthcoming analysis we will choose $\mathcal{S}(I,0)=0$ so that 
\begin{gather*}
  \mathcal{S}(I,x)=\int_0^x \partial_{x'} \mathcal{S}(I,x')\,dx'\,,
  \\
  \partial_I \mathcal{S}(I,x)=\int_0^x \partial_{I x'} \mathcal{S}(I,x')\,dx'\,,
\end{gather*}
and hence
\[\partial_I \mathcal{S}(I,0)=0\,.\]
The meaning of the second relation is that the zero angle corresponds to either the positive part of the $y$-axis or the negative part. 
Consequently, we have for any $I>0$
\begin{equation}\label{phi_10}
\varphi^{-1}_1 (I, \theta=0)=0 \qquad \text{and}\qquad \varphi^{-1}_2 (I, \theta=0)\neq 0\,,
\end{equation}
and in particular,
\begin{equation}\label{partial_I_phi_10}
\partial_I \varphi^{-1}_1 (I, \theta=0)=0\,.
\end{equation}
The identity~\eqref{phi_10} follows from the fact that $K$ vanishes only for $I=0$.
For any $I>0$, the $x$ and $y$ components cannot simultaneously vanish.

Now let us write $\varphi = (\varphi_1, \varphi_2)$ and define
\begin{equation*}
I_t \defeq I(X_t) = \varphi_1(X_t)\,,
\qquad\text{and}\qquad
\theta_t \defeq \theta(X_t) = \varphi_2(X_t)\,.
\end{equation*}
In the absence of random fluctuations (i.e.\ when $v=0$), the Hamiltonian system~$\dot X_t = \grad^\perp H(X_t)$ becomes
\begin{equation*}
\dot{I}_t = 0\,,
\qquad\text{and}\qquad
\dot{\theta_t} = \omega (I_t)\,,
\qquad\text{with}\qquad
\omega(I) = K'(I)\,.
\end{equation*}
In the presence of random fluctuations, the Hamiltonian system \eqref{e:noisyHamiltonianUnscaled} becomes
\begin{equation}\label{e:Itheta1}
\dot I_t
= \eps v(t) a(I_t, \theta_t) \,,
\qquad\text{and}\qquad
\dot \theta_t
= \omega(I_t)
+\eps v(t) b(I_t, \theta_t)\,,
\end{equation}
where
\begin{equation*}
a = e_2 \cdot \grad \varphi_1 \circ \varphi^{-1} \,,
\quad\text{and}\quad
b = e_2 \cdot \grad \varphi_2 \circ \varphi^{-1} \,.
\end{equation*}

\begin{example}
For illustration purposes, we now explicitly compute the action-angle variables when the Hamiltonian is quadratic.
From \eqref{generating_eq}, we have
\[
  x = \sqrt{\frac{I}{\pi}}\cos \paren[\Big]{ 2\pi \theta- \frac{\pi}{2} }
  \qquad\text{and}\qquad
  y =  \sqrt{\frac{I}{\pi}}\sin \paren[\Big]{ 2\pi \theta- \frac{\pi}{2} } 
\]
where $\theta$ gives the angle of $X=(x,y)^T$ from the negative part of the vertical axis. One can then see that the action variable is a multiple of the Hamiltonian,
\[K(I) = \frac{I}{2\pi}\qquad\text{and} \qquad \omega(I) = \frac{1}{2\pi},\]
 while the angle variable corresponds to the angle on a trajectory, which is a circle in the case of $H(X)=|X|^2/2$, with period $1$.
 In this case, the functions $a$ and $b$ are as follows
\begin{equation}\label{eq:quadab}
  a(I,\theta)= 2\sqrt{\pi I}\sin\paren[\Big]{ 2\pi \theta- \frac{\pi}{2}}
  \quad\text{and}\quad
  b(I,\theta)=\frac{1}{2\sqrt{\pi I}}\cos\paren[\Big]{ 2\pi \theta-\frac{\pi}{2}}\,.
\end{equation}
\end{example}

\subsection{The main theorem}\label{s:mainthm}

To study the long-time behavior of \eqref{e:Itheta1}, we let
\begin{equation*}
I^\eps_t = I_{t/\eps^2}\qquad\text{and}\qquad \theta^\eps_t = \theta_{t/\eps^2}\,.
\end{equation*} 
Now~\eqref{e:Itheta1} becomes
\begin{equation}\label{e:Itheta1eps}
\dot I_t^\eps = \frac{1}{\eps} v\Big(\frac{t}{\eps^2}\Big) a(I_t^\eps,\theta_t^\eps)\,, \qquad \dot \theta_t^\eps = \frac{\omega(I_t^\eps)}{\eps^2} + \frac{1}{\eps} v\Big(\frac{t}{\eps^2}\Big)b(I_t^\eps,\theta_t^\eps) \,,
\end{equation}
with $I^\eps_0 = I_0$ and $\theta^\eps_0 = \theta_0$. Using the approach in~\cite{Garnier95} to separate fast and slow motions, we study this system by splitting the angle variable $\theta_t^\eps$ into two parts
\begin{equation*}
\theta_t^\eps = \psi_t^\eps + \tau_t^\eps \,.
\end{equation*} 
The evolution of $\tau_t^\epsilon$ (fast motion) is obtained by averaging~\eqref{e:Itheta1eps} over the angular coordinate, and $\psi_t^\epsilon$ (slow motion) is the remainder.
That is we require
\begin{equation}\label{e:tau}
\dot \tau_t^\eps = \frac{\omega(I_t^\eps)}{\eps^2} + \frac{1}{\eps}v\Big(\frac{t}{\eps^2}\Big) \av{ b(I_t^\eps, \cdot)}
\end{equation}
with initial condition $\tau^\eps_0 = 0$, and where we have the notation
\begin{equation*}
\av{g} \defeq \int_{\theta = 0}^1 g(\theta) \, d\theta \,.
\end{equation*}
Considering only the slow motion variables $(I_t^\eps,\psi_t^\eps)$, the system~\eqref{e:Itheta1eps} becomes
\begin{equation}\label{systemIpsi}
\begin{split}
\dot I_t^\eps &= \frac{1}{\eps} v\Big(\frac{t}{\eps^2}\Big) A(I_t^\eps,\psi_t^\eps,\tau_t^\eps) \\
\dot \psi_t^\eps &= \frac{1}{\eps} v\Big(\frac{t}{\eps^2}\Big) B(I_t^\eps,\psi_t^\eps,\tau_t^\eps)
\end{split}
\end{equation}
with
\begin{equation*}
A(I,\psi,\tau) \defeq a(I,\psi+\tau)\qquad\text{and}\qquad B(I,\psi,\tau) \defeq b(I,\psi+\tau) - \av{b(I,\cdot)}\,.
\end{equation*}
The above equations are coupled with the initial conditions $I^\eps_0 = I_0$, and $\psi^\eps_0=\theta_0$. Note that~$A$ and $B$ are both 1-periodic and mean-zero with respect to $\tau$. This latter property is mandatory to deal with the long-range correlation case as illustrated in Proposition \ref{p:qConv}.
In fact, if the system possesses a component with zero frequency, a noise with long-range correlations will charge this component and cause the system to blow up as~$\epsilon \to 0$.

It is straightforward to see this mean-zero property in $\tau$ for $B$, but for $A$ we use that
\begin{equation*}
\textrm{Jac}\, \varphi^{-1}(I,\theta) = [\textrm{Jac}\,\varphi(\varphi^{-1}(I,\theta))]^{-1}\end{equation*} 
and that $\det \textrm{Jac}\,\varphi(X) = 1$ since $\varphi$ is a symplectic transformation, to obtain the following relations
\begin{equation}\label{relation_phi}
\begin{split}
\partial_I \varphi^{-1}_1(I,\theta)&= \partial_y \varphi_2(X),\quad \partial_I \varphi^{-1}_2(I,\theta)= -\partial_x \varphi_2(X)\,,\\
\partial_\theta \varphi^{-1}_1(I,\theta)& = - \partial_y \varphi_1(X), \quad \partial_\theta \varphi^{-1}_2(I,\theta)= \partial_x \varphi_1(X)\,.
\end{split}
\end{equation}
Therefore, we have
\begin{equation*}
a(I,\theta) = -\partial_\theta\varphi^{-1}_1(I,\theta)\end{equation*}
which is clearly mean-zero with respect to $\theta$ since we have the derivative of a periodic function. Note also that from these relations the mean-zero property in $\theta$ for $b$ is not clear. This is the reason why we introduce the compensation $\av{b}$ in the definition of $B$.

Our main result obtains the limiting behavior of~\eqref{systemIpsi} as $\epsilon \to 0$ under the following assumptions.
\begin{itemize}[\textbullet]
  \item
    The function~$K$ is smooth, and
    \begin{equation}\label{e:hypOmega}
      \inf_{I\geq 0} K'(I) =\inf_{I\geq 0} \omega(I) >\omega_0 > 0
    \end{equation}
    for some strictly positive number~$\omega_0$.

  \item 
    There exist $r>0$, and positive constants $c_{1,r}$, $c_{2,r}>0$ such that for any $I\in (0,r)$
    \begin{equation}\label{hyp_omega}
      c_{1,r} I \leq K(I) \leq c_{2,r} I \qquad\text{and}\qquad |\omega'(I)| \leq \frac{c_{2,r}}{I} \,.
    \end{equation}
\end{itemize}

Note that these conditions imply that $K(0)=0$ and that $K$ is an increasing function in $I$.
These assumptions are not too restrictive since any Hamiltonian satisfies~\eqref{e:hypOmega} and~\eqref{hyp_omega} near non-degenerate critical points.
The following proposition is a consequence of the above assumptions, which concerns bounds on the function $a$.
The proof is given in Appendix~\ref{proofpropbound_a}.
\begin{proposition}\label{l:bound_a}
  Under \eqref{e:hypOmega} and \eqref{hyp_omega} there exist $r>0$ and a constant $C_r>0$ such that for any $I \in (0,r)$ we have
  \[
  \sup_{\theta \in \mathbb{T}} \Big(|a(I,\theta)| + |\partial_\theta a(I,\theta)|\Big) \leq C_r \sqrt{I}\,,
  \]
  and
  \[
  \sup_{\theta \in \mathbb{T}} |\partial_I a(I,\theta)| \leq \frac{C_r}{ \sqrt{I}}\,.
  \]
\end{proposition}
The proposition mainly describes the behavior of $a, \partial_Ia$ and $\partial_{\theta}a$ around $I=0$. These bounds will allow us to prove the main result of the paper that characterizes the limiting process of the sequence of processes $(I^{\eps},\psi^{\eps})_{\eps>0}$.
\begin{theorem}\label{mainresult}
  Assume \eqref{e:hypOmega} and \eqref{hyp_omega} hold, the family $(I^\eps, \psi^\eps)_{\eps>0}$ (defined in~\eqref{e:tau}--\eqref{systemIpsi}) converges in distribution in $\mathcal{C}([0,\infty),\mathbb{R}^2)$ to a process $(I_t,\psi_t)_{t\geq 0}$, where $(I_t)_{t\geq 0}$ is the unique weak solution of the SDE
  \begin{equation}\label{eq:I}\begin{split}
  dI_{t}&=\int_{\tau=0}^1  \, a(I_t,\tau)\, dW_t(I_t,\tau)\, d\tau \\
  &\quad+\Big[\int_{u=0}^\infty R(u)\int_{\tau=0}^1 \partial_I \big(a(I,\tau + \omega(I)u)\big)_{|I=I_t}  a(I_t,\tau)\\
  &\qquad\qquad+ a(I_t,\tau + \omega(I_t)u) \partial_I a(I_t,\tau) \,d\tau\,du\Big]\, dt\,,
  \end{split}\end{equation}
  with initial condition $I_{t=0}=I_0$. Here, $W$ is a real valued Brownian field with covariance function
  \begin{equation}\label{def_cov}
    \begin{split}
      \E[ W_t(y,\phi_1)W_s(y,\phi_2)]
      &=  t\wedge s
      \int_{u=0}^\infty
      R(u)
      \\
      &\quad\times
      \int_{\tau=0}^1
      \phi_1(\tau + \omega(y)u) \phi_2(\tau)
      + \phi_1(\tau)\phi_2(\tau + \omega(y)u)
      \, d\tau \,du\,,
    \end{split}
  \end{equation}
  for any $\phi_1,\phi_2 \in L^2_0(\T)$ where
  \[L^2_0(\T) \defeq \big\{\phi \in L^2(\T): \av{\phi} = 0\big\}\,.\]
  Explicitly, we can write
  \[
	  W_t(y,\tau) = \frac{1}{\sqrt{2}} \sum_{n\in\Z^*} e^{2i\pi n\tau}R^{1/2}_n(y) (B^1_{t,n}-iB^2_{t,n})
  \]
  where $B^1_{t,n}$ and $B^2_{t,n}$ are two independent real valued cylindrical Brownian motions satisfying $B^1_{t,-n}=B^1_{t,n}$, $B^2_{t,-n}=-B^2_{t,n}$, and $R_n=2\int_{0}^{\infty}R(u)\cos(2\pi n\omega(I)u)\,du$.
  
  Also, we have
  \begin{equation}\label{eq:psi}\begin{split}
  d\psi_{t}&=\int_{\tau=0}^1  \, b(I_t,\tau)\, dW_t(I_t,\tau) \,d\tau \\
  &\qquad+\Big[ \int_{u=0}^\infty R(u) \int_{\tau=0}^1 \partial_I \big(b(I,\tau + \omega(I)u)\big)_{|I=I_t}  a(I_t,\tau) \\
  &\qquad\qquad+ b(I_t,\tau + \omega(I_t)u) \partial_I a(I_t,\tau) \,d\tau\,du\,\Big ] dt\,,
  \end{split}\end{equation} 
  with $\psi_{t=0} = \theta_0$.
\end{theorem}

Note that the action variable~$I$ does not depend on the slow angular motion $\psi$.
However, the distribution of the slow angular motion~$\psi$ is completely determined by the motion of the action variable~$I$ and the Brownian field.

We also note that the functions~$\partial_I a$ and~$\partial_I b$ appearing in~\eqref{eq:I} and~\eqref{eq:psi} may be singular at $I = 0$.
Nevertheless, since the minimum of $H$ is non-degenerate, the function~$a \partial_I a$ has no singularity at $I = 0$ and there is no singularity on the right-hand side of equation~\eqref{eq:I}.
The term $a \partial_I b$ appearing in~\eqref{eq:psi} may still have an $O(I^{-1/2})$ singularity at $I = 0$.
However, since the point $I = 0$ is inaccessible (Proposition~\ref{p:etaInf}), the right-hand side of~\eqref{eq:psi} is well defined.

We now use Theorem~\ref{mainresult} to prove Theorem~\ref{t:mainIntro}.
\begin{proof}[Proof of Theorem~\ref{t:mainIntro}]
  Using action angle coordinates, we convert the system~\eqref{e:hamiltonian} to~\eqref{systemIpsi}.
  Since the Hamiltonian~$H$ has exactly one non-degenerate critical point, the assumptions~\eqref{e:hypOmega}--\eqref{hyp_omega} are satisfied.
  Now by Theorem~\ref{mainresult} we see that the family of processes~$(I^\epsilon, \psi^\epsilon)_{\eps>0}$ converges in distribution to the pair $(I, \psi)$ which solves~\eqref{eq:I}, \eqref{eq:psi}.
  Since $H(X^\epsilon) = K(I^\epsilon)$, we now obtain Theorem~\ref{t:mainIntro} by considering $\mathcal{X} = K(I)$. considering $\mathcal{X} = K(I)$. The infinitesimal generator for $I$ can be written in the form 
  \[
  \frac{1}{2}(\partial_I \tilde \Sigma (I) \partial_I) 
  \] 
  where 
  \[
  \tilde \Sigma(I) = 2\int_{u=0}^\infty R(u) \int_{\tau=0}^1 a(I,\tau+\omega(I)u)a(I,\tau) \,du \,d\tau\,.
  \]
  Now, using \eqref{eq:relation_a}, \eqref{eq:partial_theta_varphi} and that $\varphi^{-1}(I, \tau + \omega(I)u)=\check X^{x,y}_u$ for $\varphi^{-1}(I,\tau)=(x,y)$, together with $\omega(I) = 1/\Lambda(\mathcal X)$, we have
  \[
  \tilde \Sigma(I) = \frac{2}{\omega(I)}\int_{u=0}^\infty du \, R(u) \oint_{\{H=\mathcal X\}} \partial_y H(\check X^{x,y}_u) \partial_y H(x,y) |\nabla H(x,y)|^{-1}dl(x,y)\,.
  \]
  Finally, using $\partial_I = \omega(I)\partial_{\mathcal X}$ provides the desired result.
\end{proof}

For the quadratic Hamiltonian we can explicitly derive the stochastic differential equations for $(I_t)_{t\geq 0}$ and $(\psi_t)_{t\geq 0}$. Using equation \eqref{eq:quadab} we obtain
\begin{equation}\label{eq:I_quad}
d I_t = 2 \sqrt{m I_t} dB^1_t  + 2 m \, dt
\qquad\text{and}
\qquad
d \psi_t = \frac{m}{2\pi \sqrt{I_t}}\, dB^2_t
\end{equation}
with
\[
m = \pi \int_0^\infty R(u)\cos(u)\, du\,,
\]
and where $B^1$ and $B^2$ are two independent standard Brownian motions. One can easily remark that $(I_{t}/m)_{t\geq 0}$ is a 2-dimensional squared Bessel process, so that $(I_{t}/ 2\pi)_{t\geq 0}$ has the same law as $(|W_t|^2/2)_{t\geq 0}$, for $W$ being a 2-dimensional Brownian motion with covariance matrix given by \eqref{e:qcov}. Then, we recover the result of Proposition \ref{p:qConv}.

   The proof of Theorem \ref{mainresult} will be given in Section \ref{sec:property_limit_eq}.
\subsection{Generalization for Hamiltonian systems with one degree of freedom and 2-dimensional noises} 

Theorem \ref{mainresult} can be readily extended to general Hamiltonian systems with one degree of freedom and 2-dimensional noises. In fact, our proof does not depend on the particular shapes of Hamiltonians for a one-degree of freedom oscillator problems. Also, considering 2-dimensional noises brings no additional difficulties except more tedious notations in the proof of Theorem \ref{mainresult}. Here we state this extension of Theorem \ref{mainresult} for the sake of completeness, but the detailed proof will be omitted.

We consider the following one-degree of freedom Hamiltonian system
\[
X^\eps_t = \frac{1}{\eps} \nabla^\perp H(X^\eps_t) + \frac{1}{\eps} \mathbf{v}\Big(\frac{t}{\eps^2}\Big)\,,\qquad X^\eps_0 = X_0\,,
\]
where the Hamiltonian $H$ is as in Theorem \ref{mainresult}, and $\mathbf{v}$ is a 2-dimensional noise with covariance function
\[\mathbf{R}_{jl}(t-s) \defeq\E[\mathbf{v}_{j}(t)\mathbf{v}_{l}(s)] =  \int_S \frac{\lambda_{jl}(p)}{|p|^{2\alpha}} e^{-\mu |p|^{2\beta} |t-s|}\,dp\qquad j,l\in\{1,2\}\,,\]
where $(\lambda_{jl})_{j,l\in\{1,2\}}$ is a $2\times 2$ symmetric matrix valued function with assumptions similar to the ones considered in Section \ref{s:noise}.

Introducing the same action-angle coordinates as in Section \ref{s:actionangle}, they now satisfy
\[
\dot I_t^\eps = \frac{1}{\eps} \mathbf{v}\Big(\frac{t}{\eps^2}\Big) \cdot \mathbf{A}(I_t^\eps,\psi_t^\eps,\tau_t^\eps) \qquad\text{and}\qquad
\dot \psi_t^\eps = \frac{1}{\eps} \mathbf{v}\Big(\frac{t}{\eps^2}\Big) \cdot \mathbf{B}(I_t^\eps,\psi_t^\eps,\tau_t^\eps),
\]
with 
\begin{equation*}
\dot \tau_t^\eps = \frac{\omega(I_t^\eps)}{\eps^2} + \frac{1}{\eps}\mathbf{v}\Big(\frac{t}{\eps^2}\Big) \cdot \av{ b(I_t^\eps, \cdot)}\,,
\end{equation*}
and where
\begin{equation*}
\mathbf{A}(I,\psi,\tau) \defeq \begin{pmatrix} a_1(I,\psi+\tau) \\ a_2(I,\psi+\tau) \end{pmatrix} \qquad\text{and}\qquad \mathbf{B}(I,\psi,\tau) \defeq \begin{pmatrix} b_1(I,\psi+\tau) - \av{b_1(I,\cdot)} \\ b_2(I,\psi+\tau) - \av{b_2(I,\cdot)} \end{pmatrix},
\end{equation*}
with
\[
a_j(I,\theta) = e_j \cdot \nabla \varphi_1(\varphi^{-1}(I,\theta))
\qquad\text{and}\qquad
b_j(I,\theta)=e_j \cdot \nabla \varphi_2(\varphi^{-1}(I,\theta)) \,,
\]
for $j \in \set{1, 2}$.

The extension of Theorem \ref{mainresult} is as follows.
\begin{theorem}\label{t:mainGen}
The family $(I^\eps, \psi^\eps)_{\e>0}$ converges in distribution in $\mathcal{C}([0,\infty),\mathbb{R}^2)$ to a process $(I_t,\psi_t)_{t\geq 0}$, where $(I_t)_{t\geq 0}$ is the unique weak solution of the SDE
  \begin{equation}\label{eq:IT}\begin{split}
  dI_{t}&=\int_{\tau=0}^1  \, a(I_t,\tau) \cdot dW_t(I_t,\tau) \,d\tau \\
  &\quad+\int_{u=0}^\infty  \int_{\tau=0}^1 \partial_I \big(a^T(I,\tau + \omega(I)u)\big)_{|I=I_t} \mathbf{R}(u) a(I_t,\tau) \\
  &\qquad\qquad+ a^T(I_t,\tau + \omega(I_t)u) \mathbf{R}(u) \partial_I a(I_t,\tau) \,d\tau\,du\, dt,
  \end{split}\end{equation} 
  with initial condition $I_{t=0}=I_0$. Here, $W$ is a 2-dimensional real valued Brownian field with covariance function
  \begin{align*}
  \E[ W^j_t(y,\phi_1)W^l_s(y,\phi_2)]
  &=  \paren{ t\wedge s }
  \int_{u=0}^\infty
  R_{jl}(u)
  \\
  &\quad\times
  \int_{\tau=0}^1
  \phi_1(\tau + \omega(y)u) \phi_2(\tau)
  + \phi_1(\tau)\phi_2(\tau + \omega(y)u)
  \, d\tau \,du\,,
  \end{align*}
  for any $j,l\in\{1,2\}$ and $\phi_1,\phi_2 \in L^2_0(\T)$ where
  \[L^2_0(\T) \defeq \big\{\phi \in L^2(\T): \av{\phi} = 0\big\}\,.\]
  Also, we have
  \begin{equation}\label{eq:psiT}\begin{split}
  d\psi_{t}&=\int_{\tau=0}^1  \, b(I_t,\tau) \cdot dW_t(I_t,\tau)\, d\tau \\
  &\quad+\int_{u=0}^\infty  \int_{\tau=0}^1 \partial_I \big(b^T(I,\tau + \omega(I)u)\big)_{|I=I_t} \mathbf{R}(u) a(I_t,\tau) \\
  &\qquad\qquad+ b^T(I_t,\tau + \omega(I_t)u) \mathbf{R}(u) \partial_I a(I_t,\tau) \,d\tau\,du\, dt\,,
  \end{split}\end{equation} 
  with $\psi_{t=0} = \theta_0$.
\end{theorem}

\section{Properties of the Limiting Equation for the Action Variable}\label{sec:proofI}

In this section, we study properties of the equation \eqref{eq:I}. The main issue concerns the behavior of $a$ at $I=0$, but we will see how to define a unique global solution for this equation that does not reach $0$ with probability one. In the case of a quadratic Hamiltonian, we have seen in \eqref{eq:I_quad} that $(I_t)_{t\geq 0}$ is related to a 2-dimensional squared Bessel process (see \eqref{eq:I_quad}) which is known to reach $0$ with probability zero \cite[Proposition 3.22 p. 161]{KaratzasShreve91}. 

\subsection{Existence of a solution} First, we remark that the function $a$ is not Lipschitz in $I$ as it typically has a square-root singularity near~$0$.
So it is not immediately apparent that equation~\eqref{eq:I} has solutions.
To construct solutions to equation~$\eqref{eq:I}$, define
\begin{gather}
  \label{def_fourieraRa}
  a_n(I) \defeq \int_{\tau=0}^1 a(I,\tau)e^{-2i\pi n \tau}\,d\tau
  \\
  \label{def_fourieraRR}
  R_n(I) \defeq2 \int_{u=0}^{\infty} R(u)\cos(2\pi n \omega(I)u)\,du\,.
\end{gather}
and let
\begin{align}\label{e:bfa}
  \nonumber
  \mathbf{a}(I) &\defeq \sum_{n\in\mathbb{Z^\ast}} |a_n(I)|^2 R_n(I)
    \\
    &= 2 \int_{u=0}^\infty R(u) \int_{\tau = 0 }^1  a(I,\tau) a (I, \tau + \omega(I) u) \, d\tau \,du\,,
  \\
  \nonumber
  \mathbf{b}(I) &\defeq \int_{u=0}^\infty R(u) \int_{\tau = 0 }^1\Big[ \partial_I (a (I, \tau + \omega(I) u) )  a(I,\tau)
  \label{def_bfb}
  \\
    &\qquad\qquad\qquad\qquad\qquad+ \partial_I a(I,\tau) a (I, \tau + \omega(I) u)\Big] \,d\tau \,du\,.
\end{align}
Note that equality in~\eqref{e:bfa} is justified by using the Fourier series expansion for $a(I,\tau)$, and using the fact that $a(I,\tau)$ is mean-zero in $\tau$.
That is,
\[
a(I,\tau) = \sum\limits_{n\in\mathbb Z^*}e^{2i\pi n\tau}a_n(I)\,.
\]

\begin{definition}\label{def_eq_I}
  A weak solution in the interval $(0,\infty)$ to equation~\eqref{eq:I} is a filtered probability space $(\Omega, \mathcal{F}, \mathcal F_t, \P)$ satisfying the usual conditions and a continuous pair of processes $(I_t, W_t)$ such that the followings hold.
  \begin{enumerate}
    \item
      The process $I$ takes values in $[0, \infty]$ with $I_0 \in (0, \infty)$,
    \item
      The process~$W$ is a standard $\mathcal F_t$-adapted Brownian motion on the Hilbert space $L^2_0(\T)$ with covariance function given by~\eqref{def_cov}.
    \item 
      For any $t, M > 0$ we have
      \begin{gather}
  \int_{s=0}^{t\wedge \zeta_M} ( \mathbf{a}(I_s) + |\mathbf{b}(I_s)|) \, ds < \infty \,,
  \\
  \label{EDS_I_def}
  I_{t\wedge \zeta_M} = I_0  + \int_{s=0}^{t\wedge \zeta_M}\int_{\tau=0} a(I_s,\tau)\, dW_s(I_s,\tau)  + \int_{s=0}^{t\wedge \zeta_M} \mathbf{b}(I_s)\,ds\quad \forall t\geq 0\,,
      \end{gather}
      almost surely.
      Here $\zeta_M \defeq \zeta_{1, M} \varmin \zeta_{2, M}$ where $\zeta_{1, M}$ and $\zeta_{2, M}$ are defined by
      \begin{align*}
  \zeta_{1,M} &\defeq\inf\{t\geq 0 \st |I_t| \geq M \}\\
  \zeta_{2,M} &\defeq\inf\{t\geq 0 \st |I_t| \leq 1/M \}\,.
      \end{align*}
  \end{enumerate}
\end{definition}
For notational convenience we denote
\begin{equation}\label{def_zetainfty}
\zeta_\infty \defeq \lim_{M\to \infty} \zeta_M = \inf\{t\geq 0 \st I_t\not\in (0,\infty)\}
\end{equation}
to be the first exit time of $I$ from $(0, \infty)$.

We will now work with the canonical probability space~$\mathcal C([0, \infty), \R)$ and canonical filtration 
\[
  \mathcal{M}_t = \sigma(h_s,\,0\leq s\leq t)\,.
\]
To construct a solution to \eqref{eq:I}, we note first that when restricted to the interval~$[1/M, M]$, the function~$a$ is Lipschitz.
Thus we know~\eqref{eq:I} has a strong solution when truncated to this interval (see Proposition~\ref{prop_main}, below).
If we denote this truncated solution by $(I^M, \zeta_M)$ we obtain an increasing sequence of stopping times~$\zeta_M$, an increasing family of~$\sigma$-algebras $\mathcal M_{\zeta_M}$, and a sequence of probability measures~$\P^M_I$ such that
\[ \P^{M+1}_I = \P^{M}_I \qquad\text{on} \quad \mathcal{M}_{\zeta_M}\qquad\forall M>0\,. \]
As a result, we can define $\P^0$ on $\bigcup_{M} \mathcal{M}_{\zeta_M}$ by
\[\P^0_I(O) = \P^M_I(O)\qquad \text{for}\quad O \in \mathcal{M}_{\zeta_M}\,,\]
and extend $\P^0_I$ to a probability measure on the $\sigma$-algebra $\mathcal{M}_{\zeta_\infty}$
such that for any $M \in \N$ we have
\begin{equation}\label{restrict_P0}
\P^{0}_I = \P^{M}_I \qquad\text{on} \quad \mathcal{M}_{\zeta_M}\,.
\end{equation}
Note that
\begin{equation}\label{tribu_eta_inf}
  \mathcal{M}_{\zeta_\infty} = \sigma\paren[\Big]{\bigcup_M \mathcal{M}_{\zeta_M} } \,,
\end{equation}
which can be easily verified from the fact that
\[\mathcal{M}_{\zeta_\infty} = \sigma(h_{t\wedge \zeta_\infty}, t\geq 0)\,,\] 
(see for instance~\cite[Lemma 1.3.3 p. 33]{StroockVaradhan06}).
As a result $\P^0_I$ provides a solution to \eqref{eq:I}, that we can denote by $I=(I_t)_{t\geq 0}$, in the sense of Definition \ref{def_eq_I}.
Note that the process $(I_t)_{t\geq 0}$ is nonnegative with probability one thanks to the Portmanteau theorem \cite[Theorem 2.1 p. 16]{Billingsley99}.

In the next two sections, we will first show (using the bounds in Proposition \ref{l:bound_a}) that  $\P^0_I(\zeta_\infty = \infty) = 1$.
That is, the process~$I$ does not reach~$0$ or blow up in finite time.
This will show that $\P^0_I$ is the unique extension on $\mathcal{M}\defeq \sigma(\bigcup_{t\geq0}\mathcal{M}_t)$ satisfying \eqref{restrict_P0}.
Additionally, we will show that $I=(I_t)_{t\geq 0}$ is the unique global solution to \eqref{eq:I}. 

\subsection{Non-explosion and inaccessibility to~\texorpdfstring{$0$}{0}}

Our aim in this section is to show that the process~$I$ can not explode in finite time nor reach $0$.
\begin{proposition}\label{p:etaInf}
  Consider a weak solution to~\eqref{eq:I} (as in definition~\ref{def_eq_I}), and let~$\zeta_\infty$ be as in~\eqref{def_zetainfty}.
  Then
  \[
    \P^0_I(\zeta_\infty = \infty) = 1\,.
  \]
\end{proposition}
This result together with \eqref{restrict_P0} implies that for any $t\geq 0$
\[
\lim_{M\to \infty}\P^M_I(\zeta_M \leq t) = \lim_{M\to \infty}\P^0_I(\zeta_M \leq t) = 0\,.
\]
Therefore, \cite[Theorem 1.3.5 p. 34]{StroockVaradhan06} guarantees that $\P^0_I$ is the unique extension satisfying \eqref{restrict_P0}. 

Before proving Proposition~\ref{p:etaInf} we need to introduce the \emph{scale function}
\[p(x) \defeq \int_1^x \exp\left(-2 \int_1^\xi \frac{\mathbf{b}(\nu)}{\mathbf{a}(\nu)} \,d\nu \right) d\xi,\quad x\in(0,\infty)\]
where $\mathbf{a}$ and $\mathbf{b}$ are defined by \eqref{e:bfa} and \eqref{def_bfb} respectively.
Observe
\begin{equation}\label{diff_eq_p}
p''(x) = -2 \frac{\mathbf{b}(x)}{\mathbf{a}(x)} p'(x).
\end{equation}
Standard results express the exit probability of $I$ in terms of the scale function, and we prove them in our context for completeness.
\begin{lemma}\label{l:exitProbA}
  For any $0 < x_- < x_+ < \infty$ we have
  \[\P(I_{\zeta_{x_-,x_+}} =x_-) = \frac{p(x_+)-p(I_0)}{p(x_+)-p(x_-)}\quad\text{and}\quad \P(I_{\zeta_{x_-,x_+}} =x_+) = \frac{p(I_0) - p(x_-)}{p(x_+)-p(x_-)}\,,\]
  where $\zeta_{x_-,x_+} = \zeta_{1,x_+}\wedge \zeta_{2,x_-}$.
\end{lemma}
\begin{proof}[Proof of Lemma \ref{l:exitProbA}.] The proof consists of applying the It\^o formula for $p(I_t)$ with $t\leq \zeta_{x_-,x_+} \leq \zeta_\infty$. To do this we rewrite the stochastic integral of $\eqref{eq:I}$ in the Fourier domain, that is
  \[
  \int_{s=0}^t \int_{\tau=0}^1 a(I_s,\tau)\,dW_s(I_s,\tau)\,d\tau = \int_{s=0}^t \sum_{n\in\mathbb{Z}^\ast} a_n(I_s) R^{1/2}_n(I_s) \frac{1}{\sqrt{2}}(dB^1_{s,n} - i dB^2_{s,n})\,, 
  \]    
  providing an explicit semi-martingale representation for $(I_t)_{t\geq 0}$. Therefore, applying the It\^o formula to $p(I_t)$ we have 
  \[\begin{split}
      \MoveEqLeft
  p(I_{t\wedge \zeta_{x_-,x_+}})  = p(I_{0}) + \int_0^{t\wedge \zeta_{x_-,x_+}} p'(I_s) \,dI_s + \frac{1}{2}\int_0^{t\wedge \zeta_{x_-,x_+}} p''(I_s) d\langle I\rangle_s \\
    &= p(I_{0}) + \int_0^{t\wedge \zeta_{x_-,x_+}} p'(I_s) \sum_{n\in\mathbb{Z}^\ast} a_n(I_s) R^{1/2}_n(I_s) \frac{1}{\sqrt{2}}(dB^1_{s,n} - i dB^2_{s,n})\\
    &\qquad+ \int_0^{t\wedge \zeta_{x_-,x_+}} p'(I_s) \mathbf{b}(I_s)\, ds\\
    &\qquad+ \frac{1}{2}\int_0^{t\wedge \zeta_{x_-,x_+}}  -2 \frac{\mathbf{b}(I_s)}{\mathbf{a}(I_s)}p'(I_s)  \sum_{n\in\mathbb{Z}^\ast} |a_n(I_s)|^2 R_n(I_s)\,ds\\
    & = p(I_{0}) + \int_0^{t\wedge \zeta_{x_-,x_+}} p'(I_s) \sum_{n\in\mathbb{Z}^\ast} a_n(I_s) R^{1/2}_n(I_s) \frac{1}{\sqrt{2}}(dB^1_{s,n} - i dB^2_{s,n})
  \end{split}\] 
  thanks to \eqref{diff_eq_p} and \eqref{e:bfa}. Note that the stochastic integral on the right-hand side of the last line is a martingale starting from $0$, so by taking the expectation, and passing to the limit $t\to \infty$ we obtain
  \[
  p(I_{0}) = \E[p(I_{\zeta_{x_-,x_+}})] = p(x_-)\P(I_{\zeta_{x_-,x_+}} = x_-)  + p(x_+)\P(I_{\zeta_{x_-,x_+}} = x_+)\,.  
  \]
  Solving for $\P(I_{\zeta_{x_-,x_+}} = x_-)$ and $\P(I_{\zeta_{x_-,x_+}} = x_+)$ we obtain the desired result.
\end{proof}

With this lemma, we can now turn to the proof of Proposition \ref{p:etaInf}.

\begin{proof}[Proof of Proposition \ref{p:etaInf}.]
  Let us start with the following remark. The term $\mathbf{b}$ can be written as 
  \[
  \mathbf{b}(I) = \frac{1}{2}\partial_I \mathbf{a}(I)\,,
  \]
  so that
  \[\begin{split}
  p(0_+) &= \lim_{x\to 0_+} \int_1^x \exp\left(-2 \int_1^\xi \frac{\mathbf{b}(\nu)}{\mathbf{a}(\nu)} d\nu \right) d\xi \\
  &= \lim_{x\to 0_+}\mathbf{a}(1) \int_1^x  \exp\left(-\ln(\mathbf{a}(\xi) ) \right) \, d\xi \\
  & = \lim_{x\to 0_+}\mathbf{a}(1) \int_1^x \frac{d\xi}{\mathbf{a}(\xi) } \, .
  \end{split}\]
  Now, according to \eqref{e:bfa} together with Proposition \ref{l:bound_a}, we have
  \[0\leq \mathbf{a}(\xi) \leq C \, \xi\]
  since for any $n\in \mathbb{Z}^\ast$
  \begin{align}\label{ineq_R}
    \nonumber
  0\leq R_n(\xi)&=2\int_0^{\infty}R(u)\cos(2\pi n\omega(\xi)u)\, du\\
  \nonumber
  &= 2\int_0^{\infty} \Big(\int_S r(p)e^{-\mu|p|^{2\beta}u} \, dp\Big)\cos(2\pi n\omega(\xi)u)\, du \\
  &\nonumber =2\int_S \,dp\,r(p) \Big( \int_0^{\infty} \cos(2\pi n\omega(\xi)u)e^{-\mu |p|^{2\beta}u}\,du \Big) \\
  &=\int_S dp\,r(p)\frac{2\mu|p|^{2\beta}}{\mu^2|p|^{4\beta}+4\pi^2 n^2\omega^2(\xi)}\leq \frac{C_{\omega_0}}{n^2} \int_S |p|^{2\beta} r(p)  \,dp\,.
  \end{align}
  As a result, we have
  \[p(0_+) \leq\lim_{x\to 0_  +}\mathbf{a}(1) \int_1^x \frac{d\xi}{C\xi }  = -\infty \,, \]
  which implies $p(0_+)=-\infty$. To conclude the proof we consider two cases, that is $p(\infty)=\infty$ and $p(\infty)<\infty$. Following the same lines as \cite[Proposition 5.22 p. 345]{KaratzasShreve91}, if $p(\infty)=\infty$ we directly have the desired result, and if $p(\infty)<\infty$ we have
  \begin{align*}
  \lim\limits_{x_-\to 0_+}\lim\limits_{x_+\to \infty} \P(I_{\zeta_{x_-,x_+}}=x_+)=\lim\limits_{x_-\to 0_+}\lim\limits_{x_+\to \infty}\frac{p(I_0) - p(x_-)}{p(x_+)-p(x_-)} = 1\,.
  \end{align*}
  
  This implies $\P(\zeta_\infty = \zeta_{1, \infty})=1$ meaning that if $I_t$ exits the interval $(0,\infty)$, it almost surely does so at $\infty$.
  Here
  \[\zeta_{1, \infty} = \lim_{M\to \infty} \zeta_{1, M}\,.\] 
  To prove that $\P( \zeta_{1, \infty} = \infty)=1$, and then conclude the proof of the Proposition \ref{p:etaInf}, we need the following result. 
  \begin{lemma}\label{bound_I}
    For any $T>0$ we have
    \[\lim_{M\to \infty }\P \paren[\Big]{ \sup_{t\in[0,T]} I_{t\wedge \zeta_{M}} \geq M } = 0\,.\]
  \end{lemma}
  Indeed, writing
  \[\begin{split}
  \P(\zeta_{1, \infty} < \infty) & = \lim_{T\to\infty}\P(\zeta_{1, \infty}\leq T) \\
  & = \lim_{T\to\infty}\lim_{M_0 \to \infty}\P\paren[\Big]{ \zeta_{1, \infty}\leq T,\inf_{t\in [0,\zeta_\infty)} I_t > 1/M_0},\\
  &=\lim_{T\to\infty}\lim_{M_0 \to \infty}\lim_{M \to \infty}\P\paren[\Big]{\zeta_{1, M}\leq T,\inf_{t\in [0,\zeta_M)} I_t > 1/M_0} \,.
  \end{split}\]
  To justify the second equality, we note that if $\zeta_{1,\infty}$ is finite, then $\inf_{t\in [0,\zeta_{\infty})}I_t$ must be positive since $\zeta_{1,\infty}=\zeta_{\infty}$ almost surely. We have for any $M>M_0$
  \[ \begin{split}
  \P\Bigl(\zeta_{1, M}\leq T,\inf_{t\in [0,\zeta_M]} I_t > \frac{1}{M_0} \Bigr)&\leq \P(\zeta_{1, M}\leq T,\, \zeta_M =\zeta_{1, M})\\
  & \leq \P\paren[\Big]{\sup_{t\in[0,T]} I_{t\wedge \zeta_{M}} \geq M }.
  \end{split}\]
  Therefore,   Lemma \ref{bound_I} will imply $\P(\zeta_{1, \infty} < \infty)=0$.
  \begin{proof}[Proof of Lemma \ref{bound_I}.]
    First we remark that using Chebychev's inequality
    \begin{align}\label{ineq_I}
      \nonumber\P\paren[\Big]{\sup_{t\in[0,T]} I_{t\wedge\zeta_{M}} \geq M}
	&= \P\paren[\Big]{ \sup_{t\in[0,T]} I_{t\wedge \zeta_{M}}^2 \geq M^2 }
	\\
	&\leq \frac{1}{M^2} \E \sup_{t\in[0,T]} I_{t\wedge \zeta_{1, M}}^2 \,.
    \end{align}
    Let $T'\leq T$, we also have 
    \begin{align*}
      \E\sup_{t\in[0,T']} I_{t\wedge \zeta_{M}}^2
	&\leq 3I_0 + 3\E\sup_{t\in[0,T']}  N^2_t + 3\E\sup_{t\in[0,T']} \Bigl(\int_{0}^{t\wedge \zeta_{M}} \mathbf{b}(I_s) ds\Bigr)^2 \\
    &= 3I_0 + J_{T'} + K_{T'}\,, 
    \end{align*}
    and where, using \eqref{EDS_I_def}
    \[N_{t\wedge\zeta_{M}} \defeq I_{t\wedge \zeta_{ M}} - I_0 - \int_{0}^{t\wedge \zeta_{M}} \mathbf{b}(I_s)\, ds = \int_{s=0}^{t\wedge \zeta_{M}} \int_{\tau=0}^1 a(I_s,\tau) \,dW_s(I_s,\tau)\,d\tau \]
    is a martingale with quadratic variation
    \[
    \langle N\rangle_t = \int_0^{t\wedge\zeta_{M}} \mathbf{a}(I_s)\,ds\,.
    \]
    
    For the term $J_{T'}$ we apply Burkholder-Davis-Gundy inequality \cite[Theorem 3.28 p. 166]{KaratzasShreve91} together with \eqref{ineq_R} from which we obtain
    \begin{multline*}
    J_{T'}  \leq C \E \int_0^{T'\wedge \zeta_{M}}\mathbf{a}(I_{s})\,ds  
    \leq C \int_0^{T'} \E[\mathbf{a}(I_{s\wedge \zeta_{M}})]\,ds
    \\
    \leq C_{\omega_0} \int_S |p|^{2\beta} |r(p)|\,dp \int_{s=0}^{T'} \int_{\tau=0}^1 \E[|a(I_{s\wedge \zeta_{M}},\tau)|^2] \, d\tau \, ds \,.
    \end{multline*}
    Moreover, using Proposition \ref{l:bound_a}, we have 
    \[\begin{split}
    \int_{0}^1 \E [|a(I_{s\wedge \zeta_{M}},\tau)|^2] \, d\tau
    &= \int_{0}^1 \Bigl(
    \E [|a(I_{s\wedge \zeta_{M}},\tau)|^2\mathbf{1}_{(I_{s\wedge \zeta_{M}} \leq r)}]
    \\
    &\qquad+ \E [|a(I_{s\wedge \zeta_{M}},\tau)|^2\mathbf{1}_{(I_{s\wedge \zeta_{M}} > r)}]
    \Bigr)
    \, d\tau
    \\
    &\leq r\,C + C_r \E I_{s\wedge \zeta_{M}}^2\\
    &\leq r\,C + C_r \E\sup_{t\in[0,s]}I_{t\wedge\zeta_{M}}^2 \,,
    \end{split}\]
    and hence
    \[
    J_{T'} \leq C_1\, T + C_2 \int_{0}^{T'} \E\paren[\Big]{ \sup_{t\in[0,s]}I_{t\wedge\zeta_{M}}^2 } ds\,.
    \]
    Now, for $K_{T'}$, we use Cauchy-Schwarz inequality
    \[\begin{split}
    \E\sup_{t\in[0,T']} \Bigl(\int_{0}^{t\wedge \zeta_{M}} \mathbf{b}(I_s)\, ds\Bigr)^2 & \leq T \E \sup_{t\in[0,T']} \int_{0}^{t\wedge \zeta_{M}} \mathbf{b}^2(I_s) \,ds\\
    &\leq T \E\int_{0}^{T'} \mathbf{b}^2(I_{s\wedge \zeta_{M}})\, ds  \,.
    \end{split}\]
    From Proposition \ref{l:bound_a}, for $I\leq r$, we again use the Cauchy-Schwarz inequality
    \[\begin{split}
    |\mathbf{b}(I)|&  \leq C \,|\omega'(I)| \, \sqrt{\int_0^1 |a(I,\tau)|^2\,d\tau \int_0^1 |\partial_\tau a(I,\tau)|^2\,d\tau}\\
    &+C_{\omega_0}  \sqrt{\int_0^1 |a(I,\tau)|^2\,d\tau \int_0^1 |\partial_I a(I,\tau)|^2\,d\tau}\\
    &\leq C.
    \end{split}\]
    As a result,
    \[\begin{split}
    \E \int_{0}^{T'} \mathbf{b}^2(I_{s\wedge \zeta_{M}}) \,ds
      &= \E \int_{0}^{T'} \mathbf{b}^2(I_{s\wedge \zeta_{M}}) \mathbf{1}_{(I_{s\wedge \zeta_{M}}\leq r)}\, ds
    \\
    &\qquad + \E  \int_{0}^{T'} \mathbf{b}^2(I_{s\wedge\zeta_{M}}) \mathbf{1}_{(I_{s\wedge \zeta_{M}}> r)} \,ds 
    \\
    &\leq C_1\, T + C_{2,r} \E \int_{0}^{T'} I_{s\wedge \zeta_{M}}^2 \mathbf{1}_{(I_{s\wedge\zeta_{M}}> r)} \,ds
    \\
    & \leq C_1\, T + C_{2,r} \int_{0}^{T'} \E\Bigl( \sup_{t\in[0,s]} I_{t\wedge \zeta_{M}}^2 \Bigr) \,ds \,,
    \end{split}\]
    and 
    \[
    K_{T'} \leq C_1\, T^2 + C_2 \, T \int_{0}^{T'} \E\Bigl(\sup_{t\in[0,s]} I_{t\wedge \zeta_{M}}^2 \Bigr) \,ds\,.
    \]
    Finally, we have
    \[
    \E \sup_{t\in[0,T']} I_{t\wedge \zeta_{M}}^2  \leq C_1(1 + T + T^2) + C_2(1 + T) \int_{0}^{T'} \E \Bigl( \sup_{t\in[0,s]} I_{t\wedge \zeta_{M}}^2 \Bigr) \, ds
    \]
    and then by Gronwall's inequality with $T'=T$
    \[
    \E \sup_{t\in[0,T]} I_{t\wedge \zeta_{M}}^2 \leq C_1(1 + T + T^2)e^{C_2(1 + T)T}\,.
    \]
    This concludes the proof of the Lemma by letting $M\to \infty$ in \eqref{ineq_I}. Consequently, the proof of Proposition \ref{p:etaInf} is also complete.
  \end{proof}
  \let\qed\relax
\end{proof}

\subsection{Uniqueness of solutions}\label{sec:uniqueness}

The weak uniqueness property for equation~\eqref{eq:I} comes from the fact that this equation has strong uniqueness on $\mathcal{M}_{\zeta_{2,M}}$ for any $M$, and then weak uniqueness on $\mathcal{M}_{\zeta_{2,M}}$ for any $M$ \cite[Proposition 3.20 p. 309]{KaratzasShreve91}, since we avoid the lack of regularity at $0$ for the action variable $I$. However, according to Proposition \ref{p:etaInf}, we have $\lim_{M\to \infty}\P^0_I(\zeta_{2,M}\leq T)=0$ for any $T>0$ which guarantees the weak uniqueness thanks to \cite[Theorem 1.3.5 p. 34]{StroockVaradhan06}.

\section{Proof of the main theorem (Theorem~\ref{mainresult})}\label{sec:property_limit_eq}

\subsection{Truncated Process}\label{sec:truncation}
We first describe the strategy used to prove Theorem \ref{mainresult}.
In addition to the difficulties concerning the behavior of $a$ around $I=0$ that we mentioned in Theorem \ref{mainresult}, we have no \emph{apriori} estimates for the process $(I^\eps_t,\psi^\eps_t)_{t\geq 0}$ that are uniform in $\eps$, in probability. This is a problem when trying to apply \cite[Theorem 4 p. 48]{Kushner84} directly to prove the tightness property, and when trying to identify the law of the subsequential limits.
To bypass this we follow the strategy developed in \cite[Chapter 11]{StroockVaradhan06} by introducing a \emph{truncated process} that does not suffer from the above problems. To relate the truncated process with the original one we introduce a family of stopping times. As we will see, these times go to $\infty$ when we remove the truncation and then finally obtain the weak convergence of the original processes.    


Let $M>0$ and consider the process $(I^{\eps,M}_t,\psi^{\eps,M}_t)_{t\geq 0}$ which is the unique solution to
\begin{equation}\label{systemIpsi_M}
\begin{split}
\dot I^{\eps,M}_t &= \frac{1}{\eps} v\Big(\frac{t}{\eps^2}\Big)  \phi_M(I^{\eps,M}_t, \psi^{\eps,M}_t)a(I^{\eps,M}_t,\psi^{\eps,M}_t+\tau^{\eps,M}_t) \\
\dot \psi^{\eps,M}_t &= \frac{1}{\eps} v\Big(\frac{t}{\eps^2}\Big)  \phi_M(I^{\eps,M}_t, \psi^{\eps,M}_t)(b(I^{\eps,M}_t,\psi^{\eps,M}_t+\tau^{\eps,M}_t) -  \av{b(I^{\eps,M}_t,\cdot)}) 
\end{split}
\end{equation}
coupled with the initial conditions $I^{\eps,M}_0 = I_0$, and $\psi^{\eps,M}_0=\theta_0$,
with 
\begin{equation*}
\dot \tau^{\eps,M}_t = \frac{\omega(I^{\eps,M}_t)}{\eps^2} + \frac{1}{\eps}v\Big(\frac{t}{\eps^2}\Big) \phi_M(I_t^{\eps,M},\psi_t^{\eps,M})\av{ b(I^{\eps,M}_t, \cdot)}\,.
\end{equation*}
Here $\phi_M$ is a smooth function on $\R^2$ such that
\begin{equation*}
0\leq \phi_M\leq 1,\quad\phi_M(I,\psi)=1 \quad \text{if} \quad 1/M \leq I \leq M \quad\text{and}\quad |\psi|\leq M\,,
\end{equation*}
and
\begin{equation*}
\phi_M(I,\psi)=0\quad \text{if}\quad  I > 2M\quad\text{or}\quad  I < 1/(2M)\quad\text{or}\quad  \vert \psi\vert > 2M \,.
\end{equation*}
Thanks to the truncation provided by the cutoff function $\phi_M$, for both action and angle variables, the process $(I^{\eps,M}_t,\psi^{\eps,M}_t)_{t\geq 0}$ does not suffer from the regularity problem for the action variable around $0$ as mentioned above. Also, its convergence in distribution in $\mathcal{C}([0,\infty), \mathbb{R}^2)$ can be proved using the perturbed-test function method and martingale properties \cite{Kushner84,PapanicolaouStroockEA76} thanks to the boundedness of the process. All these points are developed more precisely in Section \ref{sec:propmain}.

To simplify the notation, we now omit the superscript $M$ and denote
\begin{equation*}
Y_t^\eps \defeq \begin{pmatrix} I_t^{\eps, M} \\ \psi_t^{\eps, M}  \end{pmatrix}\,.
\end{equation*}
Now the system~\eqref{systemIpsi_M} can be rewritten as
\begin{equation}
\label{e:defEqDiff}
\dot Y^{\eps}_t
= \frac{1}{\eps} v\paren[\Big]{\dfrac{t}{\eps^2}}
F(Y^\eps_t,\tau^{\eps}_t)
\end{equation}
with
\begin{equation}\label{eq:tau}
\dot \tau^{\eps}_t
= \frac{\omega(Y^{\eps}_t)}{\eps^2}
+ \frac{1}{\eps}  v\paren[\Big]{\dfrac{t}{\eps^2}}G(Y^{\eps}_t)\,.
\end{equation}
Here
\begin{equation}
\label{e:defF}
\begin{split}
F(Y,\tau) &\defeq
\begin{pmatrix}  \phi_M(I, \psi)a(I,\psi+\tau) \\ \phi_M(I, \psi)(b(I,\psi+\tau) -  \av{b(I,\cdot)})  \end{pmatrix},\\
G(Y) &\defeq \phi_M(I,\psi)\av{b(I,\cdot)}\,,
\end{split}
\end{equation}
and $a,b$ are defined in equation~\eqref{e:Itheta1}. Note that $F$ is a smooth bounded function (with bounded derivatives) 1-periodic and mean-zero with respect to $\tau$ (the truncation only affects $\psi$, not $\tau$), and $\omega(Y_t^\eps) = \omega ( I_t^{\eps, M})$.

To relate the truncated and original processes we introduce first some notations. In the remaining of the paper, the space of all possible outcomes is $\mathcal{C}([0,\infty),\R^2)$, and we denote the corresponding canonical filtration by $\mathcal{M}_t = \sigma(y_s,0\leq s\leq t)$. The laws of the truncated and original processes on $\mathcal{C}([0,\infty),\R^2)$ will be denoted respectively by $\P^{\eps, M}$ and $\P^\eps$. Now, we consider the stopping times
\begin{equation*}
\eta_{1,M}(y)\defeq\inf\{t\geq 0: \| y_t \|\geq M \}\quad \text{and}\quad \eta_{2,M}(y)\defeq\inf\{t\geq 0: \vert y^1_t \vert\leq 1/M \}\,,
\end{equation*}
for any $y=(y^1,y^2)\in \mathcal{C}([0,\infty),\R^2)$ and 
\begin{equation}\label{def_etaM}
\eta_{M}\defeq \eta_{1,M} \wedge \eta_{2,M}\,.
\end{equation}
It is not hard to see that $\eta_{1,M}$ and $\eta_{2,M}$ are lower semi-continuous so that $\eta_{M}$ is also lower semi-continuous, that is their lower level sets are closed subsets of $\mathcal{C}([0,\infty),\R^2)$.
The latter property will be used to obtain the convergence of the original process from the truncated one.

From the definition of the above stopping times and the cutoff function $\phi_M$ it is clear that
\begin{equation}\label{e:probAEps}
\P^{\eps,M} = \P^{\eps}\qquad\text{on} \quad \mathcal{M}_{\eta_M}\,.
\end{equation}
This is the relation that links the truncated and original processes and that will be also used below in Section \ref{cv_I_psi} to show the convergence of the original process.

For now, we have the following convergence results for the truncated process. We return the index $M$ for a moment to emphasize the truncation.

\begin{proposition}\label{prop_main}
  
  The family $(Y^{\eps})_{\eps>0}\defeq (I^{\eps,M},\psi^{\eps,M})_{\eps >0}$ converges in distribution in $\mathcal{C}([0,\infty),\R^2)$ to the unique solution $(Y_t^M)_{t\geq 0} = (I_t^M, \psi_t^M)_{t\geq 0}$ to the martingale problem with generator defined by
  
  \begin{equation}\label{def_L}
  \mathcal{L} h(y) \defeq
  \int_{\tau=0}^1 \int_{u=0}^\infty
  R(u)
  \tilde F(0,y,\tau)
  \cdot \nabla ( \tilde F(u,y,\tau)\cdot \nabla h(y))
  \, du \, d\tau
  \,,
  \end{equation}
  where
  \begin{equation*}    
  \tilde F(u,y,\tau) \defeq F(y,\tau +\omega(y)u)\,,
  \end{equation*}
  with $F$ defined in~\eqref{e:defF} and starting point $(I_0,\psi_0)$.
\end{proposition}

From this characterization, we can deduce the SDEs satisfied by the limiting process $Y_t^M = (I_t^M, \psi_t^M)$ up to the stopping time $\eta_M$. 

\begin{proposition}\label{prop_SDE_I_psi_M}
  Denoting $\tilde I^M_t \defeq I^M_{t\wedge \eta_M}$, we have
  \begin{equation}\label{eq:IM}\begin{split}
  d\tilde I^M_t&=\int_{\tau=0}^1  \, a(\tilde I^M_t,\tau)\, dW_{t\wedge \eta_M}(\tilde I^M_t,\tau)\, d\tau \\
  &\quad+\int_{u=0}^\infty R(u)\Big(\int_{\tau=0}^1  \partial_I\big( a(I,\tau + \omega(I)u)\big)_{|I=\tilde I^M_t} a(\tilde I^M_t,\tau) \\
  &\qquad\qquad+ a(\tilde I^M_t,\tau + \omega(\tilde I^M_t)u) \partial_I a(\tilde I^M_t,\tau) \Big)\,d\tau\, du\, dt,
  \end{split}\end{equation} 
  with $\tilde I^M_{t=0} = I_0$, and where $W$ is defined in~\eqref{def_cov}. Also, we have
  \begin{equation}\label{eq:psiM}\begin{split}
  d\psi^M_{t\wedge \eta_M}&=\int_{\tau=0}^1  \, b(\tilde I^M_t,\tau)\, dW_{t\wedge \eta_M}(\tilde I^M_t,\tau) \,d\tau \\
  &\quad+\int_{u=0}^\infty R(u)\Big(\int_{\tau=0}^1 \partial_I \big(b(I,\tau + \omega(I)u)\big)_{|I=\tilde I^M_t}  a(\tilde I^M_t,\tau) \\
  &\qquad\qquad+ b(\tilde I^M_t,\tau + \omega(\tilde I^M_t)u) \partial_I a(\tilde I^M_t,\tau)\Big) \,d\tau\,du\, dt,
  \end{split}\end{equation}
  with $\psi^M_{t=0} = \theta_0$.
\end{proposition}

As we can see, this result is similar to the one of Theorem \ref{mainresult}, with coefficients $a$ and $b$ themselves, but involving the cutoff $M$ through the stopping time $\eta_M$.
The proofs of Proposition~\ref{prop_main} and~\ref{prop_SDE_I_psi_M} are presented in Sections~\ref{sec:propmain} and~\ref{sec:propSDE}, respectively.
In the next section, we will show how these two propositions can be used to prove Theorem \ref{mainresult}.
\subsection{Convergence of \texorpdfstring{$(I^\eps,\psi^\eps)_{\eps>0}$}{(I,psi)}}\label{cv_I_psi}

In order to prove Theorem \ref{mainresult}, we need to remove the cutoff $M$ to show the convergence in distribution of $(I^\eps,\psi^\eps)_{\eps>0}$ to $(I_t,\psi_t)_{t\geq 0}$, where $I$, $\psi$ are defined by \eqref{eq:I}-\eqref{eq:psi}.
This is similar to Lemma 11.1.1 (p. 262) in~\cite{StroockVaradhan06}, and we adapt the proof to our situation.

In what follows, we denote by $\P^M$ and $\P^0$ the distribution of $Y^M=(I^M_t,\psi^M_t)_{t\geq 0}$ (defined in Proposition \ref{prop_main}) and $Y=(I_t,\psi_t)_{t\geq 0}$ (defined in Proposition \ref{mainresult}), respectively. Note that according to \eqref{eq:psiM} and \eqref{eq:psi} the distributions $\P^M$ and $\P^0$ are completely determined by their first marginal $\P^M_I$ and $\P^0_I$. Also, in view of \eqref{eq:IM} and \eqref{eq:psiM}, it is straightforward to see that (with $\eta_M$ defined in~\eqref{def_etaM})
\begin{equation}\label{eq:proba}
\P^0 = \P^M \qquad\text{on}\quad \mathcal{M}_{\eta_M}\,.
\end{equation}
To prove this convergence, let $\Phi$ be a continuous bounded function on $\mathcal{C}([0,\infty),\R^2)$, and write
\begin{align*}
  \MoveEqLeft
  \E^{\P^{\eps}}[\Phi(y_t,t\in [0,T])] - \E^{\P}[\Phi(y_t,t\in [0,T])] \\
    &= \E^{\P^{\eps}}[\Phi(y_t,t\in [0,T])] - \E^{\P^{\eps,M}}[\Phi(y_t,t\in [0,T])]\\
    &\qquad+\E^{\P^{\eps,M}}[\Phi(y_t,t\in [0,T])] - \E^{\P^{M}}[\Phi(y_t,t\in [0,T])]\\
    &\qquad+\E^{\P^{M}}[\Phi(y_t,t\in [0,T])] - \E^{\P^0}[\Phi(y_t,t\in [0,T])]\,.
\end{align*} 
Considering the following decomposition 
\[\begin{split}
\E^{\P^{\eps}}[\Phi(y_t,t\in [0,T])] & = \E^{\P^{\eps}}[\Phi(y_t,t\in [0,T])\mathbf{1}_{\{\eta_M \leq T\}}] \\
&\qquad+ \E^{\P^{\eps}}[\Phi(y_t,t\in [0,T])\mathbf{1}_{\{\eta_M > T\}}] 
\end{split}\]
with 
\[\begin{split}
    \MoveEqLeft
\E^{\P^{\eps}}[\Phi(y_t,t\in [0,T]) \mathbf{1}_{\{\eta_M > T\}}] = \E^{\P^{\eps}}[\Phi(y_{t\wedge \eta_M},t\in [0,T])\mathbf{1}_{\{\eta_M > T\}}] \\
&=\E^{\P^{\eps,M}}[\Phi(y_{t\wedge \eta_M},t\in [0,T])\mathbf{1}_{\{\eta_M > T\}}] \\
&=\E^{\P^{\eps,M}}[\Phi(y_{t},t\in [0,T])\mathbf{1}_{\{\eta_M > T\}}]\\
&=\E^{\P^{\eps,M}}[\Phi(y_{t},t\in [0,T])] - \E^{\P^{\eps,M}}[\Phi(y_{t},t\in [0,T])\mathbf{1}_{\{\eta_M \leq T\}}] \,.
\end{split}\]
Here we used~\eqref{e:probAEps} in the second line.
Using~\eqref{e:probAEps} again we note
\[\begin{split}
\E^{\P^{\eps}}[\Phi(y_t,t\in [0,T])] & - \E^{\P^{\eps,M}}[\Phi(y_t,t\in [0,T])]\\
& \leq \|\Phi\|_\infty (\P^{\eps}(\eta_M \leq T)+\P^{\eps,M}(\eta_M \leq T)) \\
&\leq 2\|\Phi\|_\infty \P^{\eps,M}(\eta_M \leq T)\,.
\end{split}\]

The same results hold for $\P^0$ and $\P^M$ instead of $\P^\eps$ and $\P^{\eps,M}$ respectively, but using \eqref{eq:proba} instead of \eqref{e:probAEps}. As a result, we obtain
\[\begin{split}
\limsup_{\eps\to 0}|\E^{\P^{\eps}}[\Phi(y_t,t\in [0,T])] & - \E^{\P^0}[\Phi(y_t,t\in [0,T])] | \\
&\leq 2\|\Phi\|_\infty\limsup_{\eps\to 0}\P^{\eps,M}(\eta_M \leq T) \\
& + 2\|\Phi\|_\infty\P^0(\eta_M \leq T)\,.
\end{split}\] 

Notice that the stopping time $\eta_M$ is lower semi-continuous on $\mathcal{C}([0,\infty),\R^2)$, so that $(\eta_M \leq M)\in \mathcal{M}_{\eta_M}$ is a closed subset of $\mathcal{C}([0,\infty),\R^2)$.  According to the Portmanteau theorem \cite[Theorem 2.1 p. 16]{Billingsley99} we have
\[
\limsup_{\eps\to 0}|\E^{\P^{\eps}}[\Phi(y_t,t\in [0,T])] - \E^{\P^0}[\Phi(y_t,t\in [0,T])] | \leq 4\|\Phi\|_\infty \P^0(\eta_M \leq T)\,,
\] 
for any $M$.
Moreover, we have
\[
\P^0(\eta_M \leq T) \leq \P^0(\zeta_M \leq T) + \P^0(\zeta'_M \leq T)\,.
\]
where for any any $y=(y^1,y^2)\in\mathcal{C}([0,\infty),\R^2)$ we define
\[\zeta'_M(y) \defeq \inf\{t\geq 0: |y^2_t|\geq M\}\,.\]
By Proposition~\ref{p:etaInf} we know
\[
\P^0(\zeta_M \leq T) = \P^0_I(\zeta_M \leq T)\underset{M\to \infty}{\longrightarrow} 0\,.
\]

For the other term, let $0<\eta<1$ and write
\begin{align*}
  \P^0(\zeta'_M \leq T) \leq \P\Big(&\sup_{t\in[0,T]}|\psi_t| \geq M,  \inf_{t\in[0,T]}I_t>\eta, \sup_{t\in [0,T]} I_t < \frac{1}{\eta}\Big) \\
  &+\P\Big( \sup_{t\in [0,T]} I_t \geq \frac{1}{\eta} \Big) +\P\Big(\inf_{t\in[0,T]}I_t \leq \eta\Big)\,,  
\end{align*}
then using Markov inequality and \eqref{eq:psi}, we obtain the following for some $C_{\eta}>0$
\[\begin{split}
\P\Big(&\sup_{t\in[0,T]} |\psi_t| \geq M, \inf_{t\in[0,T]} I_t>\eta, \sup_{t\in [0,T]} I_t < \frac{1}{\eta}\Big) \\
& \leq \frac{1}{M^2} \E\Big[\sup_{t\in[0,T]}|\psi_t|^2 \mathbf{1}_{\{\inf_{t\in[0,T]} I_t >\eta,\, \sup_{t\in [0,T]} I_t < \frac{1}{\eta}\}}\Big]\\
&\leq \frac{2}{M^2} \E\Big[\sup_{t\in[0,T]}\Big| \int_{\tau=0}^1\int_{s=0}^t b(I_s,\tau) dW_s(I_s,\tau)d\tau\Big|^2 \mathbf{1}_{\{\inf_{t\in[0,T]} I_t >\eta,\,\sup_{t\in [0,T]} I_t < \frac{1}{\eta}\}}\Big]\\
& \quad+\frac{C_\eta}{M^2}\,.
\end{split}\]
Indeed, since we are on the event $\{\inf_{t\in[0,T]} I_t >\eta,\, \sup_{t\in [0,T]} I_t < \frac{1}{\eta}\}$ we avoid the singularity of the function $b$ at $I=0$ and keep $I$ bounded. Therefore, we can directly bound the drift term in \eqref{eq:psi} by bounding continuous functions $a,b, \partial_I a, \partial_I b$ by some constant $C_{\eta}$.
We will subsequently let $C_{\eta}$ be a constant that only depends on $\eta$ whose value may change from line to line. The details are as follows.

Write the Fourier expansions of~$a$ and~$b$ as
\[
	a(I,\tau) = \sum\limits_{n\in \mathbb Z^*}e^{2i \pi n\tau}a_n(I),\qquad b(I,\tau)=\sum\limits_{n\in \mathbb Z}e^{2i\pi n\tau}b_n(I)\,.
\]
Considering only the first integrand in the drift term in \eqref{eq:psi} and substituting the above expansions, we obtain
\begin{align}
  \MoveEqLeft
  \nonumber
	 \int_{u=0}^\infty R(u) \int_{\tau=0}^1 \partial_I \big(b(I,\tau + \omega(I)u)\big)_{|I=I_t}  a(I_t,\tau)\,d\tau\,du\\
  \nonumber
	&= \int_{u=0}^{\infty} R(u)\int_{\tau=0}^{1}\Big[ \sum\limits_{n\in\mathbb Z}\Big((\partial_I b)_n(I_t)+2i\pi n\omega'(I_t)u b_n(I_t) \Big)e^{2i\pi n(\tau+\omega(I_t)u)}\\
  \nonumber
	&\hspace{12em}\cdot\Big( \sum\limits_{n'\in \mathbb Z^*}a_{n'}(I_t)e^{2i\pi n'\tau}\Big)\Big]\,d\tau\,du
	\\
	\label{e:dsum1}
	& = \int_{u=0}^{\infty} R(u) \sum\limits_{n\in\mathbb Z^*}\Big((\partial_I b)_n(I_t)+2i\pi n\omega'(I_t)u b_n(I_t) \Big)a_{-n}(I_t)e^{2i\pi n\omega(I_t)u}\,du
	\\
	\label{e:dsumReal}
	& = R_1 + R_2\,,
\end{align}
where
\begin{align*}
  R_1 &\defeq \int_{u=0}^{\infty} R(u)\sum\limits_{n\in\mathbb Z^*}(\partial_Ib)_n(I_t)a_{-n}(I_t)\cos(2\pi n\omega(I_t)u) \, du\\
  R_2 &\defeq -\int_{u=0}^{\infty}R(u)\sum\limits_{n\in\mathbb Z^*}(\partial_Ib)_n(I_t)a_{-n}(I_t)\cdot 2\pi n\omega'(I_t)u\sin(2\pi n\omega(I_t)u)\,du\,.
\end{align*}
Here~\eqref{e:dsum1} followed by expanding the double sum in $n$ and $n'$, integrating in $\tau$, and observing that the only terms that do not vanish are those with $n+n'=0$.
The last equality~\eqref{e:dsumReal} followed since the expression is real-valued and so must equal its real part.

Using \eqref{ineq_R} to bound the integral in $u$ and bounding $(\partial_Ib)_n(I_t)$, $a_{-n}(I_t)$ by some constant $C_{\eta}$, we have that
\[
|R_1| \leq C_{\eta}\sum\limits_{n\in\mathbb Z^*}\frac{1}{n^2}\leq C_{\eta}\,.
\]
Similarly, we can estimate that
\begin{align*}
	\int_{u=0}^{\infty} &R(u) \cdot 2\pi nu \sin(2\pi n\omega(I_t)u)\,du\\&= 	\int_{u=0}^{\infty}\Big(\int_S r(p)e^{-\mu|p|^{2\beta}u} \, dp\Big)\cdot 2\pi nu \sin(2\pi n\omega(I_t)u)\,du\\
	&=\int_S \,dp\, r(p)\Big(\int_{u=0}^{\infty} 2\pi n u \sin(2\pi n \omega(I_t)u)e^{-\mu|p|^{2\beta}u}\,du  \Big)\\
	&=2\int_S\,dp\,r(p)\frac{(2\pi n)^2 \omega(I_t)\mu|p|^{2\beta}}{(\mu^2|p|^{4\beta}+4\pi^2 n^2\omega^2(I_t))^2}\le \frac{C_{\omega_0}}{n^2}\int_S |p|^{2\beta}r(p)\,dp\,.
\end{align*}
Therefore, we also have
\[
|R_2| \le C_{\eta}\sum\limits_{n\in\mathbb Z^*}\frac{1}{n^2}\le C_{\eta}\,.
\]
Combining these estimates we have
\begin{align*}
  \MoveEqLeft
  \Bigl|
    \int_{u=0}^\infty R(u) \int_{\tau=0}^1 \partial_I \big(b(I,\tau + \omega(I)u)\big)_{|I=I_t}  a(I_t,\tau)
    \\
  &\qquad\qquad+ b(I_t,\tau + \omega(I_t)u) \partial_I a(I_t,\tau) \,d\tau\,du\,\Bigr|
  \leq C_{\eta}\sum\limits_{n\in\mathbb Z^*}\frac{1}{n^2}\leq C_{\eta}\,.
\end{align*}

For the stochastic integral term, we introduce a smooth function $b^\eta$ such that for any $\tau\in (0,1)$ we have
\[
b^\eta(I,\tau) = b(I,\tau)\quad\text{if}\quad  I >\eta \qquad\text{and}\qquad b^\eta(I,\tau)=0\quad\text{if}\quad I<\frac{\eta}{2}\,.
\]
Then
\[\begin{split}
\E\Big[&\sup_{t\in[0,T]}\Big| \int_{\tau=0}^1\int_{s=0}^t b(I_s,\tau) dW_s(I_s,\tau)d\tau\Big|^2 \mathbf{1}_{\{\inf_{t\in[0,T]} I_t >\eta,\, \sup_{t\in [0,T]} I_t < \frac{1}{\eta}\}}\Big]\\
&=\E\Big[\sup_{t\in[0,T]}\Big|\int_{\tau=0}^1\int_{s=0}^t b^\eta(I_s,\tau) dW_s(I_s,\tau)d\tau\Big|^2 \mathbf{1}_{\{\inf_{t\in[0,T]} I_t >\eta,\, \sup_{t\in [0,T]} I_t < \frac{1}{\eta}\}}\Big]\\
&\leq C_\eta\,,
\end{split}\]
where the last inequality follows from the  Burkholder-Davis-Gundy inequality to bound the supremum of a continuous local martingale.
Finally, for any $0<\eta<1$ we have 
\[\limsup_{M\to \infty}\P^0(\zeta'_M \leq T) \leq \P\Big( \sup_{t\in [0,T]} I_t \geq \frac{1}{\eta} \Big)+ \P\Big(\inf_{t\in[0,T]}I_t \leq \eta\Big),\]
and
\begin{align*}
  \lim_{\eta\to 0}\P\Big(\inf_{t\in[0,T]}I_t \leq \eta\Big)&=\P\Big(\inf_{t\in[0,T]}I_t =0\Big)=0\,,\\
  \lim_{\eta\to 0}\P\Big( \sup_{t\in [0,T]} I_t \geq \frac{1}{\eta} \Big)&=0 \quad\quad (\text{by Lemma~\ref{bound_I}})\,.
\end{align*}
As a result, we obtain
\[
\lim_{\eps\to 0}\E^{\P^{\eps}}[\Phi(y_t,t\in [0,T])] = \E^{\P^0}[\Phi(y_t,t\in [0,T])]\,,
\]
which concludes the proof of Theorem \ref{mainresult}. 

It remains to prove Proposition~\ref{prop_main}, and we do this in the next section.

\section{Proof of Proposition \ref{prop_main}}\label{sec:propmain}

%
%

We will prove convergence of the family $(Y^{\eps})_{\eps>0}$ by first showing tightness, and then showing all subsequential limits are solutions of a well-posed martingale problem with generator \eqref{def_L}.
This follows the classical approach to prove weak convergence of a sequence of diffusions using martingale problem that is well developed in \cite{StroockVaradhan06}.
However, because the processes involved are not diffusions (and often not even Markovian), we use pseudo-generators and the perturbed test function method (see for instance~\cite{Kushner84}).

\subsection{Pseudo-generator and Perturbed test function method}  

In this section, we follow \cite[Chapter 3]{Kushner84}.
Even though the pair of processes $(Y_t^{\eps},\tau_{t}^{\eps})_{t\geq 0}$ is Markov for each $\eps>0$, the process $(Y_t^{\eps})_{t\geq 0}$ by itself is \emph{not} Markov.
Since analyzing the pair $(Y_t^{\eps},\tau_{t}^{\eps})_{t\geq 0}$ is difficult, we study the process $(Y_t^\epsilon)_{t \geq 0}$ on its own by using the pseudo-generator.
This allows us to apply martingale techniques to the non-Markovian process $(Y^\epsilon_t)_{t \geq 0}$.

\subsubsection{Pseudo-generator}\label{pseudogene} Let 
\begin{equation*}
\mathcal{G}^\eps _t\defeq\mathcal{G}_{t/\eps^2}\,,\end{equation*}
where $\mathcal{G}_{t}$ is defined by \eqref{def_flitr_v},
and $\mathcal{S}^\eps $ be the set of all measurable functions $h(t)$, adapted to the filtration $(\mathcal{G}^\eps _t)$, for which $\sup_{t\leq T} \E[\lvert h(t) \rvert ] <+\infty $, and where $T>0$ is fixed.  The $\plim$ and the pseudo-generator are defined as follows. Let $h$ and $h^\delta$ in $\mathcal{S}^\eps $ for all $\delta>0$. We say that $h=\plim_\delta h^\delta$ if
\begin{equation*}
\sup_{t, \delta }\E \lvert h^\delta(t)\rvert <+\infty\quad \text{and}\quad \lim_{\delta\rightarrow 0}\E \lvert h^\delta (t) -h(t)\rvert =0 \quad \forall t\geq 0\,.\end{equation*}
We say that $h\in \mathcal{D}(\mathcal{A}^\eps)$ the domain of $\mathcal{A}^\eps$ and $\mathcal{A}^\eps h=g$ if $h$ and $g$ are in $\mathcal{S}^\eps$ and 
\begin{equation*}
\plim_{\delta \to 0} \Bigl( \frac{\E^\eps _t  h(t+\delta)-h(t)}{\delta}-g(t) \Bigr)=0\,,
\end{equation*}
where $\E^\eps _t$ is the conditional expectation given $\mathcal{G}^\eps _t$. The difference with respect to the usual infinitesimal generator is that we have \textit{average} limit (via $\plim$) rather than pointwise limit. This greatly helps us average out the noise in the limit $\epsilon \to 0$. A useful result \cite[Theorem 1, p. 39]{Kushner84} about the pseudo-generator $\mathcal{A}^\eps$ that will be central in our proof is the following.
\begin{proposition}[Theorem 1, p.\ 30 in~\cite{Kushner84}]\label{martingale}
  Let $h\in \mathcal{D}(\mathcal{A}^\eps)$. Then
  \begin{equation*}
  M_h ^\eps (t)=h(t)-h(0)-\int _0 ^t  \mathcal{A}^\eps h(u)\,du
  \end{equation*}
  is a $( \mathcal{G}^\eps _t )$-martingale.
\end{proposition} 

In other words, we introduce an operator $\Aep$ on functions of $Y^{\epsilon}$ having similar properties to those of an infinitesimal operator of a Markov process; more specifically, it satisfies a martingale problem.

\subsubsection{Perturbed test function method}

The main idea behind the perturbed test function method is the following.
To characterize the limiting process
through a martingale problem with generator $\mathcal{A}$, we can then try to compare $\mathcal{A}^\epsilon h $ with $\mathcal{A}h$ for some test function $h$.
Unfortunately, $\mathcal{A}^\epsilon h$ may have singular terms in $\eps$, and we cannot proceed that way directly.
Instead, try and extract the effective statistical behavior of the system under consideration from $\mathcal{A}^\epsilon$ by introducing $h^{\eps}$ an appropriate perturbation of $h$.
In other words, if we can show $\lim_{\eps\to 0}\E |h^{\eps}(t)-h(t)|=0$ and $\lim_{\eps\to 0}\E |\Aep h^{\epsilon}(t)-\mathcal{A}h(Y_t^{\epsilon})|=0$ for each $t$, then we can use Proposition \ref{martingale} to conclude that the limiting process $(Y_t)_{t\geq 0}$ satisfies the martingale problem with generator $\mathcal{A}$.

To depict how the perturbation $h^\eps$ is built, let us briefly remind its construction in the simpler situation 
\[
\dot Y^\eps_t = \frac{1}{\eps} \mathbf{F}\Big(Y^\eps_t,\nu\Big(\frac{t}{\eps^2}\Big),\frac{t}{\eps^2}\Big)\qquad\text{and}\qquad \mathbf{F}(y,v,\tau) = v\, F(y,\tau)\,,
\] 
where $\nu$ is a centered ergodic Markov process with invariant measure $\mu$, and generator $\mathcal L_\nu$ satisfying the Fredholm alternative. In this situation, to apply the perturbed test function method, we consider the extended Markov process $Z^\eps_t=(Y^\eps_t, \nu(t/\eps^2), \tau_{t/\eps^2})$ whose infinitesimal generator is given by
\[
\mathcal L^\eps = \frac{1}{\eps^2}Q + \frac{1}{\eps} \, \mathbf{F}(y,v,\tau)\cdot \nabla_y\qquad\text{with}\qquad Q = \Big(\mathcal L_\nu + \frac{\partial}{\partial \tau}\Big)\,.
\]  
Looking for a perturbation $h^\eps$ of $h$ of the form
\[
h^\eps(y,v,\tau) = h(y)+ \eps h_1(y,v,\tau) + \eps^2 h_2(y,v,\tau)\,,
\]
we have
\begin{align}\label{eq:sketch}
	\mathcal L^\eps h^\eps(y,v,\tau) & = \frac{1}{\eps}\big( Q h_1(y,v,\tau) + \mathbf{F}(y,v,\tau)\cdot\nabla_y h(y)\big) \\
	& + Q h_2(y,v,\tau) + \mathbf{F}(y,v,\tau) \cdot\nabla_y h_1(y,v,\tau) \nonumber\\
	& + \mathcal O(\eps)\,.\nonumber
\end{align}
To cancel the singular term we have to solve a Poisson problem of the form
\[
Q p = - q(y,v,\tau)\,.
\] 
Assuming the centering condition
\[
\int_0^1 \E_\mu [ q(y,\nu(0),\tau) ]\,d\tau = 0\,,
\] 
where the expectation is taken with respect to the invariant measure $\mu$ of the process $\nu$, the Poisson problem admits
\[p = \int_0^\infty e^{uQ}q \, du\]
as a solution, so that
\[
p(Z^\eps_t)  = \int_0^\infty  \E \Big[ q\Big( Y^\eps_{t}, \nu\Big(u + \frac{t}{\eps^2}\Big),u+\frac{t}{\eps^2}\Big) \,\Big\vert \,  \nu(s/\eps^2),\,s\leq t\Big]\, du\,,
\]
from the Markov property. Remembering that $\nu$ is centered and $F$ is mean zero with respect to $\tau$, we then have
\begin{equation}\label{eq:sketch_h1}
	h_1(Z^\eps_t)  = \int_0^\infty  \E \Big[ \mathbf{F}\Big( Y^\eps_{t}, \nu\Big(u + \frac{t}{\eps^2}\Big),u+\frac{t}{\eps^2}\Big)\cdot\nabla_y h ( Y^\eps_{t}) \,\Big\vert \,  \nu(s/\eps^2),\,s\leq t\Big]\, du\,.
\end{equation}
To treat the term of order 1 in \eqref{eq:sketch}, which still has fast oscillations through $h_1$, we solve again the Poisson problem but this time with
\begin{equation}\label{eq:sketch_q}
	q(y,v,\tau) = \mathbf{F}(y,v,\tau) \cdot\nabla_y h_1(y,v,\tau) - \int_0^1\E_\mu[\mathbf{F}(y,\nu(0),\tau') \cdot\nabla_y h_1(y,\nu(0),\tau')]d\tau'\,,
\end{equation}
which satisfies the above centering condition. As a result, we can chose
\begin{align}\label{eq:sketch_h2}
	h_2(Z^\eps_t) & = \int_0^\infty  \E \Big[ (\mathbf{F}\cdot\nabla_y h_1)\Big( Y^\eps_{t}, \nu\Big(u + \frac{t}{\eps^2}\Big),u+\frac{t}{\eps^2}\Big)  \,\Big\vert \,  \nu(s/\eps^2),\,s\leq t\Big] du \\
	&-\int_0^1\E_\mu[\mathbf{F}(Y^\eps_{t},\nu(0),\tau') \cdot\nabla_y h_1(Y^\eps_{t},\nu(0),\tau')]\,d\tau'\,.\nonumber
\end{align}
Finally, due to the centering in \eqref{eq:sketch_q}, the expansion \eqref{eq:sketch} becomes
\begin{align*}
	\mathcal L^\eps h^\eps(Z^\eps_t) & = \int_0^1\int_0^\infty \E_\mu[\mathbf{F}(Y^\eps_{t},\nu(0),\tau) \cdot\nabla_y (\mathbf{F}(Y^\eps_{t},\nu(u),\tau+u)\cdot\nabla_y h(Y^\eps_{t}))]\,du\, d\tau \nonumber\\
	& + \mathcal O(\eps)\,,\nonumber
\end{align*}
and we obtain the same expression as in \eqref{def_L}.

However, considering non-Markovian fluctuations, as well as long-range correlations, and a more complex periodic structure, make the construction of $h^\eps$ more delicate. The main concern is to handle the non-integrability at infinity of the correlation function of $V$. To overcome this difficulty we must take advantage of the periodic structure as in the proof of Proposition \ref{p:qConv}. The details are in the next section where we will construct $h^\eps$ to show tightness of $(Y^{\eps})_{\eps>0}$, which implies imply weak subsequential convergence of stochastic processes.

Before going into the proofs, we remind the reader that $(Y^{\eps})_{\eps>0}$ is defined via a truncated process in equation \eqref{e:defEqDiff} and there is an omitted index $M$ in the functions $F$ and $G$ to indicate the truncation. This allows us to bound several expressions in the proofs in terms of $M$.
\subsection{Tightness}\label{tightnesssec}

In this section, we prove the tightness of the family $(Y^{\eps})_{\eps>0 }$, seen as a family of continuous-time processes. Then, according to \cite[Theorem 13.4]{Billingsley99} it suffices to prove the tightness of  $(Y^{\eps})_{\eps>0 }$ in $\mathcal{D}([0,T],\R^2)$ (which is the set of c\`adl\`ag functions with values in $\R^2$ equipped with the Skorokhod topology).

\begin{proposition}\label{tightness}
  The family $(Y^{\eps})_{\eps>0}$  is tight in $\mathcal{D}([0,T],\R^2 )$. 
\end{proposition}
The proof of Proposition \ref{tightness} consists of applying \cite[Theorem 4, p. 48]{Kushner84}. In what follows, let $h$ be a bounded smooth function on $\R^2$ with bounded derivatives, and set $h_0 ^\eps (t)=h(Y^{\eps}_t)$. Our goal is to construct perturbations to $h_0^{\eps}$ as described above.  The pseudo-generator at $h_0^\eps$ is then given by
\begin{equation}\label{Aef0}
\mathcal{A}^\eps h_0 ^\eps (t)= \frac{1}{\eps}\int V\Big(\frac{t}{\eps^2},dp\Big) F\big(Y^{\eps}_t,\tau^{\eps}_t\big)\cdot \nabla h(Y^{\eps}_t)
\end{equation}
which is simply differentiating $h_0^{\eps}(t)$ with respect to $t$ (the derivative $\dot Y^{\eps}_t$ is stated in equation \eqref{e:defEqDiff}). The goal now is to modify the test function $h^\eps_0$ using a small perturbation $h^\eps_1$ so that the pseudo-generator $\mathcal{A}^\eps(h^\eps _0 +h^\eps _1)$ does not blow up in $\eps$ anymore. The first perturbation of $h^\eps_0$ is defined as follows.
\begin{align}\label{eq:h1}
  \nonumber
  \MoveEqLeft
  h^\eps _1 (t) = \frac{1}{\eps}\int_{t}^{\infty} du \int \E^\eps_t \Big[V\Big(\frac{u}{\eps^2},dp\Big) F\Big(Y^{\eps}_t,\tau^{\eps}_t + \frac{u-t}{\eps^2}\, \omega(Y^{\eps}_t)\Big)\Big]\cdot \nabla h(Y^{\eps}_t)\\
  \nonumber
  &=\eps\int_{0}^{\infty} du\int \E^\eps_t \Big[V\Big(u+\frac{t}{\eps^2},dp\Big) F\Big(Y^{\eps}_t,\tau^{\eps}_t + u\, \omega(Y^{\eps}_t)\Big)\Big]\cdot \nabla h(Y^{\eps}_t) \\
  &=\eps\int_{0}^{\infty} du\int \sum_{n\in\Z^*} \E^\eps_t \Big[V\Big(u+\frac{t}{\eps^2},dp\Big)\Big] e^{2i\pi n(\tau^{\eps}_t + u\,\omega(Y^{\eps}_t))} F_n(Y^{\eps}_t)\cdot \nabla h(Y^{\eps}_t),
\end{align}
where
\[F_n(y)\defeq \int_0^1  F(y,\tau)e^{-2i\pi n\tau}\, d\tau\,.\]

According to Lemma~\ref{l:mar}, we have
\begin{equation*}
\E^\eps_t V\Big(u+\frac{t}{\eps^2},dp\Big) = e^{-\mu |p|^{2\beta} u } V\Big(\frac{t}{\eps^2},dp\Big),\end{equation*}
so that
\begin{equation*}
\begin{split}
h^\eps _1 (t) &=\eps \sum_{n\in\Z^*} e^{2i\pi n \tau^{\eps}_t} \int V\Big(\frac{t}{\eps^2},dp\Big)      F_n(Y^{\eps}_t)\cdot \nabla h(Y^{\eps}_t)   \int_{0}^\infty  e^{-\mu |p|^{2\beta} u } e^{2i\pi n \omega(Y^{\eps}_t) u}\, du\\
&=\eps \sum_{n\in\Z^*} e^{2i\pi n \tau^{\eps}_t} \int \frac{V(t/\eps^2,dp)}{\mu \vert p \vert^{2\beta}-2 i\pi n\, \omega(Y^{\eps}_t)}      F_n(Y^{\eps}_t)\cdot \nabla h(Y^{\eps}_t)
\end{split}\end{equation*}

To apply \cite[Theorem 4 p. 48]{Kushner84} and then prove Proposition \ref{tightness}, we only have to prove the two following lemmas. 
\begin{lemma}\label{bound2}  For any $T>0$, and $\eta>0$
  \begin{equation*}
  \lim_{\eps\to 0} \P\Big(\sup_{t \in[0, T]} \lvert h^\eps _1 (t)\rvert>\eta\Big) =0\qquad \text{and} \qquad \lim_{\eps\to 0} \sup_{t \in[0, T]}\E\lvert h^\eps _1 (t) \rvert =0\,.
  \end{equation*}
\end{lemma}
\begin{lemma}\label{A1}
  For any $T>0$, $\big\{\mathcal{A}^\eps \big(h^\eps _0 +h^\eps _1\big)(t), \eps \in(0,1), 0\leq t\leq T\big\}$ is uniformly integrable.
\end{lemma}

Before going through the proofs of these two lemmas, we make one remark
that the test function $h^\eps _1$, above, is well defined since we sum over $n\neq 0$.
This is because $F$ has mean zero respect to the $\tau$, and explains why this mean-zero condition is crucial to the analysis.
For $n=0$, the above function $h^\eps _1$ would not be defined in the case of long-range correlations. 

\begin{proof}[Proof of Lemma \ref{bound2}.] 
  We rewrite $h_1^\eps$ as
  \begin{equation*}h_1^\eps(t)=\eps V\Big(\frac{t}{\eps^2},\varphi_{t,\eps}\Big),\end{equation*}
  with
  \begin{equation}\label{eq:phi}\varphi_{t,\eps}(p)\defeq \sum_{n\in\Z^*} e^{2i\pi n \tau^{\eps}_t}\frac{1}{\mu \vert p \vert^{2\beta}-2 i\pi n\, \omega(Y^{\eps}_t)}      F_n(Y^{\eps}_t)\cdot \nabla h(Y^{\eps}_t)\,.\end{equation}
  By the Cauchy-Schwarz inequality
  \begin{equation*}
  \begin{split}
  |\varphi_{t,\eps}(p)| &\leq C_{h,\omega_0} \sum_{n\in\Z^*}\frac{1}{|n|} \| F_n(Y^{\eps}_t)\|\\
  &\leq C_{h,\omega_0} \Big(\sum_{n\in\Z^*}\frac{1}{|n|^2}\Big)^{1/2} \Big(\sum_{n\in\Z^*}\| F_n(Y_t^{\eps})\|^2\Big)^{1/2}\\
  &\leq C_{h,\omega_0} \Big(\sum_{n\in\Z^*}\frac{1}{|n|^2}\Big)^{1/2}\Big(\int_0^1 \| F(Y^{\eps}_t,\tau)\|^2\,d\tau\Big)^{1/2}\\
  & \leq C_{h,\omega_0} \sup_{\substack{|y|\leq 2M }} \Big(\int_0^1 \| F(y,\tau)\|^2\,d\tau\Big)^{1/2}. 
  \end{split}\end{equation*}
  
  If $\beta \geq 1/2$, we have the Lipschitz bound
  \begin{equation*}\begin{split}
  |\varphi_{t,\eps}(p)-\varphi_{t,\eps}(q)| &\leq (|p|^{2\beta}-|q|^{2\beta}) C_{h,\omega_0,F} \sum_{n\in\Z^*}\frac{1}{|n|^2}\\
  & \leq| |p|-|q| | \sup_{r\in S} |r|^{2\beta-1} C_{h,\omega_0,F} \sum_{n\in\Z^*}\frac{1}{|n|^2} \leq L |p-q|
  \end{split}\end{equation*}
  with
  \begin{equation*}
    L=  \sup\limits_{r\in S}|r|^{2\beta-1} C_{h,\omega_0,F} \sum_{n\in\Z^*}\frac{1}{|n|^2} <\infty \,.
  \end{equation*}
  Using inequality \eqref{boundV1} of Lemma~\ref{l:bound} with $k=\infty$, along with the Chebychev's inequality, we have
  \begin{equation*}
  \P\Big(\sup_{ t \in[0, T]} \lvert h^\eps _1 (t)\rvert>\eta\Big) \leq \frac{1}{\eta} \E\sup_{t \in[0, T]} \lvert h^\eps _1 (t)\rvert \leq C'_{h,\omega_0,F}\, C(\eps) \,,
  \end{equation*}
  with $\lim_\eps C(\eps)=0$. 
  
  If $\beta < 1/2$, the function $\varphi_{t,\eps}$ belongs to $W^{1,k}(S)$ with $k\in (1,1/(1-2\beta))$ since
  \begin{equation*}
  \int_S |\partial_p\varphi_{t,\eps}(p)|^k\, dp \leq  C''_{h,\omega_0,F} \int_S |p|^{k(2\beta-1)}\,dr<\infty\,,
  \end{equation*}
  so that using inequality \eqref{boundV1} of Lemma~\ref{l:bound} but with $k\in (1,1/(1-2\beta))$ leads again to 
  \begin{equation*}
  \P\Big(\sup_{t \in[0, T]} \lvert h^\eps _1 (t)\rvert>\eta\Big) \leq C'_{h,\omega_0,F}\, C(\eps)\,.
  \end{equation*}
  
  To prove the second limit in the lemma, one can just remark that
  \begin{equation*}
  \E\lvert h^\eps _1 (t) \rvert  \leq \E\sup_{t \in[0, T]} \lvert h^\eps _1 (t)\rvert  
  \end{equation*} 
  for any $t\in[0,T]$, and then concludes the proof of Lemma \ref{bound2}.
\end{proof}

\begin{proof}[Proof of Lemma \ref{A1}.] For proving uniform integrability, it is sufficient to show that 
  \begin{equation*}
  \sup_{\eps,t}\E\vert \mathcal{A}^\eps(h^\eps_0+h^\eps_1)(t)\vert^2\leq C\,.
  \end{equation*}
  Following the definition of pseudo-generator in Section \ref{pseudogene}, we have that
  \begin{align*}
    \mathcal{A}^{\eps}&(h_1^{\eps})(t) = -\mathcal{A}^{\eps}(h_0^{\eps})(t)\\&+\frac{1}{\eps}\int_{t}^{\infty} du \int \E^\eps_t \Big[V\Big(\frac{u}{\eps^2},dp\Big)\Big]\frac{d}{ds} \Big[F\Big(Y^{\eps}_s,\tau^{\eps}_s + \frac{u-s}{\eps^2}\, \omega(Y^{\eps}_s)\Big)\cdot \nabla h(Y^{\eps}_s)\Big]\Big|_{s=t}
  \end{align*}
  which is essentially differentiating $h^{\eps}_1(t)$ with respect to $t$ (we however remark that we do not differentiate the $t$ in $\E_t^\eps)$. After some lengthy but straightforward computations (similarly to those used to obtain~\eqref{eq:h1}), we obtain
  \begin{align}\label{def_Af1}
    \nonumber
    \MoveEqLeft
  \mathcal{A}^\eps(h^\eps_1)(t)=-\mathcal{A}^\eps(h^\eps_0)(t)+\sum_{n,m\in\Z^*} e^{2i\pi (n+m) \tau^{\eps}_t} \iint \frac{V(t/\eps^2,dp)V(t/\eps^2,dp')}{\mu |p|^{2\beta} - 2i\pi n\omega(Y^{\eps}_t)}\\
  \nonumber
  &\qquad\qquad\cdot \Big(  F_{n}(Y^{\eps}_t) ^T \nabla^2 h(Y^{\eps}_t)F_{m}(Y^{\eps}_t) +\nabla h(Y^{\eps}_t)^T \textrm{Jac} F_{n}(Y^{\eps}_t)F_{m}(Y^{\eps}_t) \\
  \nonumber
  &\qquad\qquad\qquad-\frac{2i\pi n}{\mu|p|^{2\beta}-2i\pi n \omega(Y^{\eps}_t)} F_{n}(Y^{\eps}_t)\cdot \nabla h(Y^{\eps}_t)F_{m}(Y^{\eps}_t)\cdot \nabla \omega(Y^{\eps}_t)\Big) \\
  \nonumber
  &\qquad+\sum_{n\in\Z^*} 2i\pi n \, e^{2i\pi n \tau^{\eps}_t} \iint \frac{V(t/\eps^2,dp)V(t/\eps^2,dp')}{\mu |p|^{2\beta} - 2i\pi n\omega(Y^{\eps}_t)} \\
  \nonumber
  &\qquad\qquad\cdot F_{n}(Y^{\eps}_t)\cdot \nabla h(Y^{\eps}_t) G(Y^{\eps}_t)\\
  \nonumber
  &=-\mathcal{A}^\eps(h^\eps_0)(t) + V\Big(\frac{t}{\eps^2},\varphi^1_{t,\eps}\Big)V\Big(\frac{t}{\eps^2},\mathbf{1}_S\Big) +V\Big(\frac{t}{\eps^2},\varphi^2_{t,\eps}\Big)V\Big(\frac{t}{\eps^2},\mathbf{1}_S\Big)\\
  &\defeq -\mathcal{A}^\eps(h^\eps_0)(t) +\mathcal{P}_1^\eps(t)+\mathcal{P}_2^\eps(t)
  \end{align}
  where $\varphi_{t,\eps}^j$, $j=1,2$ are defined similarly as in \eqref{eq:phi}. As in the proof of Lemma \ref{bound2}, it is not hard to see that the $\varphi^j_{t,\eps}$, $j=1,2$ are Lispschitz in $p$ on $S$, and that
  \begin{equation*}
  \begin{split}
  |\varphi^1_{t,\eps}(p)| &\leq C_{h,\omega_0,\nabla \omega} \sum_{m\in\Z^*} |F_{m}(Y^{\eps}_t)|  \sum_{n\in\Z^*} \frac{1}{|n|}\Big( (1+|n|) |F_{n}(Y^{\eps}_t)| + \|\textrm{Jac}F_{n}(Y^{\eps}_t)\| \Big)\\
  &\leq  \tilde{C}_{h,\omega_0,\nabla \omega}  \sup_{\substack{|y|\leq 2M }} \int_0^1\Big(  | F(y,\tau)|^2 + |\partial_\tau F(y,\tau)|^2 +\|\textrm{Jac}F(y,\tau)\|^2 \Big)\,d\tau\,.
  \end{split}\end{equation*}
  The same lines give also
  \begin{equation*}
  |\varphi^2_{t,\eps}(p)| \leq C_{h,\omega_0}  \sup_{\substack{|y|\leq 2M}} \big|G(y)\big|\Big(\int_0^1  |\partial_\tau F(y,\tau)|^2\,d\tau\Big)^{1/2}.
  \end{equation*}
  Then, we can use inequality \eqref{boundV2} of Lemma~\ref{l:bound}, and obtain
  \begin{equation*}\sup_{\eps,t} \E[|\mathcal{P}_j^\eps(t)|^2]\leq C,\quad j=1,2\		,.\end{equation*}  
  which concludes the proof of Lemma \ref{A1}, and then Proposition \ref{tightness} as well. 
\end{proof}

\subsection{Identification of the limit}

In this section, we identify all the limit points of $(Y^{\eps})_{\eps>0}$ via a well-posed martingale problem as stated in the following proposition. By abuse of notations, we still denote by $(Y^{\eps})_{\eps>0}$ a converging subsequence and by $(Y_t)_{t\geq 0}$ a limit point. 

\begin{proposition}\label{identification}
  All the limit points $(Y_t)_{t\geq 0}$ of $(Y^{\eps})_{\eps>0}$ are solutions of a well-posed martingale problem with generator 
  \begin{equation}\label{prop_LM}
  \mathcal{L} h(y) = \int_0^1 d\tau \int_0^\infty du \, R(u)\tilde F(0,y,\tau)\cdot \nabla_y ( \tilde F(u,y,\tau)\cdot \nabla_y h(y))
  \end{equation}
  with
  \begin{equation*}
  \tilde F(u,y,\tau) =  F(y,\tau +\omega(y)u)\,.\end{equation*}
\end{proposition}

  To prove this proposition we use the notion of pseudo-generator introduced in Section \ref{pseudogene} and the perturbed-test-function technique that we have already used in Section \ref{tightnesssec} for the proof of tightness. Thanks to the pseudo-generator we can characterize the subsequential limits of $(Y^{\eps})_{\eps>0}$ as solutions of a well-posed martingale problem. 
  
  The outline of the proof is as follows.
  Recall in Section~\ref{tightnesssec} we saw that $\mathcal{A}^\eps(h^\eps _0)$ has a singular $1/\eps$ term, and modified the test function $h^\eps _0$ with a small perturbation $h^\eps _1$ to remove this singular term.
  As a result, $\mathcal{A}^\eps(h^\eps _0 +h^\eps _1)$ is not singular any more.
  However, it is not yet the generator of a martingale problem as $\mathcal{A}^\eps(h^\eps _0 +h^\eps _1)$ still has oscillatory terms in the form of $e^{2i\pi n\tau_t^{\eps}}$.
  We will show that these terms essentially vanish as $\eps \to 0$ by introducing a small perturbation $h_2^{\eps}(t)$ to cancel these oscillations, and then construct another perturbation $h_3^{\eps}(t)$ to cancel oscillations that result from computing $\mathcal{A}^{\eps}(h_2^{\eps})(t)$.
  After we have constructed the perturbed test function 
  \[
    h^{\eps}(t) = h_0^\eps(t)+h_1^\eps(t)+h_2^\eps(t)+h^\eps_3(t)\,,
  \]
  we show that the pseudo-generator
  \[
     \mathcal{A}^\eps(h_0^\eps+h_1^\eps+h_2^\eps+h^\eps_3)(t)
  \]
  has the desired form of a generator related to a martingale problem, plus a negligible term. Combining with tightness, we can show that the limiting process $(Y_t)_{t\geq 0}$ satisfies a martingale problem.

  We now carry out this approach.
  For convenience, the proofs of needed auxiliary lemmas will be deferred to Sections \ref{sec:lemmaf2}-\ref{sec:rewrite_gen}.
 \begin{proof}[Proof of Proposition \ref{identification}]
  We split $\mathcal{A}^\eps(h^\eps _0 +h^\eps _1)$ into three parts
  \begin{equation*}
  \mathcal{A}^\eps(h^\eps _0 +h^\eps _1)(t)=\mathcal{A}^\eps_1(t)+\mathcal{A}^\eps_2(t)+\mathcal{P}_2^\eps(t)\,,
  \end{equation*} 
  where $\mathcal{P}_2^\eps$ is defined in \eqref{def_Af1}, and $\mathcal A^\epsilon_1(t)$, $\mathcal A^\epsilon_2(t)$ are given by
  \begin{gather}
    \nonumber
  \mathcal{A}^\eps_1(t)=\sum_{\substack{n,m\in\Z^*\\ n+m\neq 0}} e^{2i\pi (n+m) \tau^\eps_t} \iint V\Big(\frac{t}{\eps^2},dp\Big)V\Big(\frac{t}{\eps^2},dp'\Big) W_{n,m}(p, Y^{\eps}_t)\,,
  \\
  \label{eq:A2}
  \mathcal{A}^\eps_2(t)=\sum_{n\in\Z^*}\iint V\Big(\frac{t}{\eps^2},dp\Big)V\Big(\frac{t}{\eps^2},dp'\Big)W_{n}(p, Y^{\eps}_t)\,.
  \end{gather}
  Here
  \begin{align}\label{def_W}
    \nonumber
  W_{n,m}(p, Y^{\eps}_t)&\defeq \frac{1}{\mu |p|^{2\beta} - 2i\pi n\omega(Y^{\eps}_t)}\Big(  F_{n}(Y^{\eps}_t) ^T \nabla^2 h(Y^{\eps}_t)F_{m}(Y^{\eps}_t) \\
  \nonumber
  &\qquad\qquad+\nabla h(Y^{\eps}_t)^T \textrm{Jac} F_{n}(Y^{\eps}_t)F_{m}(Y^{\eps}_t)\Big) \\
  &\qquad-\frac{2i\pi n}{(\mu|p|^{2\beta}-2i\pi n \omega(Y^{\eps}_t))^2} F_{n}(Y^{\eps}_t)\cdot \nabla h(Y^{\eps}_t)F_{m}(Y^{\eps}_t)\cdot \nabla \omega(Y^{\eps}_t)\,,
  \end{align}
  and
  \begin{equation*}
  W_{n}(p, Y^{\eps}_t) = W_{n,-n}(p, Y^{\eps}_t)\,.
  \end{equation*}
  
  The term $\mathcal{A}^\eps_2$ no longer has oscillatory components as $\mathcal{A}_2^{\eps}$ only picks up the terms with $n+m=0$.
  We now show that this term is related to the generator of a martingale problem.
    \begin{lemma}\label{lemma_L}
    For all $t\geq 0$, we have
    \begin{equation*}
    \lim_{\eps \to 0}\E\Bigl\vert \int_0^t (\mathcal{A}^\eps_2(u)-\mathcal{L} h(Y^{\eps}_u)) \,du \Bigr\vert=0\,.
    \end{equation*}
    with
    \begin{equation*}
    \mathcal{L} h(y)= \sum_{n\in \Z^*}
    \int dp\, r(p)W_{n}(p,y)\,,\end{equation*}
    where $W_{n}$ is defined through \eqref{def_W}.
  \end{lemma}  
  We present the proof of Lemma~\ref{lemma_L} in Section~\ref{sec:lemma_L}.
   To show the well-posedness of the limiting martingale problem, we need to rewrite $\mathcal{L}$ as a second-order differential operator, which eventually gives us the form in \eqref{prop_LM}. We do this in Section \ref{sec:rewrite_gen}, below.
   
   For the oscillatory term $\mathcal{A}^\eps_1$, we introduce two more small perturbations $h^\eps _2$ and $h^\eps _3$ to prove that the oscillations will vanish as $\eps \to0$. Note that $\mathcal{P}_2^{\eps}$ also has oscillatory terms.
   These can be treated in a manner similar to the oscillatory terms in $\mathcal{A}^\eps_1$, so that these terms vanish in the limit $\eps\to 0$.
   Since the treatment of this will be similar to the treatment of $\mathcal{A}^\eps_1$, we omit the details.
  
  Similarly to the construction of $h_1^{\eps}$ in equation \eqref{eq:h1}, we introduce the second perturbation
  \begin{equation*}
  \begin{split}
  h^\eps _2 &(t)   = \int_t^\infty du \sum_{\substack{n,m\in\Z^* \\ n+m\neq 0}} e^{2i\pi (n+m)(\tau^{\eps}_t+ (u-t)\omega(Y^{\eps}_t)/\eps^2)}\\
  &\cdot \iint\Big( \E^\eps_t \Big[V\Big(\frac{u}{\eps^2},dp\Big)V\Big(\frac{u}{\eps^2},dp'\Big)\Big] - \E[V(0,dp)V(0,dp')]\Big)W_{n,m}(p, Y^{\eps}_t)\,.
  \end{split}\end{equation*}
  Making the change of variable $u \to t+\eps^2 u$, we have
  \begin{equation*}
  \begin{split}
  h^\eps _2 & (t)   = \eps^2 \int_0^\infty du \sum_{\substack{n,m\in\Z^* \\ n+m\neq 0}} e^{2i\pi (n+m) (\tau^{\eps}_t+ u\omega(Y^{\eps}_t))}W_{n,m}(p, Y^{\eps}_t)\\
  &\cdot \iint \Big(\E^\eps_t \Big[ V\Big(u+\frac{t}{\eps^2},dp\Big)V\Big(u+\frac{t}{\eps^2},dp'\Big)\Big]-\E[V(0,dp)V(0,dp')]\Big)\,.
  \end{split}
  \end{equation*} 
  According to formula \eqref{markovvar}, one has
  \begin{equation*}
  \begin{split}
  \E^\eps_t \Big[ V & \Big(u+\frac{t}{\eps^2},dp\Big)V\Big(u+\frac{t}{\eps^2},dp'\Big)\Big]-\E[V(0,dp)V(0,dp')]\\
  &=e^{-\mu(\vert p \vert^{2\beta}+\vert p' \vert^{2\beta}) u} \Big( V\Big(\frac{t}{\eps^2},dp\Big) V\Big(\frac{t}{\eps^2},dp'\Big)-r(p)\delta(p-p')\,dp\,dp'\Big)\,,
  \end{split}\end{equation*} 
  so that
  \begin{equation*} 
  \begin{split}
  h^\eps _2 (t)&= \eps^2 \sum_{\substack{n,m\in\Z^* \\ n+m\neq 0}} e^{2i\pi (n+m) \tau^{\eps}_t} \\
  &\times \iint \frac{V(t/\eps^2,dp) V(t/\eps^2,dp')-r(p)\delta(p-p')\,dp\,dp'}{\mu (|p|^{2\beta}+|p'|^{2\beta}) - 2i\pi (n+m)\omega(Y^{\eps}_t)}W_{n,m}(p, Y^{\eps}_t)\,.
  \end{split}\end{equation*}
 We claim that $h^{\eps}_2$ is uniformly small.
   \begin{lemma}\label{lemmaf2} 
    We have 
    \begin{equation*}
    \lim_{\eps\to 0}\sup_{t\in[0,T]} \E \vert h^\eps_2(t)\vert =0\,.
    \end{equation*}
  \end{lemma}
 For clarity of presentation we postpone the proof lf Lemma~\ref{lemmaf2} to Section \ref{sec:lemmaf2}.
  As before, following the definition of pseudo-generator in Section \ref{pseudogene}, we have
  \begin{equation*}
  \mathcal{A}^\eps(h_2^\eps)(t)=-\mathcal{A}^\eps_1(t)+ \mathcal{D}_\eps(t) +\mathcal{R}^1_\eps(t)\,,
  \end{equation*}
  with
  \begin{equation*}
  \begin{split}
  \mathcal{D}_\eps(t)& \defeq \sum_{\substack{n,m\in\Z^*\\ n+m\neq 0}} e^{2i\pi (n+m) \tau_t^{\eps}} \int dp\, r(p)W_{n,m}(p, Y^{\eps}_t) \,,
  \end{split}\end{equation*}
  and~$\mathcal R_\epsilon^1(t)$ to be the remainder (an explicit formula for $\mathcal{R}^1_{\eps}(t)$ is presented in to Section \ref{sec:lemmaReps}, below).
  Even though oscillatory terms are present in~$\mathcal R^1_\epsilon$, we make the following claim.
  \begin{lemma}\label{lemmaReps}
    We have 
    \begin{equation*}
    \lim\limits_{\eps\to 0}\sup_{t\in[0,T]}\E\vert \mathcal{R}^1_\eps (t)\vert =0\,.
    \end{equation*}
  \end{lemma}
  The proof of this lemma is presented in Section~\ref{sec:lemmaReps}, below.
  To treat the oscillations in $ \mathcal{D}_\eps(t)$  we introduce the third perturbation 
  \begin{align*}
  h^\eps_3(t)&= \int_t^{\infty}du\sum_{\substack{n,m\in\Z^*\\ n+m\neq 0}} e^{2i\pi (n+m) (\tau_t^{\eps}+(u-t)\omega(Y_t^{\eps})/\eps^2)}e^{-\eps (u-t)} \int dp\, r(p)W_{n,m}(p, Y^{\eps}_t)\\
  &=\e^2\sum_{\substack{n,m\in\Z^*\\ n+m\neq 0}} \frac{e^{2i\pi(n+m)\tau_t^{\eps}}}{\e^3-2i\pi(n+m)\omega(Y_t^{\eps})} \int dp\, r(p)W_{n,m}(p, Y^{\eps}_t)
  \end{align*}
  We will see that
  \begin{equation*}
  \mathcal{A}^\eps(h^\eps_3)(t)=-\mathcal{D}_\eps(t) +\mathcal{R}^2_{\eps}(t)\end{equation*}
  with explicit formula for $\mathcal{R}^2_{\eps}(t)$ presented in Section \ref{sec:boundf3}, below.
  We claim $h^\epsilon_3$ and $R^2_\epsilon$ both vanish as~$\epsilon \to 0$.
  \begin{lemma}\label{boundf3} We have
    \begin{equation*}
    \lim_{\eps\to 0}\sup_{t\in[0,T]}\E|h^\eps_3(t)|=0\qquad\text{and}\qquad     \lim\limits_{\eps\to 0}\sup_{t\in[0,T]}\E\vert \mathcal{R}^2_\eps (t)\vert =0\,.
    \end{equation*}
  \end{lemma}
  We will prove Lemma~\ref{boundf3} in Section~\ref{sec:boundf3}, below.
   As a result, we now have a perturbed test function
   \begin{equation*}
   h(t)=h_0^\eps(t)+h_1^\eps(t)+h_2^\eps(t)+h^\eps_3(t)\,,\end{equation*}
   with $h^\eps_0(t) = h(Y^{\eps}_t)$, for which
   \begin{equation}\label{estimA}
   \mathcal{A}^\eps(h_0^\eps+h_1^\eps+h_2^\eps+h^\eps_3)(t)=\mathcal{A}^\eps_2(t)+o(\eps)\,.
   \end{equation}
   Using Theorem \ref{martingale} we see that for all $N\geq 1$, any bounded continuous function $\Psi$, and every sequence $0< s_1<\cdots <s_N \leq s <t$ we must have
   \begin{equation}\label{MGN}
   \E\Big[ \Psi\big( Y^{\eps}_{s_1},\dots,Y^{\eps}_{s_N} \big)\Big( M^\eps_h(t)-M^\eps_h(s)\Big) \Big]=0\,.
   \end{equation}
   Together with Lemmas \ref{bound2}, \ref{lemma_L}, \ref{lemmaf2}, \ref{lemmaReps}, \ref{boundf3}, and estimate \eqref{estimA}, we have
   \begin{equation}\label{estimMG}
   \lim_{\eps\to 0}\E\Big[ \Psi\big( Y^{\eps}_{s_1},\dots,Y^{\eps}_{s_N} \big)\Big( h(Y^{\eps}_t)-h(Y^{\eps}_{s})-\int_{s}^t \mathcal{L}h(u)\,du\Big) \Big]=0 \,,
   \end{equation}
   where $\mathcal L$ is defined in~\eqref{prop_LM}. Combining with $Y^{\eps}_{t=0} = Y_0$ for every $\eps>0$, the equation \eqref{estimMG} implies that any subsequential limit $(Y_t)_{t\geq 0}$ of $(Y_{\eps})_{\eps>0}$ is a solution of the martingale problem for the generator $\mathcal{L}$ with initial condition $Y_{t=0} = Y_0$.
   We claim that the solution to this martingale problem is unique (we prove this in Section \ref{sec:rewrite_gen}, below).
   Given this, along with the tightness proven in Proposition~\ref{tightness}, Prokhorov's theorem implies the weak convergence of the sequence of processes $(Y^{\eps})_{\eps>0}$ to the process $(Y_t)_{t\geq 0}$ as desired.
   This completes the proof of Proposition \ref{prop_main}.
     \end{proof} 
  

\subsection{Proof of Lemma \ref{lemmaf2}}\label{sec:lemmaf2}
We now prove Lemma~\ref{lemmaf2}.
By using Lemma~\ref{l:bound}, the proof is similar that of Lemma~\ref{A1}.
\begin{proof}[Proof of Lemma~\ref{lemmaf2}]
  We split the integrand in $h^{\eps}_2$ into two parts as follows
  \begin{equation*} 
    \begin{split}
      h^\eps _2 (t)&= \eps^2 \sum_{\substack{n,m\in\Z^* \\ n+m\neq 0}} e^{2i\pi (n+m) \tau^{\eps}_t} \\
      &\qquad\cdot \iint \frac{V(t/\eps^2,dp) V(t/\eps^2,dp')-r(p)\delta(p-p')\,dp\,dp'}{\mu (|p|^{2\beta}+|p'|^{2\beta}) - 2i\pi (n+m)\omega(Y^{\eps}_t)}W_{n,m}(p, Y^{\eps}_t)\\
      &\defeq h^\eps _{21} (t) - h^\eps _{22} (t)
    \end{split}
  \end{equation*}
  where $h_{21}^{\eps}$ contains the term $V(t/\eps^2,dp)V(t/\eps^2,dp')$ and $h_{22}^{\eps}$ contains the term $r(p)\delta(p-p')$. 
  
  For $h_{22}^{\eps}$, it is straightforward to see that
  \begin{align}\label{e:h22}
    \nonumber
  \vert h^\eps_{22}(t)\vert &\leq \eps^2 C_{h,\omega_0,\nabla \omega}  \int dp \, r(p)  \\
  \nonumber
  &\qquad\cdot\sum_{\substack{n,m\in\Z^* \\ n+m\neq 0}} \frac{1}{\vert n  (n+m) \vert}\vert F_{m} (Y^{\eps}_t)\vert\Big(\big(1+|n|)\vert F_{n} (Y^{\eps}_t)\vert +  \|\textrm{Jac} F_{n}(Y^{\eps}_t)\| \Big )\\ 
  \nonumber
  &\leq \eps^2 C_{h,\omega_0,\nabla \omega}  \int dp \, r(p)\\
  &\qquad\cdot \sum_{n\in\Z^*}\frac{1}{\vert n \vert^2} \sup_{\substack{|y|\leq 2M  }} \int_0^1  \Big(| F(y,\tau)|^2 + |\partial_\tau F(y,\tau)|^2 +\|\textrm{Jac}F(y,\tau)\|^2 \Big)d\tau\,.
  \end{align} 
  We rewrite
  \begin{equation*} h^\eps_{21}(t) = \eps^2 V\Big(\frac{t}{\eps^2},\varphi_{1,t,\eps}\Big),\end{equation*}
  with
  \begin{equation*}\varphi_{1,t,\eps}(p)=V\Big(\frac{t}{\eps^2},\varphi_{2,t,p,\eps}\Big)\end{equation*}
  and
  \begin{equation*}\begin{split}
  \varphi_{2,t,p,\eps}(p')&=\sum_{\substack{n,m\in\Z^* \\ n+m\neq 0}} e^{2i\pi (n+m) \tau^{\eps}_t} W_{n,m}(p, Y^{\eps}_t)\\
  &\qquad\cdot \frac{1}{\mu (|p|^{2\beta}+|p'|^{2\beta}) - 2i\pi (n+m)\omega(Y^{\eps}_t)}\,.
  \end{split}\end{equation*}
  For readers' convenience, we pause to recall the meaning of the above notation.
  The function $\varphi_{1,t,\eps}$ of the variable $p$ is obtained by applying the measure $V(t/\e^2,dp')$ to the function $\varphi_{2,t,p,\eps}(p')$, where $\varphi_{2,t,p,\eps}$ (a function of the variable $p'$), is defined as above.
  Finally, $h_{21}^{\eps}(t)$ is obtained by applying the measure $V(t/\e^2,dp)$ to the function $\varphi_{1,t,\eps}(p)$.
  
  For $h^\eps_{22}$, one has
  \begin{equation*}|\varphi_{2,t,p,\eps}(p')|\leq C\,,\end{equation*}
  uniformly in $p$ and $p'$. 
  Following the same lines as for the first perturbation of the test function $h^\eps_1$ but using inequality \eqref{boundV2} of Lemma~\ref{l:bound} (with $k=\infty$ if $\beta\geq 1/2$ and $k\in (1,1/(1-2\beta))$ if $\beta<1/2$), we have for some $C'>0$
  \begin{equation*}\sup_{p\in S}|\varphi_{1,t,\eps}(p)| \leq \sup_{\varphi \in W_{k,C'}}\Big|V\Big(\frac{t}{\eps^2},\varphi\Big)\Big|\,.\end{equation*}
  In fact, the derivative of $\varphi_{2,t,p,\eps}$ in $p'$ can be uniformly bounded in $p$ in the corresponding $L^k(S)$ space. 
  
  Now, let us look at the derivative of $\varphi_{1,t,\eps}(p)$ in $p$. One can write
  \begin{equation*}\partial_p\varphi_{1,t,\eps}(p) = |p|^{2\beta-1}V\Big(\frac{t}{\eps^2},\varphi'_{2,t,p,\eps}\Big)+V\Big(\frac{t}{\eps^2},\varphi''_{2,t,p,\eps}\Big),\end{equation*}
  where $\varphi'_{2,t,p,\eps}$ and $\varphi''_{2,t,p,\eps}$ belong to some $W_{k,C''}$, and where $C''$ does not depend on $p$. Then,
  \begin{equation*}|\partial_p\varphi_{1,t,\eps}(p)| \leq (|p|^{2\beta-1} + 1)\sup_{\varphi \in W_{k,C''}}\Big|V\Big(\frac{t}{\eps^2},\varphi\Big)\Big|\,.\end{equation*}
  As a result, considering $\tilde C = \max(C',C'')$, and
  \begin{equation*}\tilde h^\eps_{21}(t) \defeq \frac{h^\eps_{21}(t)}{\sup_{\varphi \in W_{k,\tilde C}}|V(t/\eps^2,\varphi)|} = \eps^2 V\Big(\frac{t}{\eps^2},\tilde \varphi_{1,t,\eps}\Big)\,,\end{equation*}
  with 
  \begin{equation*}\tilde \varphi_{1,t,\eps} = \frac{ \varphi_{1,t,\eps} }{\sup_{\varphi \in W_{k,\tilde C}}|V(t/\eps^2,\varphi)|},\end{equation*}
  we have
  \begin{equation*}\|\tilde \varphi_{1,t,\eps}\|_{W^{1,k}(S)}\leq 1.\end{equation*}
  Therefore, 
  \begin{equation*}|\tilde h^\eps_{21}(t)| \leq  \eps^2 \sup_{\varphi \in W_{k,\tilde C}}\Big|V\Big(\frac{t}{\eps^2},\varphi\Big)\Big|,\end{equation*}
  and hence
  \begin{equation*}\E|h^\eps_{21}(t)|\leq \eps^2 \E\sup_{\varphi \in W_{k,\max(1,\tilde C)}}\Big|V\Big(\frac{t}{\eps^2},\varphi\Big)\Big|^2=\eps^2 \E\sup_{\varphi \in W_{k,\max(1,\tilde C)}}|V(0,\varphi)|^2\,.\end{equation*}
  Using Lemma~\ref{l:bound}, this concludes the proof.
\end{proof} 

\subsection{Proof of Lemma \ref{lemmaReps}}\label{sec:lemmaReps}

We devote this section to proving Lemma~\ref{lemmaReps}.
   By following the computations in~\eqref{def_Af1}, we obtain the following explicit formula for $\mathcal{R}_1^{\eps}(t)$:
  %
  \begin{align*}
    \MoveEqLeft
  \mathcal{R}_1^{\eps}(t) =  \eps v\Big(\frac{t}{\eps^2}\Big)\sum_{\substack{n,m\in\Z^* \\ n+m\neq 0}}\iint \frac{2i\pi (n+m)e^{2i\pi (n+m)\tau^{\eps}_t}}{\mu (|p|^{2\beta}+|p'|^{2\beta}) - 2i\pi (n+m)\omega(Y^{\eps}_t)}\\
  &\qquad\cdot \Big( G(Y_t^{\eps}) - \frac{ \nabla \omega(Y_{t}^{\eps})\cdot F(Y_t^\eps,\tau_t^\eps)}{\mu (|p|^{2\beta}+|p'|^{2\beta}) - 2i\pi (n+m)\omega(Y^{\eps}_t)} \Big)\\
  &\qquad\cdot \Big( V\Big(\frac{t}{\eps^2},dp\Big) V\Big(\frac{t}{\eps^2},dp'\Big)-r(p)\delta(p-p')\,dp\,dp'\Big)W_{n,m}(p, Y^{\eps}_t)\,\\
  &+\eps v\Big(\frac{t}{\eps^2}\Big) \sum_{\substack{n,m\in\Z^* \\ n+m\neq 0}}\iint \frac{e^{2i\pi (n+m)\tau^{\eps}_t}}{\mu (|p|^{2\beta}+|p'|^{2\beta}) - 2i\pi (n+m)\omega(Y^{\eps}_t)}\\
  &\qquad\cdot \Big( V\Big(\frac{t}{\eps^2},dp\Big) V\Big(\frac{t}{\eps^2},dp'\Big)-r(p)\delta(p-p')\,dp\,dp'\Big)\\
  &\qquad\cdot \nabla_y W_{n,m}(p,y)_{|y=Y_t^{\eps}}\cdot F(Y_t^\eps,\tau_t^\eps)\,.\\
  \end{align*}
  From here, we can see that $\mathcal{R}_1^{\eps}(t)$ can be treated in the same manner as we treated $h^{\eps}_{2}(t)$ in Section \ref{sec:lemmaf2}.
  The completes the proof.

\subsection{Proof of Lemma \ref{boundf3}}\label{sec:boundf3} 

We devote this section to proving Lemma~\ref{boundf3}.
  The bounds for $h_3^\eps(t)$ can be obtained in the same manner as the bounds for $h_{22}^{\eps}$ in equation \eqref{e:h22}.
  As a result, we obtain
  \[
      \lim_{\eps\to 0}\sup_{t\in[0,T]}\E |h^\eps_3(t)|=0 \,.
  \]
  Treating $R_2^{\eps}(t)$ in the same manner as $R_1^{\eps}(t)$ we obtain
  \begin{align*}
    \MoveEqLeft
    \mathcal{R}_2^{\eps}(t) =  \eps v\Big(\frac{t}{\eps^2}\Big)\sum_{\substack{n,m\in\Z^* \\ n+m\neq 0}}\int\,dp\, \frac{2i\pi (n+m)r(p)e^{2i\pi (n+m)\tau^{\eps}_t}}{\eps^3 - 2i\pi (n+m)\omega(Y^{\eps}_t)}\\
    &\qquad\cdot \Big( G(Y_t^{\eps}) - \frac{ \nabla \omega(Y_{t}^{\eps})\cdot F(Y_t^\eps,\tau_t^\eps)}{\eps^3 - 2i\pi (n+m)\omega(Y^{\eps}_t)} \Big) W_{n,m}(p, Y^{\eps}_t)\,\\
    &+ \eps^3 \sum_{\substack{n,m\in\Z^* \\ n+m\neq 0}}\int\,dp\, \frac{r(p)e^{2i\pi (n+m)\tau^{\eps}_t}}{\eps^3 - 2i\pi (n+m)\omega(Y^{\eps}_t)}W_{n,m}(p, Y^{\eps}_t)\\
    &+\eps v\Big(\frac{t}{\eps^2}\Big) \sum_{\substack{n,m\in\Z^* \\ n+m\neq 0}}\int\,dp\, \frac{r(p)e^{2i\pi (n+m)\tau^{\eps}_t}}{\eps^3 - 2i\pi (n+m)\omega(Y^{\eps}_t)} \nabla_y W_{n,m}(p,y)_{|y=Y_t^{\eps}}\cdot F(Y_t^\eps,\tau_t^\eps)\,.\\
  \end{align*}
  From here, we can see that each of the terms in $\mathcal{R}_2^{\eps}(t)$ can be treated in the same manner as $h^{\eps}_{2}(t)$ in Section \ref{sec:lemmaf2}.
  This concludes the proof.
  \subsection{Proof of Lemma \ref{lemma_L}}\label{sec:lemma_L}

  We devote this section to proving Lemma~\ref{lemma_L}.
  The proof analyzes how $V(t/\eps^2,dp)V(t/\eps^2,dp')$ in $\mathcal{A}_2^\eps$ (defined through Formula \eqref{eq:A2}) averages in the limit $\eps\to 0$. This is essentially similar to the analysis of the covariance in the proof of Proposition~\ref{p:qConv} in Section~\ref{sec:quad} in the case when the Hamiltonian was quadratic.
  
    Let $\eta>0$ and write
    \begin{equation*}
    \begin{split}
    \int_0^t (\mathcal{A}^\eps_2(u)-\mathcal{L}^M h(Y^{\eps}_u))\, du&=\Big(\int_0^{[t/(\eps \eta)] \eps \eta}+\int_{[t/(\eps \eta)] \eps \eta}^t \Big) (\mathcal{A}^\eps_2(u)-\mathcal{L} h(Y^{\eps}_u))\, du \\
    &=\sum_{q=0}^{[t/(\eps \eta)]-1} \int_{q\eps \eta}^{(q+1)\eps \eta} (\mathcal{A}^\eps_2(u)-\mathcal{L} h(Y^{\eps}_u))\, du \\
    &\qquad+\int_{[t/(\eps \eta)] \eps \eta}^t  (\mathcal{A}^\eps_2(u)-\mathcal{L} h(Y^{\eps}_u)) \,du \\
    &\defeq R^\eps_1(t)+R^\eps_2(t)\,.
    \end{split}\end{equation*}
    Following the computations of Lemma \ref{lemmaf2}, we have
    \begin{equation*}
    \E \vert R^\eps_2(t)\vert \leq \eps \tilde C\,.
    \end{equation*}
    so that we only have to take care of $R^\eps_1(t)$. For this term we consider the decomposition
    
    \begin{equation*}
    \begin{split}
      R^\eps_1(t)
	&=\sum_{q=0}^{[t/(\eps \eta)]-1} \int_{q\eps \eta}^{(q+1)\eps \eta} (\mathcal{A}^\eps_2(u)-\mathcal{L}h(Y^{\eps}_u)) \,du
      \\
      &=\sum_{q=0}^{[t/(\eps \eta)]-1} \int_{q\eps \eta}^{(q+1)\eps \eta} \, du
      \\
	&\qquad\qquad\qquad\cdot \sum_{n\in\Z^*} \iint \Big(V\Big(\frac{u}{\eps^2},dp\Big)V\Big(\frac{u}{\eps^2},dp'\Big)-r(p)\delta(p-p')\,dp\,dp'\Big)
      \\
      &\qquad\qquad\cdot W_{n}(p, Y^{\eps}_{q\eps\eta})
      \\
      &\qquad+\sum_{q=0}^{[t/(\eps \eta)]-1} \int_{q\eps \eta}^{(q+1)\eps \eta} \, du
      \\
	&\qquad\qquad\qquad\cdot\sum_{n\in\Z^*} \iint \Big(V\Big(\frac{u}{\eps^2},dp\Big)V\Big(\frac{u}{\eps^2},dp'\Big)-r(p)\delta(p-p')\,dp\,dp'\Big)\\
      &\qquad\qquad\cdot \Big(W_{n}(p, Y^{\eps}_u) - W_{n}(p, Y^{\eps}_{q\eps\eta})\Big)\\
      &\defeq R^\eps_{11}(t)+R^\eps_{12}(t) \,.
    \end{split}\end{equation*}
    For $R^\eps_{12}(t)$, we have
    \begin{equation*}
    \begin{split}
    \E \vert R^\eps_{12}(t) \vert & \leq  \frac{C_{h,\omega_0,\nabla \omega, F}}{\eps}\sum_{q=0}^{[t/(\eps \eta)]-1} \int_{q\eps \eta}^{(q+1)\eps \eta} du \int_{q\eps \eta}^u du' \int dp\,r(p) \leq \eta \tilde C\,,
    \end{split}\end{equation*}
    following the computations in the proof of Lemma \ref{lemmaf2}.
    It only remains to treat $R^\eps_{11}(t)$, and this term requires more care.
    We have
    \begin{equation*}
    \begin{split}
    \E \vert R^\eps_{11}(t) \vert &\leq C_{h,\omega_0,\nabla \omega, F^M} \sum_{q=0}^{[t/(\eps \eta)]-1} \sqrt{I^\eps_{q,\eta}}\,,
    \end{split}\end{equation*}
    with
    \begin{equation*}
      I^\eps_{q,\eta} \defeq \E\Bigl[\Bigl\vert\int_{q\eps \eta}^{(q+1)\eps \eta} du  \iint\Big[ V\Big(\frac{u}{\eps^2},dp\Big) V\Big(\frac{u}{\eps^2},dp'\Big)-r(p)\delta(p-p')\,dp\,dp'\Big] \Bigr\vert^2\Bigr].
    \end{equation*}
    Using Gaussian property of $V$, we have
    \begin{equation*}
    \begin{split}
    I^\eps_{q,\eta}&=  \int_{q\eps \eta}^{(q+1)\eps \eta}du_1 \int_{q\eps \eta}^{(q+1)\eps \eta}du_2 \\
    &\qquad\cdot \int \Big(\E\Big[V\Big(\frac{u_1}{\eps^2},dp_1\Big)V\Big(\frac{u_2}{\eps^2},dp_2\Big)\Big]\E\Big[V\Big(\frac{u_1}{\eps^2},dp'_1\Big) V\Big(\frac{u_2}{\eps^2},dp'_2\Big)\Big] \\
    &\qquad\qquad\qquad+  \E\Big[V\Big(\frac{u_1}{\eps^2},dp_1\Big) V\Big(\frac{u_2}{\eps^2},dp'_2\Big)\Big]\E\Big[V\Big(\frac{u_1}{\eps^2},dp'_1\Big) V\Big(\frac{u_2}{\eps^2},dp_2\Big)\Big]\Big),
    \end{split}\end{equation*}
    and by the symmetry between $u_1$ and $u_2$, we can write
    \begin{equation*}
    \begin{split}
    I^\eps_{q,\eta} & \leq C_{h,\omega_0,\nabla \omega, F} \iint \,dp \,d p' \, r(p)r(p')\\
    &\qquad\cdot  \int_{q\eps \eta}^{(q+1)\eps \eta}du_1 \int_{q\eps \eta}^{u_1}du_2 \, e^{-\mu(\vert p \vert^{2\beta}+\vert p'\vert^{2\beta})( u_1 - u_2)/\eps^2}.
    \end{split}\end{equation*}
    Now, we split the integration domain of $p$ into the region where $|p|\leq \nu$, and the region where $|p| >\nu$.
    This leads to the decomposition 
    \begin{equation*}
    I^\eps_{q,\eta}  \leq J^\eps_{1,q,\nu} + J^\eps_{2,q,\nu}.\end{equation*}
    
    For the second term (when $|p|>\nu$), we have
    \begin{equation*}\begin{split}
    \int_{q\eps \eta}^{(q+1)\eps \eta}\,du_1 \int_{q\eps \eta}^{u_1}\,du_2 \,e^{-\mu(\vert p \vert^{2\beta}+\vert p'\vert^{2\beta})( u_1 - u_2)/\eps^2} & \leq \frac{\eps^2}{\mu|p|^{2\beta}}\int_0^{\eps \eta}(1-e^{-\mu|p|^{2\beta}u_1/\eps^2})\,du \\
    &\leq  C_{\eta,\nu}\eps^3\,,
    \end{split}\end{equation*}
    so that 
    \begin{equation*}
    \sum_{q=0}^{[t/(\eps \eta)]-1} \sqrt{J^\eps_{2,q,\eta}} \leq \eps^{1/2} C_{\eta,\nu,T}\,. \end{equation*}

    For the first term (when $\abs{p} \leq \nu$), we have
    \begin{equation*}\begin{split}
    \sum_{q=0}^{[t/(\eps \eta)]-1} \sqrt{J^\eps_{1,q,\nu}} & \leq  C_{T,\eta}\Big( \int_{\{|p|\leq \nu \}\times S} \,dp\,dp'\,r(p)r(p')\Big)^{1/2}\\
    & \leq C_{T,\eta}\Big( \int_{\{|p|\leq \nu\}}\frac{dp}{|p|^{2\alpha}}\Big)^{1/2}\,.\end{split}\end{equation*}
    Since~$\alpha < 1/2$, the dominated convergence theorem implies that this converges to~$0$.
    This concludes the proof of Lemma \ref{lemma_L}.

  \subsection{Well-posedness of the martingale problem}\label{sec:rewrite_gen}
  We devote this section to showing that the martingale problem with generator~$\mathcal L$ (see equation~\eqref{prop_LM}) is well-posed.
  First, we rewrite $\mathcal{L} h(y)$ as
  \begin{equation*}
  \mathcal{L} h(y)=\frac{1}{2}\sum_{j,l=1}^2  \fraka_{jl}(y)\partial^2_{x_lx_j} h(y)+\sum_{j=1}^2  \partial_{x_j} h(y)\frakb_j(y),
  \end{equation*}
  where
  \begin{equation}\label{def_a}\begin{split}
  \fraka_{jl}(y) \defeq \sum_{n\in\Z^*} R_n(y) F_{j,n}(y)F_{-l,n}(y) = \sum_{n\in\Z^*} R_n(y) F_{j,n}(y)\overline{F_{l,n}(y)}
  \end{split}\end{equation}
  with (equivalently as in formula \eqref{def_fourieraRR})
  \begin{equation}\label{def_rn}\begin{split}
  R_n(y) & \defeq \int dp\,r(p)\Big( \frac{1}{\mu|p|^{2\beta}-2i\pi n\omega(y)} +\frac{1}{\mu|p|^{2\beta}+2i\pi n\omega(y)} \Big)\\
  &= \int dp\,r(p) \frac{2\mu|p|^{2\beta}}{\mu|p|^{4\beta}+4\pi^2 n^2\omega^2(y)}\\
  & = 2\int_0^\infty du \,\E[v(u)v(0)]\cos(2\pi n \omega(y) u )\,,
  \end{split}  \end{equation}
  and
  \begin{align}\label{def_b}
    \nonumber
    \MoveEqLeft
  \frakb_j(y) \defeq \sum_{n\in\Z^*}\int dp\, \frac{r(p)}{\mu |p|^{2\beta}-2i\pi n \omega(y)} \\
  &\cdot\Bigl( \sum_{l=1}^2 \Bigl(\partial_{x_l}F_{j,n}(y) F_{l,-n}(y)-\frac{2i\pi n }{\mu |p|^{2\beta}-2i\pi n \omega(y)} F_{j,n}(y) F_{l,-n}(y)\partial_{x_l} \omega(y)\Bigr)\Bigr).
  \end{align}
  
  The uniqueness of the martingale problem is guaranteed by \cite[Corollary 4.9 p. 317, Proposition 3.20 p. 309, and Theorem 2.5 p. 287]{KaratzasShreve91} together with \cite[Theorem 1.2 p. 129]{Friedman06}, where the latter provides the locally Lipschitz behavior of a nonnegative square root. It is not hard to see that $\frakb$ is locally Lipschitz with respect to $y$, as well as $\fraka(y)$ is a symmetric nonnegative matrix, which is $\mathcal{C}^2(\R^2)$ with respect to $y$. The smoothness property is direct from \eqref{def_a} and \eqref{def_rn}. Regarding the nonnegativity, we have for any $x\in \R^2$
  \begin{equation}\label{var_quad}x^T \fraka(y) x = \big<{\sigma}^*(y)(x),{\sigma}^*(y)(x)\big>_{\mathscr{H}}\geq 0
  \end{equation}
  with the Hilbert space 
  \begin{equation*}
    \mathscr H = \Big\{ u: \Z^* \to \mathbb{C}: \overline{u_{-n}} =u_n \quad\text{and}\quad \sum_{n\in\Z^* }|u_n|^2<\infty\Big\}\,,\end{equation*}
  the adjoint ${\sigma}_M^*$ is an operator from $\R^2$ to $\mathscr{H}$ defined by 
  \begin{equation*}
  {\sigma}_M^*(y)(n,x)\defeq  R^{1/2}_n(y)F_{1,n}(y)x_1 + R^{1/2}_n(y) F_{2,n}(y)x_2\,,\end{equation*}
  and $\sigma$, from $\mathscr{H}$ to $\R^2$, is defined by
  \begin{equation}\label{def_sigM}{\sigma}(y)(u,j)\defeq \sum_{n\in\Z^* } R^{1/2}_n(y) F_{j,n}(y)u_{-n}\qquad j=1,2\,.\end{equation}
  Hence, we have an explicit self-adjoint Hilbert-Schmidt square root for $\fraka$ that we use in Section \ref{sec:propSDE} to derive a SDE for the process $(Y_t)_{t\geq 0}$. It is also clear that $\sigma$ is locally Lipschitz with respect to $y$, so that we can conclude the uniqueness of the martingale problem. 
  
  Let us finish by rewriting the martingale problem in a more convenient way and let us start with $\frakb$. Using the fact that
  \begin{equation*}\begin{split}
  \frac{-2i\pi n \partial_{x_l}\omega(y)}{(\mu |p|^{2\beta} - 2i\pi n \omega(y))^2} &= \partial_{y_l} \int_0^\infty e^{-(\mu|p|^{2\beta}-2i\pi n \omega(y))u}\,du \\
  &=2i\pi n \partial_{x_l}\omega(y) \int_0^\infty du\,u \, e^{-(\mu|p|^{2\beta}-2i\pi n \omega(y))u}\,,
  \end{split}\end{equation*}
  we sum Fourier series in \eqref{def_b} to obtain
  \begin{equation*}
  \frakb(y) =  \int_0^1 d\tau  \int_0^\infty du\, R(u) \textrm{Jac}_y[\tilde F(u,y,\tau)]\tilde F(0,y,\tau)\,,
  \end{equation*} 
  where
  \begin{equation*}
  \tilde F(u,y,\tau) = F(y,\tau +\omega(y)u)\,.\end{equation*}
  For $\fraka$ we have the following formulation
  \begin{equation*}
  \begin{split}
  \fraka_{jl}(y) &= \int_0^1 d\tau \int_0^\infty du\, R(u) \Big(\tilde F_j(u,y,\tau)\tilde F_l(0,y,\tau) + \tilde F_j(0,y,\tau)\tilde F_l(u,y,\tau)\Big)\\
  & = 2 \int_0^1 d\tau \int_0^\infty du\, R(u)\tilde F_j(u,y,\tau)\tilde F_l(0,y,\tau),\quad j,l\in\{1,2\}\,.
  \end{split}\end{equation*}
  Finally, to conclude, we just have to remark that according to \eqref{e:defF}, as well as the definition of $\fraka$ involving the integration of the periodic variable over a whole period, we obtain the formulation \eqref{prop_LM}.

\section{Proof of Proposition \ref{prop_SDE_I_psi_M}}\label{sec:propSDE}

The proof of this proposition follows the treatment in \cite[Proposition 4.6 p. 315]{KaratzasShreve91} and consists of deriving the SDE corresponding to the martingale problem with generator \eqref{def_L}. To prove this correspondence, we consider first the function $h_j(y)=y_j$ for any $j\in\{1,2\}$, so that
\begin{equation}\label{def_M}
M^0_j(t)=Y_{j,t}-Y_{j,0}-\int_0^t \frakb_j(Y_s)\,ds
\end{equation}
is a martingale, where $\frakb$ is defined as $\frakb^M$ in \eqref{def_b}. Proceeding the same way choosing $h_{jl}(y)=y_j y_l$ one can show that for any $j,l\in\{1,2\}$
\begin{equation*}
M^0_j(t)M^0_l(t)-\int_0^t \fraka_{jl}(Y_s)\,ds
\end{equation*}   
is a martingale.
(Recall, $\fraka$ is defined by \eqref{def_a}.)
Therefore $M^0$ is a martingale with quadratic variation given by 
\begin{equation*}
t\mapsto \int_0^t \fraka_{jl}(Y_s)\,ds\,.
\end{equation*}
Using~\eqref{var_quad}, we note $\fraka = \sigma {\sigma}^*$, where $\sigma$ is defined by \eqref{def_sigM}.
We can now apply a standard martingale representation theorem involving Hilbert spaces (see \cite[Theorem 8.2]{DaPratoZabczyk96} for instance).
This guarantees the existence of a complex valued cylindrical Brownian motion $B$ ($B=(B^1-iB^2)/\sqrt{2}$, with $B^1$ and $B^2$ being two independent real valued cylindrical Brownian motions) on the Hilbert space $\mathcal{T}$,
possibly defined on an extension of the probability space under consideration, and such that with probability one
\begin{equation*}
M^0(t)=\int_0^t \sigma(Y_s)\,dB_s\\
\end{equation*}  
for all $t\geq 0$. As a result, from \eqref{def_M} we obtain the following SDE for $(Y_t)_{t\geq 0}$, 
\begin{equation*}
dY_{t}=\sigma(Y_t)dB_{t} + \frakb(Y_t)dt\,.
\end{equation*}
Note that, because of symmetries, and that for the cylindrical Brownian motion on $\mathcal{T}$ we necessarily have $B^1_{-n} = B^1_{n} $ and $B^2_{-n} = - B^2_{n}$, this defines a real valued equation. 

Regarding $\sigma$, one can write
\begin{equation}\label{def_sig}
\begin{split}
\sigma_j(Y_t)dB_{t} & = \frac{1}{\sqrt{2}} \sum_{n\in\Z^* } \tilde F_{j,n}(Y_t) R^{1/2}_n(Y_t)(dB^1_{t,n}-i dB^2_{t,n}) \\
& = \int_0^1 d\tau \tilde F_j(Y_t,\tau)  \,dW_t(Y_t,\tau)
\end{split}
\end{equation}
where 
\begin{equation*}
dW_t(y,\tau) = \frac{1}{\sqrt{2}} \sum_{n\in\Z^*} e^{2i\pi n\tau}R^{1/2}_n(y) (dB^1_{t,n}-idB^2_{t,n}) \end{equation*}
is a real valued Brownian field with covariance function
\begin{equation*}
\begin{split}
\E W_t(y,\tau)W_s(y,\tau')& = t\wedge s \sum_{n\in\Z^*}  e^{2i\pi n(\tau-\tau')} R_n(y)\\
& = t\wedge s \sum_{n\in\Z^*}  e^{2i\pi n(\tau-\tau')}\int dp\, r(p)\\
&\qquad\cdot \int_0^\infty du\,\Big[ e^{-(\mu |p|^{2\beta}+2i\pi n \omega(y))u }+e^{-(\mu |p'|^{2\beta}-2i\pi n \omega(y))u }\Big]\\
& = t\wedge s \int_0^\infty du \, R(u)\\
&\qquad\cdot \sum_{n\in\Z^*} \Big[e^{2i\pi n(\tau-\tau' + \omega(y)u)}+e^{2i\pi n(\tau-\tau' - \omega(y)u)}\Big].
\end{split}\end{equation*}
so that for any test functions $\phi_1$, $\phi_2$ in the Hilbert space
\begin{equation*}
L^2_0(\T) = \Big\{\phi \in L^2(\T):\int_0^1 \phi(\tau) d\tau = 0\Big\}\,,\end{equation*}
the covariance function reads
\begin{equation*}
\begin{split}
\E W_t(y,\phi_1)W_s(y,\phi_2) & = t\wedge s \int_0^\infty \,du\,  R(u) \\
&\cdot \int_0^1 d\tau\Big[ \phi_1(\tau + \omega(y)u)\phi_2(\tau)+\phi_1(\tau)\phi_2(\tau + \omega(y)u)\Big]\,.
\end{split}\end{equation*}

To determine the SDE satisfied by $(I^M_t)_{t\geq 0}$ we just have to consider the first coordinate of $(Y_t)_{t\geq 0}$. Then, we have

\begin{equation*}
dI^M_t  = \int_0^1 d\tau\, a^M(I^M_t,\tau) \,dW_t(I^M_t,\tau) + \int_0^1 d\tau \int_0^\infty du\, R(u) \tilde A(u,I^M_t,\tau)\,dt
\end{equation*}
with 
\begin{equation*}\begin{split}
\tilde A(u,I,\tau) & =  \partial_I \big( a^M(I,\tau + \omega(I)u) \big) a^M(I,\tau) \\
&+ \partial_\tau a^M(I,\tau + \omega(I)u) \big( b^M(I,\tau) - \av{b^M(I)}\big)\,.
\end{split}\end{equation*}
Using the fact that $a^M$ is mean-zero in $\tau$, and integrating by parts, we have
\begin{equation*}\begin{split}
\int_0^1 d\tau\, \partial_\tau & a^M(I,\tau + \omega(I)u) \big( b^M(I,\tau) - \av{b^M(I)}\big)\\
= &\int_0^1 d\tau \partial_\tau a^M(I,\tau + \omega(I)u)b^M(I,\tau)  - \int_0^1d\tau a^M(I,\tau + \omega(I)u) \partial_\tau b^M(I,\tau)\,.
\end{split}\end{equation*}

Now, we restrict our study for $t\leq \eta_M$ so that, according to \eqref{relation_phi}, we have for $I\in (1/M, M)$
\begin{equation*}\partial_\tau b^M(I,\tau) = \partial_\tau b(I,\tau) = \partial_\tau \partial_I \varphi_1^{-1}(I,\tau) = \partial_I \partial_\tau \varphi_1^{-1}(I,\tau) = -\partial_I a(I,\tau)\,,\end{equation*}
so that
\begin{equation*}
\int_0^1 d\tau \tilde A(u,I,\tau)  = \int_0^1 d\tau \,\partial_I \big( a(I,\tau + \omega(I)u) \big) a(I,\tau) +  a(I,\tau + \omega(I)u) \partial_I a(I,\tau).
\end{equation*}
We then obtain the desired result for the stopped process $\tilde I^M_t = I^M_{t\wedge \eta_M}$. Now for the process $\psi^M_{t\wedge \eta_M}$, we look at the second coordinate of $(Y_t)_{t\geq 0}$, and we have
\begin{equation*}
d\psi^M_t  = \int_0^1 d\tau\, b^M(I^M_t,\tau)\, dW_t(I^M_t,\tau) + \int_0^1 d\tau \int_0^\infty du\, R(u) \tilde B(u,I^M_t,\tau)\,dt
\end{equation*}
with 
\begin{equation*}\begin{split}
\tilde B(u,I,\tau) & =  \partial_I \big( b^M(I,\tau + \omega(I)u) - \av{b^M(I)}\big) a^M(I,\tau) \\
&\qquad+ \partial_\tau ( b^M(I,\tau + \omega(I)u) - \av{b^M(I)}) \big( b^M(I,\tau) - \av{b^M(I)}\big)\,.
\end{split}\end{equation*}
Using again that $a^M$ and $\partial_\tau b^M$ are mean-zero in $\tau$, together with an integration by parts, we obtain
\begin{equation*}\begin{split}
\int_0^1 d\tau\, \tilde B(u,I,\tau) & = \int_0^1 d\tau \,\partial_I \big( b^M(I,\tau + \omega(I)u) \big) a^M(I,\tau)\\
&\qquad -\int_0^1d\tau\, b^M(I,\tau + \omega(I)u) \partial_\tau b^M(I,\tau)\,.
\end{split}\end{equation*}
Since 
for $I\in (1/M, M)$ we know
\begin{equation*}\partial_\tau b^M(I,\tau) = -\partial_I a(I,\tau)\,,\end{equation*}
the proof of Proposition~\ref{prop_SDE_I_psi_M} is complete.


\appendix
\section{Proofs of Lemmas~\ref{l:quad1} and~\ref{l:quad2} }\label{append:quad}
  In this appendix we prove Lemmas~\ref{l:quad1} and~\ref{l:quad2}.
  While the proofs are a bit lengthy, they are a direct computation.
  
  \begin{proof}[Proof of Lemma~\ref{l:quad1}]
    From~\eqref{e:w1} we note
    \begin{equation}
      w_1^{\e}(t) = \e \int_0^{t/\e^2}\sin(\tau)v(\tau) \, d\tau\,,
    \end{equation}
    and hence
    \begin{align}
      \nonumber
      \E(w_1^{\e}(t)-w_1^{\e}(s))^2
  &= \e^2\int_{s/\e^2}^{t/\e^2}\int_{s/\e^2}^{t/\e^2} \sin(\theta)\sin(\tau)R(\theta-\tau) \, d\theta \, d\tau
      \\
      \nonumber
  &= \frac{\e^2}{2}\int_{s/\e^2}^{t/\e^2}\int_{s/\e^2}^{t/\e^2}\left [\cos(\theta-\tau)-\cos(\theta+\tau)\right]R(\theta-\tau)\, d\theta \, d\tau
      \\
      \label{e:Ew1tminusw1s2}
  &= \frac{1}{2} \paren[\big]{ I_1^\epsilon(s,t) - I_2^\epsilon(s,t) }\,,
    \end{align}
    where
    \begin{align*}
      I_1^{\e}(s,t) &\defeq \e^2 \int_{s/\e^2}^{t/\e^2}\int_{s/\e^2}^{t/\e^2}\cos(\theta-\tau)R(\theta-\tau) \, d\theta \, d\tau \,,\\
      \text{and}\qquad
      I_2^{\e}(s,t) &\defeq \e^2 \int_{s/\e^2}^{t/\e^2}\int_{s/\e^2}^{t/\e^2}\cos(\theta+\tau)R(\theta-\tau) \, d\theta \, d\tau \,.
    \end{align*}
    
    We first analyze $I_1^{\e}(s,t)$.
    Making a change of variables $z = \theta-\tau$ and integrating by parts with respect to variable $\tau$, we have
    \begin{align*}
    I_1^{\e}(s,t)
      &= \e^2 \int_{s/\e^2}^{t/\e^2}\int_{s/\e^2-\tau}^{t/\e^2-\tau}\cos(z) R(z) \, dz \, d\tau
    \\
      &= {\epsilon ^2} \brak[\Big]{ \tau \int_{s/{\epsilon ^2} - \tau }^{t/{\epsilon ^2} - \tau } \cos(z)R(z) \, dz }_{\tau  = s/{\epsilon ^2}}^{\tau  = t/{\epsilon ^2}}
      \\
      &\qquad- \epsilon^2 \int_{s/{\epsilon ^2}}^{t/{\epsilon ^2}} \tau \Bigl(
  \cos \paren[\Big]{ \frac{s}{\epsilon ^2} - \tau }
  R\paren[\Big]{ \frac{s}{\epsilon ^2} - \tau }
  - \cos \paren[\Big]{ \frac{t}{\epsilon ^2} - \tau } R\paren[\Big]{\frac{t}{\epsilon ^2} - \tau } \Bigr) \, d\tau
    \end{align*}
    The first term on the right reduces to
    \begin{equation*}
      {\epsilon ^2} \brak[\Big]{ \tau \int_{s/{\epsilon ^2} - \tau }^{t/{\epsilon ^2} - \tau }\cos(z)R(z) \, dz }_{\tau  = s/{\epsilon ^2}}^{\tau  = t/{\epsilon ^2}} = (t - s)\int_0^{(t - s)/{\epsilon ^2}} \cos(z) R(z) \, dz \,.
    \end{equation*}
    For the second term on the right, we note
    \begin{multline*}
      {\epsilon ^2}\int_{s/{\epsilon ^2}}^{t/{\epsilon ^2}} \tau \cos \left(\frac{s}{\epsilon ^2} - \tau \right)R\left(\frac{s}{\e^2}-\tau\right)\, d\tau
      \\
      = {\epsilon ^2}\int_0^{(t - s)/{\epsilon ^2}}\cos(z)R(z)zdz + s\int_0^{(t - s)/{\epsilon ^2}}\cos(z)R(z)dz \,,
    \end{multline*}
    and
    \begin{multline*}
      {\epsilon ^2}\int_{s/{\epsilon ^2}}^{t/{\epsilon ^2}}  \cos \left(\frac{t}{\epsilon ^2} - \tau \right)R\left(\frac{t}{\e^2}-\tau\right) \tau d\tau
      \\
      = -{\epsilon ^2}\int_0^{(t - s)/{\epsilon ^2}} \cos(z)R(z)zdz+ t\int_0^{(t - s)/{\epsilon ^2}}\cos(z)R(z)dz\,.
    \end{multline*}
    Hence,
    \begin{align}
      \nonumber
      I_1^{\epsilon}(s,t)
  &= 2 (t-s)\int_0^{(t-s)/\e^2} \cos(z)R(z) \, dz
    -2 \e^2\int_0^{(t-s)/\e^2} z \cos(z)R(z) \, dz
      \\
  \label{e:I11}
  &= I_{11}^\epsilon - 2 I_{12}^\epsilon \,.
    \end{align}

    Clearly
    \begin{equation}\label{e:I11bound}
      \abs{I_{11}^\epsilon}
  \leq 2 (t - s) \sup_{x \in \R} \int_0^x \cos(z) R(z) \, dz\,.
    \end{equation}
    Note that the supremum on the right-hand side is finite for all~$\gamma \in (0, 2)$, as by Remark~\ref{r:Dexists} we know  $\lim_{x \to \infty} \int_0^x \cos(z) R(z) \, dz$ exists and is finite.

    To estimate $I_{12}$, we divide the analysis into cases.
    \restartcases
    \case[$\gamma \in (0, 1)$]
    Integrating by parts we note
    \begin{align}\label{e:I12bound}
      \nonumber
      I_{12}^\epsilon
  &= \epsilon ^2 \brak[\Big]{\sin (z) z R(z) }_{z = 0}^{z = (t - s)/{\epsilon ^2}}
  -\epsilon^2\int_0^{(t - s)/{\epsilon ^2}}
       \sin(z) (R'(z)z+R(z)) \, dz\,\\
   &=
    \e^{2\gamma}(t-s)^{1-\gamma}L\Big(\frac{t-s}{\e^2} \Big)
    -\epsilon^2\int_0^{(t - s)/{\epsilon ^2}}
       \sin(z) (R'(z)z+R(z)) \, dz\,.
    \end{align}
    For the first term on the right we note that
    \begin{equation*}
      \e^{2\gamma}(t-s)^{1-\gamma}L\Big(\frac{t-s}{\e^2} \Big)
  \leq 2 \e^{2\gamma}L\Big(\frac{1}{\e^2}\Big) (t - s)^{1 - \gamma}\,,
    \end{equation*}
    for all sufficiently small~$\epsilon$.

    For the second term, we note that~\eqref{e:slowIncrease} implies
    \begin{equation*}
      |R'(z)z| = |L'(z)z^{1-\gamma}-\gamma L(z)z^{-\gamma}|\leq 2|R(z)|\,,
    \end{equation*}
    for all sufficiently large~$z$.
    Thus
    \begin{align}
      \label{e:I12bound1}
      \MoveEqLeft
      \epsilon^2\int_0^{(t - s)/{\epsilon ^2}}
  \sin(z) (R'(z)z+R(z)) \, dz
  \leq 3 \epsilon^2 \int_0^{ (t -s) / \epsilon^2 } \abs{R(z)} \, dz
      \\
      \nonumber
  &\leq 4 \epsilon^{2\gamma} L\paren[\Big]{\frac{1}{\epsilon^2}} (t - s)^{1 - \gamma}\,,
    \end{align}
    where the last inequality followed from Karamata's theorem~\cite[Th 6.2.1]{BuldyginIndlekoferEA18}, and the fact that $L$ is slowly varying.
    Combining these, we see
    \begin{equation}\label{e:I12GaLe1}
      \abs{I_{12}^\epsilon}
  \leq 10 \e^{2\gamma} L\paren[\Big]{\frac{1}{\epsilon^2}} (t - s)^{1 - \gamma}\,,
    \end{equation}
    for all sufficiently small~$\epsilon$.
    \case[$\gamma =1$]
    We follow the proof in the case $\gamma < 1$, with a few minor changes.
    For the first term on the right of~\eqref{e:I12bound}, we note that $L(z) \leq 2 R(0) z$ for all sufficiently small~$z$.
    Since $L$ is slowly varying at infinity we must certainly have $L(z) \leq z$ for all sufficiently large~$z$.
    Hence we can find a finite constant~$C$ such that $L(z) \leq C z$ for all $z \geq 0$.
    As a result, we have
    \begin{equation*}
      \epsilon^2 L\paren[\Big]{\frac{t - s}{\epsilon^2}} \leq C (t - s)\,.
    \end{equation*}

    For the remaining terms, note that the function $g(x) \defeq \int_0^x R(z) \, dz$ is slowly varying (see for instance~\cite{BinghamGoldieEA89} Proposition 1.5.9a, p.\ 26).
    Thus using~\eqref{e:I12bound1} we see that
    \begin{align}
      \epsilon^2\int_0^{(t - s)/{\epsilon ^2}}
	\sin(z) (R'(z)z+R(z)) \, dz
      &\leq 3 \epsilon^2 \int_0^{ (t -s) / \epsilon^2 } \abs{R(z)} \, dz
      \\
      \label{e:intAbsR}
      = 3 \epsilon^2 g\paren[\Big]{ \frac{t - s}{\epsilon^2} }\,.
    \end{align}
    Since $g(0) = 0$, $g'(0) = R(0) < \infty$, and $g$ is slowly varying at infinity, we must have $g(z) \leq Cz$ for all $z \geq 0$ and a finite constant~$C$.
    (As before, we allow the constant $C$ to change from line to line, provided it does not depend on $\epsilon$, $t$ and $s$.)
    Consequently,
    \begin{equation*}
      \abs{I^\epsilon_{21}} \leq C(t - s)\,,
    \end{equation*}
    as desired.

    \case[$\gamma \in (1, 2)$]
    Directly integrating $I^\epsilon_{12}$ and using Karamata's theorem we see
    \begin{equation*}
      \abs{I_{12}^\epsilon}
  \leq 4 \epsilon^{2(\gamma - 1)} L\paren[\Big]{\frac{1}{\epsilon^2}} (t - s)^{2-\gamma}\,.
    \end{equation*}
  
    \smallskip

    We now turn our attention to the term $I_2^{\e}$.
    Substituting $z = \theta-\tau$, we note
    \begin{align}
      \nonumber
      I_2^{\e}& = \e^2 \int_{s/\e^2}^{t/\e^2}\int_{s/\e^2-\tau}^{t/\e^2-\tau}\cos(z+2\tau)R(z) \, dz \, d\tau
      \\
      \label{e:iep2sim}
      & =\e^2 \int_{s/\e^2}^{t/\e^2}\int_{s/\e^2-\tau}^{t/\e^2-\tau}[\cos(z)\cos(2\tau)-\sin(z)\sin(2\tau )]R(z) \, dz \, d\tau
      \\
      &= I_{21}^\epsilon - I_{22}^\epsilon\,,
    \end{align}
    where
    \begin{align}
      \label{e:I21}
      I_{21}^\epsilon &\defeq \e^2 \int_{s/\e^2}^{t/\e^2}\int_{s/\e^2-\tau}^{t/\e^2-\tau}\cos(z)\cos(2\tau)R(z) \, dz \, d\tau \,,
      \\
      \nonumber
      \text{and}\qquad
      I_{22}^\epsilon &\defeq \e^2 \int_{s/\e^2}^{t/\e^2}\int_{s/\e^2-\tau}^{t/\e^2-\tau}\sin(z)\sin(2\tau)R(z) \, dz \, d\tau \,.
    \end{align}
    We again divide the analysis into cases.
    \restartcases
    \case[$\gamma \in (0, 1)$]
    Integrating by parts with respect to $\tau$, we observe
    \begin{align}\label{e:I21intbyparts}
      I_{21}^\epsilon
  &= \frac{\epsilon^2}{2} \brak[\Big]{ \sin(2\tau) \int_{s/{\epsilon ^2} - \tau }^{t/{\epsilon ^2} - \tau } \cos(z)R(z) \, dz }_{\tau  = s/{\epsilon ^2}}^{\tau  = t/{\epsilon ^2}}\\
      &\qquad- \epsilon^2\int_{s/{\epsilon ^2}}^{t/{\epsilon ^2}}
  \Bigl(
    \begin{multlined}[t]
      \cos \paren[\Big]{ \frac{s}{\epsilon ^2} - \tau }
        R\paren[\Big]{ \frac{s}{\epsilon ^2} - \tau }
    \\
      - \cos \paren[\Big]{ \frac{t}{\epsilon ^2} - \tau }
        R\paren[\Big]{\frac{t}{\epsilon ^2} - \tau }
    \Bigr)
    \sin(2\tau) \, d\tau \, .
    \end{multlined}
    \end{align}
    and hence by Karamata's theorem,
    \begin{equation}\label{e:I21GaLe1}
      \abs{I_{21}^\epsilon}
  \leq C \epsilon^2 \int_0^{ (t -s) / \epsilon^2 } \abs{R(z)} \, dz
  \leq C \e^{2 \gamma} L\paren[\Big]{\frac{1}{\epsilon^2}} (t - s)^{1 - \gamma}\,,
    \end{equation}
    for some finite constant $C$.
    \case[$\gamma = 1$]
    As in the previous case, we know from~\eqref{e:I21GaLe1} that
      \begin{equation*}
  \abs{I^\epsilon_{21}} \leq
    C \epsilon^2 \int_0^{(t - s) / \epsilon^2} \abs{R(z)} \, dz\,.
      \end{equation*}
      Using the same argument as that used to bound~\eqref{e:intAbsR} we obtain
      \begin{equation*}
	\abs{I^\epsilon_{21}}
	\leq C (t - s)
      \end{equation*}
      as desired.

    \case[$\gamma \in (1, 2)$]
      As before, set $g(x) \defeq \int_0^x \abs{R(z)} \, dz$, and note
      \begin{equation*}
	\abs{I_{21}^\epsilon}
	  \leq \epsilon^2 \int_{s/\epsilon^2}^{t / \epsilon^2}
	  \paren[\Big]{ g\paren[\Big]{\frac{t}{\epsilon^2} - \tau} - g\paren[\Big]{\frac{s}{\epsilon^2} - \tau}  } \, d\tau 
	  = 2\epsilon^2 \int_0^{(t - s)/\epsilon^2} g(z) \, dz \,.
      \end{equation*}
      Let $G(x) = \int_0^x g(z) \, dz$ and note that $G$ is regularly varying with index $2 - \gamma$.
      Thus
      \begin{equation*}
	\abs{I^\epsilon_{21}} \leq 2 \epsilon^2 G\paren[\Big]{\frac{t - s}{\epsilon^2}}
	  = 3\epsilon^{2\gamma - 2} \tilde L\paren[\Big]{\frac{1}{\epsilon^2}} (t - s)^{2-\gamma}
      \end{equation*}

      The estimates for $I_{22}^\epsilon$ are identical to those for $I_{21}^\epsilon$.
      Combining the above estimates we obtain~\eqref{e:EwiTight} for $i = 1$.
      The proof when $i = 2$ is identical, and this finishes the proof of Lemma~\ref{l:quad1}.
  \end{proof}

  \begin{proof}[Proof of Lemma~\ref{l:quad2}]
    We first prove equation~\eqref{e:Ewi2} for $i = 1$.
    To do this, we follow the same computation as in the proof of Lemma~\ref{l:quad1} up to~\eqref{e:I11}.
    Now note
    \begin{equation*}
      \lim_{\epsilon \to 0} I_{11}^\epsilon
  = 2 \int_0^\infty \cos(z) R(z) \, dz\,,
    \end{equation*}
    where the above integral converges absolutely for $\gamma \in (1, 2)$, and conditionally for $\gamma \in (0, 1]$ (see Remark~\ref{r:Dexists}).
    When~$\gamma \neq 1$, the proof of Lemma~\ref{l:quad1} already shows that $I_{12}^\epsilon$, $I_{21}^\epsilon$ and $I_{22}^\epsilon$ all vanish as $\epsilon \to 0$.
    When~$\gamma = 1$, the proof of Lemma~\ref{l:quad1} shows
    \begin{equation*}
      \abs{I^\epsilon_{12}} + \abs{I^\epsilon_{21}} \leq
      C \epsilon^2 \int_0^{(t - s)/\epsilon^2} \abs{R(z)} \, dz
	= C \epsilon^2 g\paren[\Big]{\frac{t - s}{\epsilon^2}}\,,
    \end{equation*}
    where $g(x) = \int_0^x \abs{R(z)} \, dz$.
    Since $g$ is slowly varying this vanishes as $\epsilon \to 0$.
    This proves~\eqref{e:Ewi2} as claimed.

  To prove~\eqref{e:Ewij}, we note
  \begin{align*}
    \MoveEqLeft
    \E (w_1^{\e}(t)-w_1^{\e}(s))(w_2^{\e}(t)-w_2^{\e}(s)) \\
      &= \e^2\int_{s/\e^2}^{t/\e^2}\int_{s/\e^2}^{t/\e^2}R\paren{ \theta-\tau} \sin(\theta)\cos(\tau) \, d\theta \, d\tau\\
      &= \frac{\e^2}{2}\int_{s/\e^2}^{t/\e^2}\int_{s/\e^2}^{t/\e^2}R( \theta-\tau) [\sin(\theta+\tau)+\sin(\theta-\tau)] \, d\theta \, d\tau \,,
      \\
      &= \tilde I_1^\epsilon + \tilde I_2^\epsilon\,,
  \end{align*}
  where
  \begin{align*}
    \tilde I_1^{\e} &\defeq \e^2 \int_{s/\e^2}^{t/\e^2}\int_{s/\e^2}^{t/\e^2}\sin(\theta-\tau)R(\theta-\tau) \, d\theta \, d\tau,\\
    \tilde I_2^{\e} &\defeq \e^2 \int_{s/\e^2}^{t/\e^2}\int_{s/\e^2}^{t/\e^2}\sin(\theta+\tau)R(\theta-\tau) \, d\theta \, d\tau.
  \end{align*}
  
  We claim $\tilde I_1^{\e} = 0$.
  To see this, make the changes of variables
  \begin{equation}
    \theta' \defeq \frac{(t+s)}{\e^2}-\theta\,,
    \quad
    \tau' \defeq \frac{(t+s)}{\e^2}-\tau \,,
  \end{equation}
  and rewrite $\tilde I_1^{\e}$ as
  \begin{align*}
    \tilde I_1^{\e}&=\e^2\int_{t/\e^2}^{s/\e^2}\int_{t/\e^2}^{s/\e^2}\sin(\tau'-\theta')R(\tau'-\theta') \, d\theta' \, d\tau'\\
      &=- \e^2 \int_{s/\e^2}^{t/\e^2}\int_{s/\e^2}^{t/\e^2}\sin(\theta'-\tau')R(\theta'-\tau') \, d\theta' \, d\tau'\\
      &=- \tilde I_1^{\e} \,.
  \end{align*}
  This implies $\tilde I_1^{\e}=0$ as claimed.
  \smallskip
  
  We now turn our attention to $\tilde I_2^{\e}$.
  Making the change of variable $z = \theta-\tau$ we note
  \begin{align*}
    \tilde I_2^{\e}& = \e^2 \int_{s/\e^2}^{t/\e^2}\int_{s/\e^2-\tau}^{t/\e^2-\tau}\sin(z+2\tau)R(z) \, dz \, d\tau\\
    & =\e^2 \int_{s/\e^2}^{t/\e^2}\int_{s/\e^2-\tau}^{t/\e^2-\tau}[\sin(z)\cos(2\tau)+\cos(z)\sin(2\tau)]R(z) \, dz \, d\tau \,.
  \end{align*}
  This is similar to the expression for $I^\epsilon_2$ (equation~\eqref{e:iep2sim}) and following the proof of Lemma~\ref{l:quad1} we see $\tilde I^\epsilon_2 \to 0$ as $\epsilon \to 0$.
  Since~$\tilde I_1^\epsilon$ and $\tilde I_2^\epsilon$ both vanish as~$\epsilon \to 0$ we obtain~\eqref{e:Ewij} as claimed.
  \end{proof}

\section{Proof of Proposition \ref{p:fbmconverge}}\label{s:fbmConvProof}

In this section we show convergence of the rescaled processes~$u^\epsilon$ (equation~\eqref{e:udef}) to fBm.
The result is similar to well known results~\cites{Taqqu74,Taqqu77,Marty05}, and the proof is presented here for completeness.

\begin{proof}[Proof of Proposition~\ref{p:fbmconverge} when $\gamma \in (0, 1)$]
  Using~\eqref{e:rDef} and~\eqref{e:udef} we note
  \begin{align*}
    \E(u^\eps(t)-u^\eps(s))^2
      &= \frac{2}{\sigma(\epsilon)^2}\int_{r_1 = s}^t \int_{r_2=s}^{r_1} R\paren[\Big]{\frac{r_1 - r_2}{\epsilon^2}} \, dr_2 \, dr_1\,.
  \end{align*}
  By~\eqref{e:RregularlyVarying} and the uniform convergence theorem~\cite[Th. 1.2.1]{BinghamGoldieEA89}, we know
  \begin{equation}\label{e:uct}
      \frac{1}{\sigma(\epsilon)^2}
      \abs[\Big]{ R \paren[\Big]{ \frac{r_1 - r_2 }{\epsilon^2} } }
      = 
  \frac{1}{L(\epsilon^{-2}) \abs{r_1 - r_2}^\gamma}
  L \paren[\Big]{ \frac{r_1 - r_2 }{\epsilon^2} }
      \leq \frac{2}{(r_1 - r_2)^{\gamma}}\,,
  \end{equation}
  for all $r_1, r_2 \in [s, t]$ and all sufficiently small~$\epsilon$.
  Hence
  \begin{equation}\label{e:unifBound1}
    \E(u^\eps(t)-u^\eps(s))^2
  \leq 4 \int_{r_1 = s}^t \int_{r_2=s}^{r_1} \frac{dr_2 \, dr_1}{(r_1 - r_2)^{\gamma}}
  = \frac{4(t - s)^{2 - \gamma}}{(1 - \gamma)(2 - \gamma)} \,.
  \end{equation}
  Moreover~\eqref{e:RregularlyVarying}, \eqref{e:uct} and the dominated convergence theorem imply
  \begin{align*}
    \lim_{\epsilon \to 0} 
      \E(u^\eps(t)-u^\eps(s))^2
  &= 2 \int_{r_1 = s}^t \int_{r_2=s}^{r_1} \frac{dr_2 \, dr_1}{(r_1 - r_2)^\gamma}
  = \frac{(t - s)^{2\mathcal H}}{(2\mathcal H - 1) \mathcal H} \,.
  \end{align*}

  Since $u^\eps$ is Gaussian this implies that the finite-dimensional distributions of $u^\eps$ converge to that of $\sigma_{\mathcal H} B^\calH$.
  For convergence in law, note~\eqref{e:unifBound1} implies
  \begin{align*}
      \E[(u^\eps(s)-u^\eps(r))^2(u^\eps(t)-u^\eps(s))^2]&\leq \E[(u^\eps(s)-u^\eps(r))^4]^{1/2}\E[(u^\eps(t)-u^\eps(s))^4]^{1/2} \\
      & \leq 3 \E[(u^\eps(s)-u^\eps(r))^2]\E[(u^\eps(t)-u^\eps(s))^2]\\
      &\leq C |t-r|^{4\calH}\,.
  \end{align*}
  Since $4\calH>1$, Theorems~13.4 and~13.5 in \cite{Billingsley99} imply that $u^\epsilon$ converges to~$\sigma_{\mathcal H} B^{\mathcal H}$ in law.
\end{proof}

\begin{proof}[Proof of Proposition~\ref{p:fbmconverge} when $\gamma \in (1, 2)$]
  Using~\eqref{e:rDef}, \eqref{e:udef} and a change of variable we note
  \begin{align}
    \nonumber
    \E(u^\eps(t)-u^\eps(s))^2
      &= \frac{2}{\sigma(\epsilon)^2}\int_{r_1 = s}^t \int_{r_2=s}^{r_1} R\paren[\Big]{\frac{r_1 - r_2}{\epsilon^2}} \, dr_2 \, dr_1
    \\
    \nonumber
      &= \frac{2}{\sigma(\epsilon)^2}\int_{r_1 = s}^t \int_{z=0}^{r_1 - s} R\paren[\Big]{ \frac{z}{\epsilon^2} } \, dz \, dr_1
    \\
    \label{e:utus1}
      &= \frac{-2}{\sigma(\epsilon)^2}\int_{r_1 = s}^t \int_{z = r_1 - s}^\infty R\paren[\Big]{ \frac{z}{\epsilon^2} } \, dz \, dr_1\,,
  \end{align}
  where the last equality followed from~\eqref{e:intR0}.
  Since~$L$ is slowly varying,~\eqref{e:RregularlyVarying} implies
  \begin{equation}\label{e:Rbd1}
    \frac{1}{\sigma(\epsilon)^2} R \paren[\Big]{ \frac{z}{\epsilon^2} }
      \leq R(0) \varmin \paren[\Big]{\frac{C}{z^{\gamma'}}}\,,
  \end{equation}
  where $\gamma' = (1 + \gamma) / 2$ and $C$ is a finite constant that is independent of~$\epsilon$.
  We will subsequently allow $C$ to change from line to line, as long as it doesn't depend on~$\epsilon$.

  Note~\eqref{e:utus1} and~\eqref{e:Rbd1} immediately imply
  \begin{equation}\label{e:unifBound2}
    \E(u^\eps(t)-u^\eps(s))^2
  \leq 2 C \int_{r_1 = s}^t \int_{z = r_1 - s}^{\infty} \frac{dz \, dr_1}{z^{\gamma'}}
  = C(t - s)^{2 - \gamma'} \,.
  \end{equation}
  Moreover, by~\eqref{e:RregularlyVarying}, \eqref{e:Rbd1} and the dominated convergence theorem, we see
  \begin{align*}
    \lim_{\epsilon \to 0} 
      \E(u^\eps(t)-u^\eps(s))^2
  &= 2 \int_{r_1 = s}^t \int_{z = r_1 -s}^{\infty} \frac{dz \, dr_2}{z^\gamma}
  = \frac{(t - s)^{2\mathcal H}}{(1 - 2\mathcal H) \mathcal H} \,.
  \end{align*}
  Now the remainder of the proof is identical to the case when $\gamma \in (0, 1)$.
\end{proof}

\begin{proof}[Proof of Proposition~\ref{p:fbmconverge} when $\gamma = 1$]
  In this case we write
  \begin{align*}
    \MoveEqLeft
    \E(u^\eps(t)-u^\eps(s))^2
      = \frac{2}{\sigma(\epsilon)^2}\int_{r = s}^t \int_{z=0}^{r - s} R\paren[\Big]{ \frac{z}{\epsilon^2} } \, dz \, dr
    \\
      &=
  \frac{2}{L(\epsilon^{-2}) \abs{\ln \epsilon}}\int_{r = s}^t \int_{z = 0}^{1} R(z) \, dz \, dr
  + \frac{2}{\sigma(\epsilon)^2}\int_{r = s}^t \int_{z = \epsilon^2}^{r - s} R\paren[\Big]{ \frac{z}{\epsilon^2} } \, dz \, dr
  \end{align*}
  The first term on the right vanishes as~$\epsilon \to 0$.
  Using the uniform convergence theorem~\cite[Th 1.2.1]{BinghamGoldieEA89} on the second term we see
  \begin{equation*}
    \lim_{\epsilon \to 0}
      \E(u^\eps(t)-u^\eps(s))^2
      = \lim_{\epsilon \to 0}
  \frac{2}{\abs{\ln \epsilon}} \int_{r = s}^t
    \int_{z = \epsilon^2}^{r -s}
      \frac{1}{z} \, dz \, dr
      = t - s\,,
  \end{equation*}
  and the convergence is uniform when $s, t$ belong to any bounded interval.
  Hence by the same argument as in the previous two cases we see that $(u^\epsilon)$ converges in law to $\sigma_{\mathcal H} B^{\mathcal H}$ as claimed.
\end{proof}

\section{Proof of Lemma~\ref{l:mar}}\label{proofpropmar}

First, to prove \eqref{markovesp} it suffices to show, for all $n\geq 1$ and $0\leq t_1\leq \dots\leq t_{n+1}$, that
\begin{equation*}\E\big[ V(t_{n+1},dp) \vert V(t_{1},\cdot),\dots ,V(t_{n},\cdot)  \big]=e^{-\mu |p|^{2\beta}(t_{n+1}-t_n)} V(t_n,dp).\end{equation*}
For $n=1$ and $0\leq t_1\leq t_2$, we write
\begin{equation*}V(t_{2},dp)=e^{-\mu |p|^{2\beta}(t_{2}-t_1)}V(t_{1},dp) + Y,\end{equation*}
where $Y$ and $V(t_{1},dp)$ are independent. In fact, since they are mean-zero Gaussian variables, we have
\begin{equation*}\begin{split} 
\E[Y(\varphi)V(t_1,\psi)]&= \E[V(t_2,\varphi) V(t_1,\psi)]- \E[V(t_1,\varphi_{t_2-t_1}) V(t_1,\psi)]\\
&=\int dp \, r(p)\varphi(p)\psi(p)(e^{-\mu |p|^{2\beta}(t_{2}-t_1)}-e^{-\mu |p|^{2\beta}(t_{2}-t_1)} )\\
&=0,
\end{split} \end{equation*}
for all $\varphi,\psi$ bounded continuous functions where $\varphi_{s}(p)=e^{-\mu |p|^{2\beta}s}\varphi(p)$. As a result, we have
\begin{equation*}\E\big[ V(t_{2},dp) \vert V(t_{1},\cdot)\big]=e^{-\mu |p|^{2\beta}(t_{2}-t_1)} V(t_1,dp).\end{equation*}
Now, let us fix $n\geq 2$ and assume that for all family $(s_j)_{j\in\{1,\dots,n\}}$ such that $0\leq s_1\leq \dots\leq s_{n}$
\begin{equation*}  \E\big[V(s_{n},dp) \vert V(s_{1},\cdot),\dots ,V(s_{n-1},\cdot)  \big]=e^{-\mu |p|^{2\beta}(s_{n}-s_{n-1})}V(s_{n-1},dp), \end{equation*}
Then, we write
\begin{equation*}V(t_{n+1},dp)=e^{-\mu |p|^{2\beta}(t_{n+1}-t_n)}V(t_{n},dp) + Y,\end{equation*}
where $Y$ and $V(t_{n},dp)$ are independent as explained above, so that 
\begin{equation*}\begin{split}
  \MoveEqLeft
  \E\big[ V(t_{n+1},dp) \vert V(t_{1},\cdot),\dots ,V(t_{n},\cdot)  \big]\\
  =&\; e^{-\mu |p|^{2\beta}(t_{n+1}-t_n)}V(t_n,dp)+\E\big[ Y  \vert V(t_{1},\cdot),\dots ,V(t_{n-1},\cdot) \big] \\
  =&\;e^{-\mu |p|^{2\beta}(t_{n+1}-t_n)}V(t_n,dp)+\E\big[  V(t_{n+1},dp)  \vert V(t_{1},\cdot),\dots ,V(t_{n-1},\cdot) \big]\\
  &-e^{-\mu |p|^{2\beta}(t_{n+1}-t_n)}\E\big[ V(t_{n},dp)  \vert V(t_{1},\cdot),\dots , V(t_{n-1},\cdot) \big]\\
  =&\; e^{-\mu |p|^{2\beta}(t_{n+1}-t_n)}V(t_n,dp)\\
  &+(e^{-\mu |p|^{2\beta}(t_{n+1}-t_{n-1})}-e^{-\mu |p|^{2\beta}(t_{n+1}-t_{n})}e^{-\mu |p|^{2\beta}(t_{n}-t_{n-1})})V(t_{n-1},dp)\\
  =&\; e^{-\mu |p|^{2\beta}(t_{n+1}-t_n)}V(t_n,dp),
\end{split}\end{equation*}
which concludes the proof of \eqref{markovesp} by induction. 

Second, to prove \eqref{markovvar} it suffices to show that for all $n\geq 1$, $0\leq t_1\leq \dots\leq t_{n+1}\leq \tilde{t}_{n+1}$ and $\varphi,\psi$ bounded continuous functions, that
\begin{align*}
    \MoveEqLeft
    \E\big[ V(\tilde{t}_{n+1},\varphi) V(t_{n+1},\psi) \vert V(t_{1},\cdot),\dots ,V(t_{n},\cdot)  \big]
    \\
    &=\E\big[ V(\tilde{t}_{n+1},\varphi) \vert  V(t_{1},\cdot),\dots ,V(t_{n},\cdot) \big]
      \E\big[ V(t_{n+1},\psi) \vert V(t_{1},\cdot),\dots ,V(t_{n},\cdot) \big]
    \\
    &\qquad+\int dp\, r(p)\varphi(p)\psi(p)
    \begin{multlined}[t]
      \Bigl(e^{-\mu |p|^{2\beta}(\tilde{t}_{n+1}-t_{n+1})}
      \\
  -e^{-\mu |p|^{2\beta}(\tilde{t}_{n+1}-t_{n})}
   e^{-\mu |p|^{2\beta}(t_{n+1}-t_{n})}
       \Bigr) \,.
    \end{multlined}
\end{align*}
This last relation is a consequence of the following lemma, which is a consequence of \cite[Theorem 10.1 and Theorem 10.2]{Rozanov87}.
\begin{lemma}\label{projcond}
  Let $(X,Y,Z_1,\dots,Z_n)$ be a Gaussian vector on a probability space $(\Omega,\mathcal{F},\P)$, and $\mathcal{G}=\sigma(Z_1,\dots,Z_n)$ be the $\sigma$-field generated by $Z_1,\dots,Z_n$. Then, the couple $(X-\E[X\vert \mathcal{G}],Y-\E[Y\vert \mathcal{G}])$ is independent of $\mathcal{G}$, and
  \begin{equation}\label{expcondproj}\E[XY\vert\mathcal{G}]=\E[X\vert \mathcal{G}]\E[Y\vert \mathcal{G}]+\E[(X-\E[X\vert \mathcal{G}])(Y-\E[Y\vert \mathcal{G}])].\end{equation}
\end{lemma}
\begin{proof}[Proof of Lemma \ref{projcond}] 
  The proof of the independence of $(X-\E[X\vert \mathcal{G}],Y-\E[Y\vert \mathcal{G}])$ with respect to $\mathcal{G}$ is a consequence of Theorem 10.1 and Theorem 10.2 of \cite{Rozanov87}. Consequently,
  \begin{equation*}
  \E\Big[\Big(X-\E[X\vert \mathcal{G}]\Big)\Big(Y-\E[Y\vert \mathcal{G}]\Big)\Big\vert \mathcal{G}\Big]=\E\Big[\Big(X-\E[X\vert \mathcal{G}]\Big)\Big(Y-\E[Y\vert \mathcal{G}]\Big)\Big],\end{equation*}
  and then
  \begin{equation*}
  \begin{split}
  \E[XY\vert\mathcal{G}]&=\E\Big[\Big(X-\E[X\vert \mathcal{G}]+\E[X\vert \mathcal{G}]\Big)\Big(Y-\E[Y\vert \mathcal{G}]+\E[Y\vert \mathcal{G}]\Big)\Big\vert\mathcal{G}\Big]\\
  &=\E[X\vert \mathcal{G}]\E[Y\vert \mathcal{G}]+\E\Big[\Big(X-\E[X\vert \mathcal{G}]\Big)\Big(Y-\E[Y\vert \mathcal{G}]\Big)\Big]\\
  &+\E[Y\vert \mathcal{G}]\underbrace{\E\big[\big(X-\E[X\vert \mathcal{G}]\big)\big\vert\mathcal{G}\big]}_{=0}+\E[X\vert \mathcal{G}]\underbrace{\E\big[\big(Y-\E[Y\vert \mathcal{G}]\big)\big\vert\mathcal{G}\big]}_{=0}.
  \end{split}\end{equation*}
\end{proof}
As a result, for all $\varphi,\psi$ bounded continuous functions, with
\begin{equation*}X=V(\tilde{t}_{n+1},\varphi)\qquad \text{and}\qquad Y=V(t_{n+1},\psi), \end{equation*}
we have
\begin{equation*}\begin{split}
    \MoveEqLeft
    \E\big[V(\tilde{t}_{n+1},\varphi) V(t_{n+1},\psi) \vert V(t_{1},\cdot),\dots ,V(t_{n},\cdot)  \big] \\
    &=\E\big[ V(\tilde{t}_{n+1},\varphi) \vert  V(t_{1},\cdot),\dots ,V(t_{n},\cdot) \big]\\
&\qquad\cdot\E\big[V(t_{n+1},\psi) \vert  V(t_{1},\cdot),\dots ,V(t_{n},\cdot) \big]\\
&\quad+P,
\end{split}\end{equation*}
where, using \eqref{markovesp},
\begin{equation*}\begin{split}
P&=\E\Big[\Big( X-\E\big[ X\vert  V(t_{1},\cdot),\dots ,V(t_{n},\cdot) \big] \Big)\Big(Y-\E\big[Y \vert V(t_{1},\cdot),\dots ,V(t_{n},\cdot) \big] \Big)\Big]\\
&=\E\Big[\Big( V(\tilde{t}_{n+1},\varphi)-V(t_n,\varphi_{\tilde{t}_{n+1}-t_n})\Big)\Big(V(t_{n+1},\psi)-V(t_n,\psi_{t_{n+1}-t_n})\Big)\Big]\\
&=\int dp \, r(p)\varphi(p)\psi(p)(e^{-\mu |p|^{2\beta}(\tilde{t}_{n+1}-t_{n+1})}-e^{-\mu |p|^{2\beta}(\tilde{t}_{n+1}-t_{n})}e^{-\mu |p|^{2\beta}(t_{n+1}-t_{n})}),
\end{split}\end{equation*}
which concludes the proof of \eqref{markovvar}.

\section{Proof of Lemma~\ref{l:bound}}\label{proofpropbound}

Let us start with the following remark. For any $(t,\varphi)\in D_{k,M}$, it is straightforward that $\varphi(t,\cdot)\in W_{k,M}$, so that

\begin{equation*}\E \sup_{(t,\varphi)\in D_{k,M}} \Big\vert V\Big(\frac{t}{\eps^2},\varphi(t,\cdot)\Big)\Big\vert \leq \E \sup_{(t,\varphi)\in \tilde D_{k,M}} \Big\vert V\Big(\frac{t}{\eps^2},\varphi\Big)\Big\vert\,,
\end{equation*}
with 
\begin{equation*}
\tilde D_{k,M} \defeq [0,T]\times W_{k,M}.
\end{equation*}
As a result, to prove inequality \eqref{boundV1}, we only need to prove that
\begin{equation*}\E\sup_{(t,\varphi)\in \tilde D_{k,M}} \Big\vert V\Big(\frac{t}{\eps^2},\varphi\Big)\Big\vert\leq C + \frac{C(\eps)}{\eps}\,,\end{equation*}
where $C>0$ and $C(\eps)>0$ goes to $0$ as $\eps\to 0$. To this end we first remark that there is a compact embedding of $\tilde D_{k,M}$ into $[0,T]\times C^0(S)$ equipped with the metric
\begin{equation}\label{def-norm}
\|(t,\varphi)\|_\eps := \sqrt{|t|}/\eps + \|\varphi\|_\infty\,.
\end{equation} 
Let us also recall \cite[Theorem 8.8]{Brezis11} that there exists $C_k>0$ such that for any $\varphi\in W^{1,k}(S)$, we have
\begin{equation*} \|\varphi\|_\infty \leq C_k \|\varphi\|_{W^{1,k}}\,.\end{equation*}
We also remark that for the pseudometric
\begin{equation*}
d_{V,\eps}((t,\varphi),(s,\psi))=\E[\vert V(t/\eps^2,\varphi)-V(s/\eps^2,\psi)\vert^2]^{1/2}\,,
\end{equation*}
for which
\begin{equation}\label{ineq-d}
\begin{split}
d^2_{V,\eps}((t,\varphi),(s,\psi)) &= \int_S dp \,r(p)(\varphi^2(p)+\psi^2(p)-2\varphi(p)\psi(p)e^{-\mu|p|^{2\beta}|t-s|/\eps^2}) \\
&= \int_S dp \,r(p) (\varphi^2(p)+\psi^2(p))(1-e^{-\mu|p|^{2\beta}|t-s|/\eps^2})\\
&+ \int_S dp \,r(p)(\varphi(p)-\psi(p))^2 e^{-\mu|p|^{2\beta}|t-s|/\eps^2}\\
&\leq C_{k,M,r} \|(t,\varphi) -(s,\psi)\|^2_\eps\,,
\end{split}
\end{equation}
we have
\begin{equation*}diam_{d_{V,\eps}}(\tilde D_{k,M})\leq \sqrt{2C_{k,M,r}(T+M)} \defeq c_{k,M,r}\,.\end{equation*}

Note now that it is not easy to see if $\tilde D_{k,M}$ is relatively compact equipped with $d_{V,\eps}$, while it is required to apply \cite[Theorem 1.5.1 p. 41]{AdlerTaylor07}. The problem is that $d_{V,\eps}$ is only a pseudometric. For instance, $d_{V,\eps}$ does not separate well points, since we have for any $s$ and $t$
\begin{equation*}
d_{V,\eps}((t,0),(s,0))=0\,.
\end{equation*}
However, according to \eqref{ineq-d}, the application $d_{V,\eps}$ provides a true metric on $\tilde D_{k,M}\setminus [0,T]\times\{0\}$. Then, instead of working directly with $\tilde D_{k,M}$ let us consider now the increasing sequence of $\|\cdot\|_\eps$-relatively compact subset of $\tilde D_{k,M}$
\begin{equation*} D^n_{k,M} \defeq \tilde D_{k,M}\cap [0,T]\times \big\{\varphi \in W^{1,k}(S):\, \|\varphi\|_\infty \geq 1/n\big\}\,.\end{equation*}
Now, since $d_{V,\eps}$ is a metric on every $D^n_{k,M}$, it is straightforward to see from \eqref{ineq-d} that all the $D^n_{k,M}$ are also $d_{V,\eps}$-relatively compact from the sequential characterization of compactness.   

We can now apply \cite[Theorem 1.5.1 p. 41 and Lemma 1.5.2 p. 44]{AdlerTaylor07} over all the $\overline{D^n_{k,M}}^{d_{V,\eps}}$, we have
\begin{equation*}
\begin{split}
\E\sup_{(t,\varphi)\in D^n_{k,M}} V\Big(\frac{t}{\eps^2},\varphi \Big) & \leq K\Big( c_{k,M,r}\ln(c_{k,M,r}) + \int_0^{c_{k,M,r}} \sqrt{\ln(N_{\eps}(r))} dr \Big)\,,
\end{split}
\end{equation*}
where $N_\eps(r)$ is the smallest number of balls covering $D^n_{k,M}$ with radius $r$, which are defined by
\begin{equation*}
B(X,r)=\{ Y\in D^n_{k,M}:\, d_{V,\eps}(X,Y)< r\}\,.
\end{equation*}
Because of \eqref{ineq-d}, it is clear that 
\begin{equation*}N_\eps(r)\leq \mathcal{N}_\eps\big(r/C^{1/2}_{k,M,r}\big)\,,\end{equation*}  
where $\mathcal{N}_\eps(u)$ stands for the smallest number of balls with radius $u$ associated to the metric defined by the norm $\|\cdot\|_\eps$. Since $D^n_{k,M}$ is defined as a product space, one can determine the smallest number of balls with radius $r$ that cover each of its components. For each parts, the metric is defined by the corresponding parts of the r.h.s of \eqref{def-norm}. Therefore, the smallest number of balls with radius $r$ that cover $[0,T]$ is of order $1/(r^2\eps^2)$, and for $\{\varphi \in W^{1,k}(S): \quad \|\varphi\|_{W^{1,k}}\leq M\}$ it is of order $\exp(1/r)$ \cite[Theorem 5.2 p. 311]{BirmanSolomjak67}. As a result, for any $n\in\N^*$
\begin{equation*}
\begin{split}
\E \sup_{(t,\varphi)\in D^n_{k,M}} V\Big(\frac{t}{\eps^2},\varphi \Big) & \leq C_1 + C_2\int_0^{c_{k,M,r}} \sqrt{\ln\Big(\frac{C \exp(1/r)}{r^2\eps^2}\Big)}\, dr\,,\\
& \leq C'_1 + C'_2\int_0^{c_{k,M,r}} \frac{dr}{\sqrt{r}} + \frac{C'_3}{\eps}\int_0^{\eps c_{k,M,r}} \sqrt{\ln\Big(\frac{1}{u}\Big)} \,du\,.
\end{split}
\end{equation*}

Finally, setting
\begin{equation*}C(\eps)\defeq C'_3\int_0^{\eps c_{k,M,r}} \sqrt{\ln\Big(\frac{1}{u}\Big)} \,du\,,\end{equation*}
and using the monotone convergence theorem, we have
\begin{equation*}E_1\defeq\E\sup_{(t,\varphi)\in \tilde D_{k,M}} V\Big(\frac{t}{\eps^2},\varphi \Big) =\lim_{n\to \infty}\E \sup_{(t,\varphi)\in D^n_{k,M}} V\Big(\frac{t}{\eps^2},\varphi \Big) \leq C''_1 + \frac{C(\eps)}{\eps}\,.\end{equation*}
We conclude the proof of the bound \eqref{boundV1} using that
\begin{equation*}\E\sup_{(t,\varphi)\in \tilde D_{k,M}} \Big\vert V\Big(\frac{t}{\eps^2},\varphi \Big) \Big\vert \leq 2E_1\end{equation*}
by symmetry.

For proving \eqref{boundV2}, we first remark that by stationarity in $t$ we have
\begin{equation*}\E\sup_{\varphi\in W_{k,M}} \Big\vert V\Big(\frac{t}{\eps^2},\varphi \Big) \Big\vert^n =\E\sup_{\varphi\in W_{k,M}}\vert V(0,\varphi)\vert^n \,,\end{equation*}
so that following the same lines as above by removing the $t$-dependence we have for the first order moment,
\begin{equation*}\E\sup_{\varphi\in W_{k,M}}\vert V(0,\varphi)\vert \leq 2\E\sup_{\varphi\in W_{k,M}} V(0,\varphi)\leq C_1\,.\end{equation*}
Finally, for the arbitrary order moments we write
\begin{equation*}\begin{split}
E'_n\defeq \E\sup_{\varphi\in W_{k,M}}\vert V(0,\varphi)\vert^n  &= \int_0^\infty du\, \P\Big(\sup_{\varphi\in W_{k,M}}\vert V(0,\varphi)\vert^n>u\Big) \\
&= \int_0^\infty du\, \P\Big(\sup_{\varphi\in W_{k,M}}\vert V(0,\varphi)\vert>u^{1/n}\Big)\\
&\leq 2\int_0^\infty du\,\P\Big(\sup_{\varphi\in W_{k,M}} V(0,\varphi) >u^{1/n}\Big)\,.
\end{split}\end{equation*}
Denoting 
\begin{equation*}
E''_1 = \E\sup_{\varphi\in W_{k,M}} V(0,\varphi)\,,
\end{equation*}
we have
\begin{equation*}\begin{split}
E'_n & \leq 2{E''_1}^n + 2\int_{{E''_1}^n}^\infty du\, \P\Big(\sup_{\varphi\in W_{k,M}} V(0,\varphi) -E''_1>u^{1/n}-E''_1\Big)\\
& \leq 2{E''_1}^n + 2\int_{{E''_1}^n}^\infty du\, e^{-(u^{1/n}-E''_1)^2/(2\sigma_D^2)}\,,
\end{split}\end{equation*}
where the last line is given by \cite[Theorem 2.1.1 p. 50]{AdlerTaylor07}
and
\begin{equation*}\sigma_D^2\defeq \sup_{\varphi\in W_{k,M}}\E V(0,\varphi)^2=\sup_{\varphi \in W_{k,M}}\int dp\,r(p)\varphi(p)^2\leq C_k^2 M^2\int dp\,r(p)\,.\end{equation*}
As a result,
\begin{equation*}\begin{split}
E'_n&\leq 2{E''_1}^n + 2(n-1)\int_{E''_1}^\infty dv\, v^{n-1} \,e^{-(v-E''_1)^2/(2\sigma_D^2)}\\
&\leq 2{E''_1}^n  + 2(n-1)\int_{0}^\infty dv\, v^{n-1} \,e^{-v^2/(2\sigma_D^2)}\,,
\end{split}\end{equation*}
which concludes the proof of the bound \eqref{boundV2}.
$\hfill \square$

\section{Proof of Proposition \ref{l:bound_a}}\label{proofpropbound_a}

The estimates for $a$, $\partial_\theta a$, and $\partial_I a$ we prove in this section are based on the following estimates on the inverse map $\varphi^{-1}$.

\begin{lemma}\label{lemma_bound_phi-1}
  There exist $r>0$ small enough, and a constant $C_r>0$ such that for any $I \in(0,r)$
  \[
  \sup_{\theta \in\mathbb{T}}\Big(\|\varphi^{-1}(I,\theta)\| + \|\partial_\theta \varphi^{-1}(I,\theta)\|\Big) \leq C_r \sqrt{I}\,,
  \]
  and
  \[
  \sup_{\theta \in\mathbb{T}}\|\partial_I \varphi^{-1}(I,\theta)\| \leq \frac{C_r}{\sqrt{I}}\,.
  \]
\end{lemma}
The proof of this lemma is postponed to the end of this appendix.

In what follows we denote by
\[
B_r = \varphi^{-1}((0,r)\times \mathbb{T})
\] 
with $r>0$ small enough so that $B_r$ describes a small neighborhood of $(0,0)\in \R^2$.

\subsection{Bound for \texorpdfstring{$a$}{a}.}

Let us remind the reader that
\[
a(I,\theta) = e_2 \cdot \nabla \varphi_1(\varphi^{-1}(I,\theta)) = \partial_y\varphi_1(\varphi^{-1}(I,\theta))\,,
\]
so that differentiating w.r.t. $y$ the relation $K(I(x,y))=H(x,y)$ we have
\[
\partial_y I(x,y) = \partial_y \varphi_1(x,y)  = \partial_y H(x,y)/\omega(I)\,.
\]
Now, using that $\partial_y H(0,0)=0$, we have for any $I\in(0,r)$
\[\begin{split}
|a(I,\theta)|& \leq \omega^{-1}_0 |\partial_y H(\varphi^{-1}(I,\theta)) -\partial_y H(0,0)|\\
& \leq \omega^{-1}_0 \sup_{B_r} \| \nabla \partial_y H \|\cdot \|\varphi^{-1}(I,\theta)\| \leq C_r \sqrt{I}.
\end{split}\]

\subsection{Bound for \texorpdfstring{$\partial_\theta a$}{partial\_theta a}}

Using that
\begin{equation}\label{eq:relation_a}
a(I,\theta) = \partial_y H(\varphi^{-1}(I,\theta))/\omega(I)\,,
\end{equation}
and now differentiating w.r.t. $\theta$, we obtain
\[
\partial_\theta a(I,\theta) = \partial_\theta \varphi^{-1}(I,\theta)\cdot \nabla \partial_y H(\varphi^{-1}(I,\theta)) /\omega(I)\,.
\]
Then, by the Cauchy-Schwarz inequality
\[\begin{split}
|\partial_\theta a(I,\theta)| & \leq \omega^{-1}_0 \|\partial_\theta \varphi^{-1} (I,\theta)\|\sup_{B_r} \| \nabla \partial_y H \| \leq C_r \sqrt{I}\,.
\end{split}\]

\subsection{Bound for \texorpdfstring{$\partial_I a$}{partial\_I a}} 

Differentiating \eqref{eq:relation_a} w.r.t. $I$, we obtain
\[\begin{split}
\partial_I a(I,\theta)& = -\frac{\omega'(I)}{\omega^2(I)} \partial_y H(\varphi^{-1}(I,\theta) ) + \frac{1}{\omega(I)} \partial_I \varphi^{-1}(I,\theta) \cdot \nabla \partial_y H(\varphi^{-1}(I,\theta))\\
& \defeq T_1 + T_2\,.
\end{split}\]
For $T_1$, using \eqref{hyp_omega} and proceeding as for the function $a$ to deal with the term $\partial_y H$ we have
\[|T_1| \leq \frac{C_r}{\sqrt{I}}\,.\]
For $T_2$, applying the Cauchy-Schwarz inequality together with Lemma \ref{lemma_bound_phi-1} yields
\[|T_2|\leq \frac{C'}{\sqrt{I}}\,,\]
which concludes the proof of Proposition \ref{l:bound_a}.

\subsection{Proof of Lemma \ref{lemma_bound_phi-1}}

 Before going into the proof of the lemma, let us remind the reader that $H$ admits a unique minimum at $(0,0)$ so that $\nabla^2 H(0, 0)$ is a positive definite matrix
\[
\nabla^2 H(0, 0) = \begin{pmatrix} \lambda_1 & \lambda_{12} \\ \lambda_{12} &\lambda_2 \end{pmatrix}
\]
with $\lambda_1$, $\lambda_2$ $>0$. In this case, the inner product corresponding to this matrix defines a norm on $\R^2$ which is equivalent to the Euclidean norm $|\cdot|$. In other words, there exist two constants $\lambda$, $\bar \lambda > 0$ such that for any $X=(x,y)^T\in \R^2$
\begin{equation}\label{eq:equiv_norm}
\lambda |X|^2 \leq X^T \nabla^2 H(0, 0) X \leq \bar \lambda |X|^2\,.
\end{equation}

\subsubsection{Bound for $\varphi^{-1}$}

Since $(0,0)$ is the unique minimum to $H$, a Taylor expansion at the second order gives us
\begin{equation}\label{DLH}
K(I) = H(X) = \frac{1}{2}X^T \nabla^2 H(0, 0) X  + o(|X|^2)\,,
\end{equation}
with $X=(x,y)^T$. For any $X\in B_r$ we can have using \eqref{eq:equiv_norm} that
\begin{equation}\label{ineq_K}
K(I)\geq \tilde\lambda |X|^2 \geq \tilde \lambda y^2
\end{equation}
with $\tilde\lambda = \lambda /2  - \mu > 0$ for some small $\mu$, yielding directly
\[
|\varphi^{-1}_2(I,\theta)| = |y| \leq \sqrt{K(I)/\tilde\lambda} = C \sqrt{K(I)-K(0)} \leq C \sup_{J \in (0,r) }\sqrt{\omega(J)} \sqrt{I} = C'\sqrt{I}.
\] 
In the same way, we also have
\[
|\varphi^{-1}_1(I,\theta)| \leq C'_r\sqrt{I}\,.
\]

\subsubsection{Bound for $\partial_\theta \varphi^{-1}$}

Differentiating w.r.t. $x$ and then $y$ the relation $K(I(X)) = H(X)$ we have
\[
\partial_x \varphi_1(X) = \partial_x H(X) / \omega(I)\qquad \text{and}\qquad \partial_y \varphi_1(X) = \partial_y H(X)/\omega(I)\,,
\]
so that using \eqref{relation_phi} we obtain
\begin{equation}\label{eq:partial_theta_varphi}
  \partial_\theta \varphi^{-1}_2(I,\theta)
    = \frac{\partial_x H(\varphi^{-1}(I,\theta))}{\omega(I)}
  \quad\text{and}\quad
  \partial_\theta \varphi^{-1}_1(I,\theta)
    = \frac{-\partial_y H(\varphi^{-1}(I,\theta))}{\omega(I)} \,.
\end{equation}
Therefore, we have
\[
\| \partial_\theta \varphi^{-1}(I,\theta) \|\leq \|\nabla H(\varphi^{-1}(I,\theta))\|/\omega_0\,,
\]
with
\begin{equation}\label{ineq_grad}\begin{split}
\|\nabla H(\varphi^{-1}(I,\theta))\| & = \|\nabla H(\varphi^{-1}(I,\theta))-\nabla H(0,0)\| \\ 
& \leq \sup_{B_r} \|\nabla^2 H\| \cdot\|\varphi^{-1}(I,\theta)\| \leq C_r \sqrt{I}\,.
\end{split}\end{equation}

\subsubsection{Bound for $\partial_I \varphi^{-1}$}

Differentiating \eqref{eq:partial_theta_varphi} w.r.t. $I$ we obtain
\[\begin{split}
\partial^2_{I\theta} \varphi^{-1}_2(I,\theta) &= -\frac{\omega'(I)}{\omega^2(I)}\partial_x H(\varphi^{-1}(I,\theta)) + \frac{1}{\omega(I)} \partial_I \varphi^{-1} (I,\theta) \cdot \nabla \partial_x H (\varphi^{-1}(I,\theta))\\
\partial^2_{I\theta} \varphi^{-1}_1(I,\theta) &= -\frac{\omega'(I)}{\omega^2(I)}\partial_y H(\varphi^{-1}(I,\theta)) + \frac{1}{\omega(I)} \partial_I \varphi^{-1} (I,\theta) \cdot \nabla \partial_y H (\varphi^{-1}(I,\theta))\,,
\end{split}\]
so that using \eqref{partial_I_phi_10} and \eqref{ineq_grad}
\[
|\partial_{I} \varphi^{-1}_1(I,\theta)| \leq \frac{C_1}{\sqrt{I}} + \frac{1}{\omega_0} \sup_{B_r} \|\nabla^2 H\| \int_0^\theta \|\partial_{I} \varphi^{-1}(I,\theta')\| d\theta'\,,
\]
and in the same way for $\varphi^{-1}_2$
\[
|\partial_{I} \varphi^{-1}_2(I,\theta)| \leq \frac{C_2}{\sqrt{I}} + |\partial_{I} \varphi^{-1}_2(I,0)| + \frac{1}{\omega_0} \sup_{B_r} \|\nabla^2 H\| \int_0^\theta \|\partial_{I} \varphi^{-1}(I,\theta')\| \,d\theta'.
\]
Letting
\[U(I,\theta) \defeq |\partial_{I} \varphi^{-1}_1(I, \theta)|+ |\partial_{I} \varphi^{-1}_2(I, \theta)|\,, \]
and gathering the last two inequalities, we have  
\[
U(I,\theta) \leq \frac{C}{\sqrt{I}} + |\partial_{I} \varphi^{-1}_2(I,0)| + C'\int_0^\theta U(I,\theta')\, d\theta'\,. 
\]
Then, applying the Gronwall lemma, we obtain 
\[
\sup_{\theta\in(0,1)}U(I,\theta) \leq  \Big(\frac{C}{\sqrt{I}} + |\partial_{I} \varphi^{-1}_2(I,0)| \Big) e^{C'}\,.
\]
Moreover, differentiating w.r.t. $I$ the relation
\[H(\varphi^{-1}(I,\theta)) = K(I)\]
we have
\[\partial_I\varphi^{-1}_1(I,\theta) \partial_x H(\varphi^{-1}(I,\theta)) + \partial_I\varphi^{-1}_2(I,\theta)\partial_y H(\varphi^{-1}(I,\theta)) = \omega(I)\]
and setting $\theta=0$ we have from \eqref{phi_10} and  \eqref{partial_I_phi_10}
\[\partial_I\varphi^{-1}_2(I,0)\partial_y H(0,\varphi^{-1}_2(I,0)) = \omega(I)\,.\]
Also, using that 
\[\begin{split}
\partial_y H(0,y) &= \partial_y H(0,0) + \partial^2_{yy} H(0,0) y  + o(y)= \lambda_2 y + o(y)\,,
\end{split}\]
so that for any $(0, y) \in B_r$ we have, for some small $\mu$,
\[
|\partial_y H(0,y)| \geq (\lambda_2 -\mu)|y|\,.
\]
As a result,
\[
|\partial_I\varphi^{-1}_2(I,0)| \leq \frac{\omega(I)}{(\lambda_2 -\mu) |\varphi^{-1}_2(I,0)|}\,,
\]
and from \eqref{phi_10} together with \eqref{DLH}, we have
\[K(I) = H(0,\varphi^{-1}_2(I,0))\leq \bar \lambda_2\big(\varphi^{-1}_2(I,0)\big)^2,\]
with $\bar \lambda_2=\lambda_2/2+\mu$ so that
\[ |\varphi^{-1}_2(I,0)| \geq \sqrt{K(I)/\bar\lambda_2} \geq \sqrt{ c_{1} I/\bar\lambda_2}\,.\]
Finally, we obtain
\[
\sup_{\theta \in \mathbb{T}} |\partial_{I} \varphi^{-1}_1(I, \theta)|+ |\partial_{I} \varphi^{-1}_2(I,\theta)|   \leq  \frac{C_r}{\sqrt{I}}\,,
\]
which gives the desired result and concludes the proof of the lemma.

\bibliographystyle{halpha-abbrv}
\bibliography{refs,preprints}
\end{document}